\documentclass[reqno]{amsart}
\RequirePackage[l2tabu, orthodox]{nag}
\usepackage{amssymb} 
\usepackage{amscd}
\usepackage{bbm}
\usepackage{color}
\usepackage{esint}
\usepackage{etoolbox}
\usepackage{framed}
\usepackage{graphicx}
\usepackage{hyphenat}
\usepackage{mathrsfs} 
\usepackage{lmodern}
\usepackage{microtype}
\usepackage{stmaryrd}
\SetSymbolFont{stmry}{bold}{U}{stmry}{m}{n}
\usepackage{hyperref} 
\usepackage{amsrefs}

\makeatletter
\DeclareFontFamily{OMX}{MnSymbolE}{}
\DeclareSymbolFont{MnLargeSymbols}{OMX}{MnSymbolE}{m}{n}
\SetSymbolFont{MnLargeSymbols}{bold}{OMX}{MnSymbolE}{b}{n}
\DeclareFontShape{OMX}{MnSymbolE}{m}{n}{
    <-6>  MnSymbolE5
   <6-7>  MnSymbolE6
   <7-8>  MnSymbolE7
   <8-9>  MnSymbolE8
   <9-10> MnSymbolE9
  <10-12> MnSymbolE10
  <12->   MnSymbolE12
}{}
\DeclareFontShape{OMX}{MnSymbolE}{b}{n}{
    <-6>  MnSymbolE-Bold5
   <6-7>  MnSymbolE-Bold6
   <7-8>  MnSymbolE-Bold7
   <8-9>  MnSymbolE-Bold8
   <9-10> MnSymbolE-Bold9
  <10-12> MnSymbolE-Bold10
  <12->   MnSymbolE-Bold12
}{}

\let\llangle\@undefined
\let\rrangle\@undefined
\DeclareMathDelimiter{\llangle}{\mathopen}%
                     {MnLargeSymbols}{'164}{MnLargeSymbols}{'164}
\DeclareMathDelimiter{\rrangle}{\mathclose}%
                     {MnLargeSymbols}{'171}{MnLargeSymbols}{'171}
\makeatother

\theoremstyle{plain} 
\newtheorem{theorem}{Theorem}[section]
\newtheorem{proposition}[theorem]{Proposition}
\newtheorem{lemma}[theorem]{Lemma} 

\theoremstyle{definition} 
\newtheorem{definition}[theorem]{Definition}
\newtheorem{assumption}[theorem]{Assumption}

\theoremstyle{remark} 
\newtheorem{remark}[theorem]{Remark}

\definecolor{shadecolor}{rgb}{1,0.8,0.3}
\newcommand{\DETAIL}[1]{}
\newcommand{\IGNORE}[1]{}

\newcommand{\A}{{\mathrm{anti}}}
\newcommand{\B}{{\mathcal{B}}}
\newcommand{\C}{{\mathscr{C}}}
\newcommand{\D}{{\mathscr{D}}}
\newcommand{\E}{{\mathcal{E}}}
\newcommand{\F}{{\mathbf{F}}}
\renewcommand{\H}{{\mathscr{H}}}

\newcommand{\K}{{\mathcal{K}}}
\renewcommand{\L}{{\mathscr{L}}}
\newcommand{\M}{{\mathscr{M}}}
\newcommand{\N}{{\mathbbm{N}}}
\newcommand{\R}{{\mathbbm{R}}}
\renewcommand{\S}{{\mathrm{sym}}}
\newcommand{\T}{{\mathrm{T}}}

\newcommand{\W}{{\mathscr{W}}}

\newcommand{\AF}{{\mathfrak{A}}}
\newcommand{\AR}{{\ADM\SUB{\RHO}}}

\newcommand{\BL}{{\mathrm{BL}}}
\newcommand{\BM}{{\boldsymbol{m}}}

\newcommand{\BS}{{\boldsymbol{s}}}
\newcommand{\BT}{{\boldsymbol{t}}}
\newcommand{\BU}{{\boldsymbol{u}}}
\newcommand{\BV}{{\boldsymbol{v}}}
\newcommand{\BW}{{\boldsymbol{w}}}
\newcommand{\BX}{{\boldsymbol{x}}}
\newcommand{\BY}{{\boldsymbol{y}}}
\newcommand{\BZ}{{\boldsymbol{z}}}
\newcommand{\CB}{{\C_\mathrm{b}}}
\newcommand{\CC}{{\C_\mathrm{c}}}
\newcommand{\CR}{{C_\RHO}}
\newcommand{\DD}{{\mathrm{d}}}
\newcommand{\FS}{{\mathscr{F}}}
\newcommand{\FX}{{\mathfrak{x}}}
\newcommand{\GS}{\geqslant}
\newcommand{\HA}{{\TST\frac{1}{2}}}
\newcommand{\HE}{{\TST\frac{1}{8}}}
\newcommand{\ID}{{\mathrm{id}}}
\newcommand{\LS}{\leqslant}
\newcommand{\MK}{{\M_\mathrm{K}}}
\newcommand{\MU}{{\boldsymbol{\mu}}}

\renewcommand{\OE}{{\TST\frac{1}{8}}}
\newcommand{\OF}{{\TST\frac{1}{4}}}
\newcommand{\OH}{{\TST\frac{1}{2}}}
\newcommand{\OR}{{\OPT\SUB{\RHO}}}

\newcommand{\PK}{{\mathbf{P}}}
\newcommand{\PP}{{\mathbbm{p}}}
\newcommand{\PR}{{\SP\SUB{\RHO}(\R^{2d})}}
\newcommand{\RP}{{\mathscr{R}_+}}
\newcommand{\SP}{{\mathscr{P}}}

\newcommand{\VV}{{\mathbbm{v}}}

\newcommand{\WR}{{\WAS\SUB{\RHO}}}
\newcommand{\XX}{{\mathbbm{x}}}

\newcommand{\YY}{{\mathbbm{y}}}
\newcommand{\ZZ}{{\mathbbm{z}}}

\newcommand{\ACC}{{\mathrm{A}}}
\newcommand{\ADM}{{\mathrm{Adm}}}
\newcommand{\ALG}{{\mathscr{A}}}
\newcommand{\BAR}{{\mathrm{b}}}
\newcommand{\BBE}{{\mathbbm{E}}}
\newcommand{\BBX}{{\boldsymbol{X}}}

\newcommand{\BRX}{{\bar{x}}}
\newcommand{\BVS}{{\mathrm{BV}}}

\newcommand{\CMU}{{C_\MU}}

\newcommand{\DOM}{{\mathrm{dom}}}
\newcommand{\DST}{\displaystyle}
\newcommand{\EPS}{\varepsilon}
\newcommand{\ENT}{{\mathrm{Ent}}}
\newcommand{\ETA}{{\boldsymbol{\eta}}}
\newcommand{\INT}{{\mathcal{U}}}
\newcommand{\LEB}{{\mathcal{L}}}
\newcommand{\LIP}{{\mathrm{Lip}}}
\newcommand{\LOC}{{\mathrm{loc}}}
 
\newcommand{\MAT}[2][]{%
	\ifstrempty{#1}{%
		\mathrm{Mat}_{#2}(\R)}{%
		\mathrm{Mat}_{#2}(\R, #1)}}
\newcommand{\MON}{{\mathrm{Mon}}}
\newcommand{\ONE}{{\mathbbm{1}}}
\newcommand{\OPT}{{\mathrm{Opt}}}
\newcommand{\ORS}{{\mathbf{s}}}
\newcommand{\ORT}{{\mathbf{t}}}
\newcommand{\ORU}{{\mathbf{u}}}
\newcommand{\ORV}{{\mathbf{v}}}
\newcommand{\ORW}{{\mathbf{w}}}
\newcommand{\ORZ}{{\mathbf{z}}}
\newcommand{\PHI}{{\boldsymbol{\phi}}}
\newcommand{\RES}{{\mathbf{R}}}
\newcommand{\RHO}{\varrho}
\newcommand{\SUB}[1]{_{\raisebox{0.25ex}{\scriptsize{$#1$}}}}
\newcommand{\SYM}[2][]{%
	\ifstrempty{#1}{%
		\mathrm{Sym}_{#2}(\R)}{%
		\mathrm{Sym}_{#2}(\R, #1)}}

\newcommand{\TST}{\textstyle}
\newcommand{\UPS}{{\boldsymbol{\upsilon}}}
\newcommand{\WAS}{{\mathrm{W}}}

\newcommand{\BETA}{{\boldsymbol{\beta}}}
\newcommand{\BRHO}{{\bar{\RHO}}}
\newcommand{\BFX}{{\mathfrak{X}}}

\newcommand{\DIAM}{{\mathrm{diam}}}
\newcommand{\DIST}{{\mathrm{dist}}}
\newcommand{\DISS}{{\mathbf{D}}}
\newcommand{\HAUS}{{\mathcal{H}}}
 
\newcommand{\SKEW}[1]{{\mathrm{Skew}_{#1}(\R)}}
 
\newcommand{\WEAK}{\DOTSB\protect\relbar\protect\joinrel\rightharpoonup}
\newcommand{\ZETA}{{\boldsymbol{\zeta}}}

\newcommand{\ALPHA}{{\boldsymbol{\alpha}}}
\newcommand{\GAMMA}{{\boldsymbol{\gamma}}}
\newcommand{\OMEGA}{{\boldsymbol{\omega}}}

\newcommand{\GRAPH}{{\mathrm{graph}}}
\newcommand{\TRACE}{{\mathrm{tr}}}

\DeclareMathOperator{\DIM}{{\mathrm{dim}}}
\DeclareMathOperator{\OSC}{{\mathrm{osc}}}
\DeclareMathOperator{\SPT}{{\mathrm{spt}}}
\DeclareMathOperator{\REL}{\square}
\DeclareMathOperator{\TAN}{{\mathrm{Tan}}}

\DeclareMathOperator{\CONV}{{\mathrm{conv}}}
\DeclareMathOperator{\HULL}{{\mathrm{conv}}}
\DeclareMathOperator*{\ESUP}{ess\,sup}
\DeclareMathOperator{\INTR}{{\mathrm{int}}}

\DeclareMathOperator{\CCONV}{{\overline{\mathrm{conv}}}}

\numberwithin{equation}{section}

\title[Variational Time Discretization]{A Variational Time
Discretization for Compressible Euler Equations}

\author
  {Fabio Cavalletti}
\address
  {Fabio Cavalletti,
   Scuola Internazionale Superiore di Studi Avanzati (SISSA),
   Via Bonomea 265,
   34136 Trieste,
   Italy}
\email
  {cavallet@sissa.it}

\author
  {Marc Sedjro}
\address
  {Marc Sedjro,
   AIMS Tanzania,
   Plot No: 288, Makwahiya Street, 
   Regent Estate,
   Dar es Salaam,
   Tanzania}
\email
  {sedjro@aims.ac.tz}

\author
  {Michael Westdickenberg}
\address
  {Michael Westdickenberg,
   Lehrstuhl f\"{u}r Mathematik (Analysis),
   RWTH Aachen University,
   Templergraben 55,
   52062 Aachen, 
   Germany}
\email
  {mwest@instmath.rwth-aachen.de}

\date{\today}
\subjclass[2000]{35L65, 49J40, 82C40}
\keywords{Compressible Gas Dynamics, Optimal Transport}

\begin{document}

\begin{abstract}
We introduce a variational time discretization for the multi-dimen\-sional
gas dynamics equations, in the spirit of minimizing movements for
curves of maximal slope. Each timestep requires the minimization of a
functional measuring the acceleration of fluid elements, over the cone
of monotone transport maps. We prove convergence to measure-valued
solutions for the pressureless gas dynamics and the compressible Euler
equations. For one space dimension, we obtain sticky particle
solutions for the pressureless case.
\end{abstract}

\maketitle


\section{Introduction}

The compressible Euler equations model the dynamics of compressible
fluids like gases. They form a system of hyperbolic conservation laws
\begin{equation}
  \left.\begin{array}{r}
    \displaystyle
    \partial_t\RHO
      +\nabla\cdot(\RHO\BV) = 0
\\[1ex]
    \displaystyle
    \partial_t(\RHO\BV)
      + \nabla\cdot(\RHO\BV\otimes\BV)
        + \nabla p = 0
\\[1ex]
    \displaystyle
    \partial_t \EPS
      + \nabla\cdot\big( (\EPS+p)\BV \big) = 0
  \end{array}\right\}
  \quad\text{in $[0,\infty)\times\R^d$.}
\label{E:FULL}
\end{equation}
The unknowns $(\RHO,\BV, \EPS)$ depend on time $t\in[0,\infty)$ and
space $x\in\R^d$ and we assume that suitable initial data (to be
specified later) is given:
\[
  (\RHO,\BV,\EPS)(t=0,\cdot) =: (\BRHO, \bar{\BV}, \bar{\EPS}).
\]
We will think of $\RHO$ as a map from $[0,\infty)$ into the space of
non-negative, finite Borel measures, which we denote by $\M_+(\R^d)$.
The quantity $\RHO$ is called the density and it represents the
distribution of mass in time and space. The first equation in
\eqref{E:FULL} (the continuity equation) expresses the local
conservation of mass, where
\begin{equation}
  \BV(t,\cdot) \in \L^2\big( \R^d,\RHO(t,\cdot) \big)
  \quad\text{for all $t\in[0,\infty)$}
\label{E:VELOC}
\end{equation}
is the Eulerian velocity field taking values in $\R^d$. The second
equation in \eqref{E:FULL} (the momentum equation) expresses the local
conservation of momentum $\BM := \RHO\BV$. The pressure $p$ will be
discussed below. Notice that $\BM(t,\cdot)$ is a finite $\R^d$-valued
Borel measure absolutely continuous with respect to $\RHO(t,\cdot)$
for all $t\in[0,\infty)$, because of \eqref{E:VELOC}. The quantity
$\EPS$ is the total energy of the fluid and $\EPS(t,\cdot)$ is again a
measure in $\M_+(\R^d)$ for all times $t\in[0,\infty)$. It is
reasonable to assume $\EPS(t,\cdot)$ to be absolutely continuous with
respect to the density $\RHO(t,\cdot)$ (no energy in vacuum). The
third (the energy) equation in \eqref{E:FULL} expresses the local
conservation of energy.

Formally, the equations \eqref{E:FULL} imply that the total mass and
energy are preserved over time. Therefore, if the fluid has finite
mass and total energy initially, then this will be the case for all
positive times. We will make this assumption in the following. Without
loss of generality, we will also assume that the mass is equal to one,
which implies that $\RHO(t,\cdot)\in\SP(\R^d)$, the space of Borel
probability measures.

To obtain a closed system \eqref{E:FULL} it is necessary to prescribe
an equation of state, which relates the pressure $p$ to the density
$\RHO$ and the total energy $\EPS$. It is provided by thermodynamics.
The following three distinct situations are important:

\subsection{Pressureless gases}

The pressure $p$ vanishes and so the total energy reduces to just
the kinetic energy: $\EPS = \frac{1}{2} \RHO|\BV|^2$. The equations
\eqref{E:FULL} take the form
\begin{equation}
  \left.\begin{array}{r}
    \displaystyle
    \partial_t\RHO
      +\nabla\cdot(\RHO\BV) = 0
\\[1ex]
    \displaystyle
    \partial_t(\RHO\BV)
      + \nabla\cdot(\RHO\BV\otimes\BV) = 0
  \end{array}\right\}
  \quad\text{in $[0,\infty)\times\R^d$,}
\label{E:PGD}
\end{equation}
and the energy equation in \eqref{E:FULL} follows formally from the
continuity and momentum equations. The system \eqref{E:PGD} has been
proposed as a simple model describing the formation of galaxies in the
early stage of the universe. Its one-dimensional version is a building
block for semiconductor models. Since fluid elements do not interact
with each other because there is no pressure, the density
$\RHO(t,\cdot)$ may become singular with respect to the
$d$-dimensional Lebesgue measure $\LEB^d$. For adhesion (or: sticky
particle) dynamics this concentration effect is actually a desired
feature; see \cite{Zeldovich1970}: If fluid elements meet at the same
location, then they stick together to form larger compounds and so
$\RHO(t,\cdot)$ can have singular parts (in particular, Dirac
measures). Consequently \eqref{E:FULL} must be understood in the sense
of distributions. While mass and momentum are conserved, kinetic
energy may be destroyed since the collisions are inelastic. In
particular, the energy equation in \eqref{E:FULL} will typically be an
inequality only. We will call the assumption of adhesion dynamics an
entropy condition.

There are now numerous articles studying the pressureless gas dynamics
equations \eqref{E:PGD} in one space dimension and establishing global
existence of solutions. Frequently, a sequence of approximate
solutions is constructed by considering discrete particles, where the
initial mass distribution is approximated by a finite sum of Dirac
measures. The dynamics of these particles are described by a finite
dimensional system of ordinary differential equations between
collision times. Whenever multiple particles collide, the new velocity
of the bigger particle is determined from the conservation of mass and
momentum, and the choice of impact law. The general existence result
is obtained by letting the number of discrete particles go to
infinity. In order to pass to the limit, several approaches are
feasible. We only mention two: One approach relies on the observation
that the cumulative distribution function associated to the density
$\RHO$ satisfies a certain scalar conservation law (see
\cite{BrenierGrenier1998}) so the theory of entropy solutions of
scalar conservation laws can be applied. Another approach makes use of
the well-known theory of first-order differential inclusions, applied
to the cone of monotone transport maps from a reference measure space
to $\R$; see \cite{NatileSavare2009}. We refer the reader to
\cites{BouchutJames1995, BouchutJames1999,
BrenierGangboSavareWestdickenberg2013, ERykovSinai1996,
GangboNguyenTudorascu2009, Grenier1995, HuangWang2001, Moutsinga2008,
NguyenTudorascu2008, PoupaudRascle1997, Wolansky2007} for more
information.

For the multi-dimensional pressureless gas dynamics equations, global
existence of solutions to \eqref{E:PGD} has been considered in
\cite{Dermoune2005}. The global existence proof in \cite{Sever2001}
for sticky particle solutions seems to be incomplete, as the authors
in \cite{BressanNguyen2014} show that for certain choices of initial
data, sticky particle solutions cannot exist. This raises the question
of the correct solution concept for the equations \eqref{E:PGD}.

\subsection{Isentropic gases} 

In this regime, the thermodynamical entropy of the fluid is assumed to
be constant in space and time. Consequently, the pressure is a
function of the density only. We introduce the internal energy
\[
  \INT[\RHO] := \begin{cases}
      \DST \int_{\R^d} U\big( r(x) \big) \,dx
        & \text{if $\RHO = r \LEB^d$,}
\\
      \infty 
        & \text{otherwise,}
    \end{cases}
\]
where $U(r) := \kappa r^\gamma$ for $r\GS 0$. The constant $\gamma>1$
is called the adiabatic coefficient, and $\kappa>0$ is another
constant. The total energy is the sum of the kinetic energy introduced
above and the internal energy. Since we are only interested in
solutions of \eqref{E:FULL} with finite total energy, the density
$\RHO(t,\cdot)$ must be absolutely continuous with respect to the
Lebesgue measure for all $t\in[0,\infty)$. Let $r(t,\cdot)$ be its
Radon-Nikod\'{y}m derivative. Then $p(t,\cdot) = P\big( r(t,\cdot)
\big) \,\LEB^d$ for all $t\in[0,\infty)$, where
\[
  P(r) = U'(r)r-U(r)
  \quad\text{for $r\GS 0$.}
\]
This setup describes polytropic gases. Other choices of $U$ are
possible, for example $U(r) = \kappa r\log r$ for isothermal gases
(then $P(r) = \kappa r$). We consider
\begin{equation}
  \left.\begin{array}{r}
    \displaystyle
    \partial_t\RHO
      +\nabla\cdot(\RHO\BV) = 0
\\[1ex]
    \displaystyle
    \partial_t(\RHO\BV)
      + \nabla\cdot(\RHO\BV\otimes\BV) 
      + \nabla P(\RHO) = 0
  \end{array}\right\}
  \quad\text{in $[0,\infty)\times\R^d$}
\label{E:IEE}
\end{equation}
(with slight abuse of notation). As in the pressureless case, the
energy equation in \eqref{E:FULL} follows formally from the continuity
and the momentum equation.

It is well-known that a generic solution to the isentropic Euler
equations will not remain smooth, even for regular initial data.
Instead the solution will have jump discontinuities along
codimension-one submanifolds in space-time, which are called shocks.
Then the continuity and the momentum equation must be considered in
the sense of distributions, and the energy equation does no longer
follow automatically. A physically reasonable relaxation is to assume
that no energy can be created by the fluid: The energy equality in
\eqref{E:FULL} must be replaced by the inequality
\begin{equation}
  \partial_t \Big( \HA\RHO|\BV|^2 + U(\RHO) \Big) + \nabla \cdot 
    \bigg( \Big( \HA\RHO|\BV|^2 + U'(\RHO)\RHO \Big) \BV \bigg) 
      \LS 0
\label{E:ED}
\end{equation}
in distributional sense. Physically, strict inequality in \eqref{E:ED}
means that mechanical energy is transformed into heat, a form of
energy that is not accounted for by the model. Notice that a
differential inequality like \eqref{E:ED} contains some information on
the regularity of solutions: The space-time divergence of a certain
non-linear function of $(\RHO,\BV)$ is a non-positive distribution, and
thus a measure. In the one-dimensional case, it is even reasonable to
look for weak solutions of \eqref{E:IEE} that satisfy differential
inequalities like \eqref{E:ED} simultaneously for a large class of
non-linear functions of $(\RHO,\BV)$ that are called entopy-entropy
flux pairs. Such an assumption on the solutions is again an entropy
condition. Using the method of compensated compactness, it is then
possible the establish the global existence of weak (entropy)
solutions of \eqref{E:IEE}. We refer the reader to \cites{Chen2000,
DingChenLuo1987, DingChenLuo1988, ChenLeFloch2000,
ChenPerepelitsa2010, DiPerna1983, LeFlochWestdickenberg2007,
LionsPerthameSouganidis1996, LionsPerthameTadmor1994} for more
information.

In several space dimensions the only available entropy-entropy flux
pair is the total energy-energy flux. Using non-linear iteration
schemes like the ones introduced by Nash \cites{Nash1954, Nash1956} to
construct isometric imbeddings of Riemannian manifolds, one can
establish the existence of a large class of initial data for which
weak solutions of \eqref{E:IEE} exist globally in time. We refer the
reader to the ground-breaking results by De Lellis and Sz\'{e}kelyhidi
\cites{DeLellisSzekelyhidi2009, DeLellisSzekelyhidi2010} and
subsequent work \cites{Chiodaroli2014, ChiodaroliKreml2014,
ChiodaroliDeLellisKreml2014, Feireisl2014} by various authors. These
results give, in fact, much more precise information: One can show
that for suitable initial data there exist \emph{infinitely many} weak
solutions of \eqref{E:IEE}, even if one requires that solutions
satisfy an entropy condition in the form \eqref{E:ED}. This is related
to the fact that there is---in addition to energy dissipation through
shocks---an additional dissipation mechanism due to very high
oscillations of the velocity field, which is reminiscent of anomalous
dissipation in turbulence. Moreover, there is a precise threshold of
H\"{o}lder regularity $1/3$ between the energy conserving and the
energy dissipating regimes. For incompressible flows, this has been
conjectured based on physical considerations by Onsager
\cite{Onsager1949}. A mathematical proof of this conjecture has been
provided in a series of recent articles; see \cites{Buckmaster2015,
BuckmasterDeLellisSzekelyhidi2016, Isett2017} and references therein.
For related results for the compressible Euler equations see
\cite{FeireislGwiazdaSwierczewskaGwiazdaWiedemann2017}. The Cauchy
problem for \eqref{E:IEE} in several space dimensions, however, has
not been solved yet: In order to apply the above methods for
\emph{any} given initial data, it currently seems necessary to allow a
small increase in energy initially, which is in violation of
\eqref{E:ED}.

\subsection{Full Euler equations}

We consider a polytropic gas with adiabatic coefficient $\gamma>1$.
Then the pressure is given in terms of $(\RHO, \BV, \EPS)$ by the
formula
\begin{equation}
  p(t,\cdot) = (\gamma-1) \bigg( \EPS-\frac{1}{2}\RHO|\BV|^2 
    \bigg)(t,\cdot)
  \quad\text{for all $t\in[0,\infty)$.}
\label{E:PIE}
\end{equation}
Density and pressure define the specific thermodynamical entropy,
given as
\[
	S := \log\Big( \frac{p}{c\RHO^\gamma} \Big)
	\quad\text{with $c := \kappa(\gamma-1)>0$ and $\gamma>1$}
\]
in the case of polytropic gases. We assume that
\[
  S(t,\cdot) \in \L^1\big( \R^d, \RHO(t,\cdot) \big)
  \quad\text{for all $t\in[0,\infty)$,}
\]
so that the entropy density $\sigma = \RHO S$ is well-defined as a
measure.

\begin{definition}[Internal Energy]\label{D:INT} 
Let $U(r,S) := \kappa e^S r^\gamma$ for all $r\GS 0$ and $S\in\R$,
where $\kappa>0$ and $\gamma>1$ are constants. Then we define the
internal energy
\[
  \INT[\RHO,\sigma] := \begin{cases}
      \DST \int_{\R^d} U\big( r(x),S(x) \big) \,dx
        & \text{if $\RHO=r\LEB^d$ and $\sigma=\RHO S$,}
\\
      \infty
        & \text{otherwise,}
    \end{cases}
\]
for all pairs of measures $(\RHO,\sigma) \in \SP(\R^d)\times
\M_+(\R^d)$.
\end{definition}

Since we are only interested in solutions with finite energy, the
density must be absolutely continuous with respect to the Lebesgue
measure, thus $\RHO(t,\cdot) = r(t,\cdot) \,\LEB^d$. In this case, we
define $P(r,S) = U'(r,S)r-U(r,S)$ (here $'$ denotes differentiation
with respect to $r$), and the pressure term in \eqref{E:FULL2} takes
the form
\begin{equation}
  p(t,\cdot) = P\big( r(t,\cdot),S(t,\cdot) \big) \,\LEB^d
  \quad\text{for all $t\in[0,\infty)$.}
\label{E:PISIGMA}
\end{equation}
Moreover, combining \eqref{E:PIE} and \eqref{E:PISIGMA} with
\eqref{E:FULL}, we obtain that 
\begin{equation}
  \partial_t \sigma + \nabla\cdot(\sigma\BV) = 0
  \quad\text{in $[0,\infty)\times\R^d$.}
\label{E:ENTCON}
\end{equation}
Equivalently, the specific entropy $S$ must be constant along
characteristics:
\begin{equation}
  \partial_t S + \BV\cdot\nabla S = 0
  \quad\text{in $[0,\infty)\times\R^d$.}
\label{E:ENTCHAR}
\end{equation}
Formally, system \eqref{E:FULL} is equivalent to the one where the
energy equation is replaced by \eqref{E:ENTCON} (or even
\eqref{E:ENTCHAR}). But since the solutions to the compressible Euler
equations may become discontinuous in finite time, the physically
reasonable relaxation is that the specific entropy should be
non-decreasing forward in time, which expresses the second law of
thermodynamics. It follows that
\begin{equation}
  \inf_{x\in\R^d} S(t,x) \GS \inf_{x\in\R^d} \bar{S}(x)
  \quad\text{for all $t\in[0,\infty)$,}
\label{E:SINF}
\end{equation}
where $\bar{S}$ is the initial specific entropy. An Eulerian argument
in support of \eqref{E:SINF}, based on entropy inequalities, was given
in \cite{Tadmor1986}. We will assume that
\[
  \inf_{x\in\R^d} \bar{S}(x) \GS \alpha
\]
for some $\alpha\in\R$. Since the shift of $S$ by $\alpha$ can be
absorbed into the constant $\kappa>0$, we may assume without loss of
generality that $\alpha=0$, thus $S$ is non-negative. As for the
isentropic Euler equations, global existence and uniqueness of weak
solutions to the full system \eqref{E:FULL} are open problems, 
even in one space dimension.

\bigskip

In this paper, we will consider a variational time discretization for
the compressible gas dynamics equations that is motivated by
minimizing movements for curves of maximal slope on metric spaces; see
\cites{AmbrosioGigliSavare2008, DeGiorgi1993,
JordanKinderlehrerOtto1998}. For any given initial data with finite
mass and total energy, we prove that sequences of approximate
solutions generated by this scheme converge to a measure-valued
solution of \eqref{E:FULL}.

\begin{definition}
We denote by $\SP_2(\R^d)$ the space of Borel probability measures
with finite second moment, endowed with the $2$-Wasserstein distance;
see Definition~\ref{D:WAS} below. For a map $t \mapsto \RHO_t \in
\SP_2(\R^d)$, $t\in[0,T]$, we denote by
\[
	\|\RHO\|_{\LIP([0,T]; \SP_2(\R^d))} 
  		:= \sup_{\substack{t_1, t_2 \in [0,T] \\ t_1\neq t_2}} 
    		\frac{\WAS_2(\RHO_{t_1}, \RHO_{t_2})}{|t_2-t_1|}
\]
the Lipschitz seminorm, with $\WAS_2$ the Wasserstein distance; see
\eqref{E:WASDE}.
\end{definition}

\begin{definition} 
We denote by $\M_\ENT(\R^d)$ the space of non-negative Borel measures
with finite second moments and total variation equal to $\ENT \in
[0,\infty)$, endowed with as suitably rescaled Wasserstein distance;
see Definition~\ref{D:WAS}.
\end{definition}

We will assume that the total momentum vanishes initially, which
implies that the total momentum vanishes for all $t>0$. This is not a
restriction as the hyperbolic conservation law \eqref{E:FULL} is
invariant under transformations to a moving reference frame in the
absence of boundaries. The momentum map $t \mapsto \BM_t =
\RHO_t\BV_t$ takes values in a convex set of $\R^d$-valued Borel
measures whose total variations are uniformly bounded, as a
consequence of a bound on the total energy. On this set, the narrow
convergence of measures is metrized by the Monge-Kantorovich norm:

\begin{definition}\label{D:KANTOO}
We denote by $\LIP(\R^d; \R^N)$ the vector space of Lipschitz
continuous maps $\zeta \colon \R^d \longrightarrow \R^N$. The
Lipschitz constant of $\zeta \in \LIP(\R^d; \R^N)$ is
\[
  \|\zeta\|_{\LIP(\R^d)} 
    := \sup_{x_1 \neq x_2} \frac{|\zeta(x_1)-\zeta(x_2)|}{|x_1-x_2|}.
\]
We denote by $\BL(\R^d;\R^N)$ the subspace of bounded functions in
$\LIP(\R^d; \R^N)$. It is a Banach space when equipped with the
bounded Lipschitz norm
\[
	\|\zeta\|_{\BL(\R^d)} := \max\Big\{ \|\zeta\|_{\L^\infty(\R^d)},
		\|\zeta\|_{\LIP(\R^d)} \Big\}.
\]
Let $\BL_1(\R^d; \R^N)$ be the space of all $\zeta \in \BL(\R^d;
\R^N)$ with $\|\zeta\|_{\BL(\R^d)}\LS 1$.

We denote by $\MK(\R^d; \R^N)$ the space of $\R^N$-valued Borel
measures $\BM$ with zero mean and finite first moment, equipped with
the Monge-Kantorovich norm
\begin{equation}\label{E:KANT} 
  \|\BM\|_{\MK(\R^d)} := \sup\bigg\{ \int_{\R^d} \zeta(x) \cdot \BM(dx) 
    \colon \zeta \in \BL_1(\R^d; \R^N) \bigg\}.
\end{equation}
The Monge-Kantorovich norm is bounded above by the total variation.
\end{definition}

We refer the reader to \cites{Dudley1966,
CastaingRaynaudDeFitteValadier2004, ChitescuMikulescuNitaIoana2016,
HilleSzarekWormZiemlanska2017} for additional information. Notice that
the integral in \eqref{E:KANT} is well-defined because $\BM$ has
finite first moment, by assumption. By Cauchy-Schwarz inequality, this
holds true whenever $\BM = \RHO\BV$ with
\[
  \RHO \in \SP_2(\R^d)
  \quad\text{and}\quad
  \BV \in \L^2(\R^d,\RHO).
\]
%

\begin{remark}
We think of weak solutions of \eqref{E:FULL} as maps $t \mapsto
(\RHO_t, \BM_t, \sigma_t)$ taking values in a convex set of vector
measures with uniformly bounded total variations, equipped with the
Wasserstein distance/Monge-Kantorovich norm. Since the maps are
Lipschitz continuous, they are strongly differentiable almost
everywhere (a.e.) in time. In particular, the time derivative of the
momentum exists as an element in the closure of the space of vector
measures with respect to the Monge-Kantorovich norm, which is a proper
subset of the dual space $\BL(\R^d;\R^d)^*$; see
\cite{BouchitteButtazzoDePascale2010} for additional properties. The
usage of the Monge-Kantorovich norm is thus very well adapted to the
structure of the equations \eqref{E:FULL}, with the time derivative of
the momentum given as the divergence of a measure field taking values
in the symmetric, positive semidefinite matrices. Testing against
$\BL_1(\R^d; \R^d)$-functions, we can (mollify and) integrate by
parts. Since the derivative of the test function is bounded in norm,
we must control the size of the matrix field, for example with respect
to the $1$-Schatten norm (the sum over the absolute values of the
singular values). For symmetric, positive semidefinite matrices, this
is simply the trace of the matrix.
\end{remark}

\begin{definition}
Let $E\subset \R$ be some subset and $(X,d)$ a metric space. We denote
by $\BVS(E,X)$ the space of maps $f \colon I \longrightarrow E$ with
finite variation
\[
	V_E(f) := \sup\sum_{i=1}^m d\big( f(t_{i-1}),f(t_i) \big),
\]
where the $\sup$ is taken over all $t_0 \LS t_1 \LS \ldots \LS t_m$
contained in $E$.
\end{definition}

We refer the reader to \cite{FleischerPorter2001} for further
information on metric space-valued functions of bounded variation. In
particular, a version of Helly's compactness theorem
for sequences of $X$-valued maps is proved there. A function $f\colon E
\longrightarrow X$ has bounded variation if and only if it factors as
$g \circ \phi$ where $\phi \colon E \longrightarrow \R$, defined as
\[
	\phi(t) := V_{(-\infty,t]\cap E}(f)
	\quad\text{for all $t\in E$},
\]
is its total variation and $g \colon \phi(E) \longrightarrow X$ is
Lipschitz continuous. We will consider spaces $\BVS(E,X)$ with $E =
[0,T]$ and $X = \MK(\R^d; \R^d)$.

\begin{remark}\label{R:MOMOS}
Since we consider densities with finite second moment (thus momenta
with finite first moment), it is possible to consider test functions
with non-compact support. With $V$ any Banach space, let $\C_*(\R^d;
V)$ be the space of all continuous functions $\zeta \colon \R^d
\longrightarrow V$ with the property that $\lim_{|x|\rightarrow\infty}
\zeta(x) \in V$ exists. Then
\[
	\C_*(\R^d; V) = V + \C_0(\R^d; V),
\]
where $\C_0(\R^d; V)$ is the closure of the space of compactly
supported continuous $V$-valued maps with respect to the the
$\sup$-norm. We define
\begin{equation}
	\AF := \big\{ u\in\C^1(\R^d;\R^D) \colon
    \text{$\nabla u \in \C_*\big( \R^d;\MAT{D\times d} \big)$} \big\},
\label{E:CSTO}
\end{equation}
with $\MAT{D\times d}$ the space of real $(D\times d)$-matrices. We
will not explicitly indicate the di\-men\-sion $D$ as it will be clear
from the context. Functions in $\AF$ grow at most linearly at
infinity. Let $\C^1_c([0,\infty))\otimes\AF$ be the space of tensor
products
\[
	\eta\otimes\zeta(t,x) := \eta(t)\zeta(x)
	\quad\text{with $\eta \in \C^1_c\big( [0,\infty) \big)$ and 
		$\zeta \in \AF$.}
\]
We will assume \eqref{E:FULL2} to hold in duality with this space,
testing against all
\[
	\eta\otimes\zeta \in \C^1_c\big( [0,\infty) \big) \otimes \AF.
\]
Notice that all products of conserved quantities $(\RHO, \BM, \sigma)$
with $\eta\otimes\zeta$ are integrable since these measures have
finite first moments. On the other hand, the derivative $\nabla_x
\zeta$ is bounded and so the integrals involving the fluxes are
well-defined as well. For all $T>0$, the tensor product
$\C([0,T])\otimes V$ is dense in $\C([0,T]; V)$ with respect to the
$\sup$-norm, with $V$ any locally convex topoligical vector space.
\end{remark}

We can now state our main existence result.

\begin{theorem}[Global Existence]\label{T:GLOBAL}
Suppose that initial data
\[
  \BRHO \in \SP_2(\R^d),
  \quad
  \bar{\BV} \in \L^2(\R^d,\BRHO),
  \quad
  \bar{S} \in \L_+^\infty(\R^d, \BRHO)
\]
is given with vanishing total momentum and finite internal energy:
\[
  \int_{\R^d} \bar{\BV}(x) \,\BRHO(dx) = 0,
  \quad
  \INT[\BRHO, \bar{\sigma}] < \infty,
\]
where $\bar{\sigma} := \BRHO \bar{S}$. Let $\ENT := \int_{\R^d}
\bar{\sigma}(dx)$ be the initial total entropy.

For any $T>0$ there exist curves
\[
\begin{gathered}
  \RHO \in \LIP\big( [0,T]; \SP_2(\R^d) \big),
  \quad
  \sigma \in \LIP\big( [0,T]; \M_\ENT(\R^d) \big),
\\
  \BM \in \LIP\big( [0,T]; \MK(\R^d; \R^d) \big)
\end{gathered}
\]
with the following properties:
\begin{enumerate}
\item The initial data is attained:
\[
   \RHO(0,\cdot) = \bar{\RHO},
  \quad
  \BM(0,\cdot) = \bar{\RHO} \bar{\BV},
  \quad
  \sigma(0,\cdot) = \bar{\sigma}.
\]
\item We have $\BM =: \RHO\BV$ and $\sigma =: \RHO S$ with
\[
	\BV(t,\cdot) \in \L^2\big( \R^d, \RHO(t,\cdot) \big),
	\quad
	S(t,\cdot) \in \L_+^\infty\big( \R^d, \RHO(t,\cdot) \big)
\]
for all $t \in [0,T]$.
\item There exist two Young measures 
\[
	\nu^1, \nu ^2 \in \L^\infty_w\big( [0,T];
		\M_+(\dot\R^d\times\BFX) \big),
\]
where $\BFX$ is a suitable compactification of the set 
\[
	X := [0,\infty) \times \R^d \times [0,S_{\max}], 
	\quad 
	S_{\max} := \|S\|_{\L^\infty(\R^d,\RHO)},
\]
of admissible $(\RHO,\BV, S)$, such that
\end{enumerate}
\begin{equation}
  \left.\begin{array}{r}
    \partial_t\RHO
      +\nabla\cdot \lbrack \RHO\BV \rbrack = 0
\\[0.5em]
    \partial_t(\RHO\BV)
      + \nabla \cdot \lbrack \RHO\BV\otimes\BV \rbrack
        + \nabla \llbracket P(\RHO,S) \rrbracket = 0
\\[0.5em]
    \partial_t (\RHO S)
      + \nabla\cdot \lbrack \RHO\BV S \rbrack = 0
  \end{array}\right\}
  \quad\text{in $\Big( \C^1_c\big( [0,T) \big)\otimes \AF \Big)^*$.}
\label{E:FULL2}
\end{equation}
Here the brackets $\lbrack \cdot \rbrack$ and $\llbracket \cdot
\rrbracket$ denote the integration of $\nu^1$ and $\nu^2$,
respectively, against suitable functions of $(\RHO,\BV,S)$. We refer
the reader to Section~\ref{SS:YM} for details.
\end{theorem}

\begin{remark}

The two Young measures from Theorem~\ref{T:GLOBAL} play slightly
different roles: One is used to describe the kinematic aspects of the
flow (transport). It is generated by a sequence of approximate
solutions that are interpolated piecewise linearly in time. The other
one is used to describe the dynamical aspects (acceleration due to
pressure gradient). It is constructed using piecewise constant
interpolants. Since the approximations of the maps $t \mapsto (\RHO_t,
\BM_t, \sigma_t)$ are sufficiently regular in time, both Young
measures generate the same conserved quantities. In order to have
equality for the non-linear terms $\RHO\BV \otimes \BV$ and
$P(\RHO,S)$, however, one needs to control the time regularity of the
total energy. This will be considered elsewhere. It requires a more
refined time interpolation, like De Giorgi's variational interpolation
for minimizing movements; see Section~3.2 in
\cite{AmbrosioGigliSavare2008}. Notice that while the derivative of $t
\mapsto \BM_t$ is uniquely determined a.e.\ in time in the closure of
the space of measures with respect to the Monge-Kantorovich norm, the
matrix measure field representing it is not.
\end{remark}

We conclude this section by highlighting some aspects of our method.

\medskip

\textbf{Maximization of Entropy Production.}
\nopagebreak

\medskip

The recent results by De Lellis, Sz\'{e}kelyhidi, and others suggest
that non-unique\-ness of weak solutions of compressible Euler
equations in several space dimensions is a fact of life. The accepted
entropy conditions, in the form \eqref{E:ED} for the isentropic case,
for example, are insufficient to select a unique solution. It is
therefore natural to at least try to identify the ``extreme''
solutions among all possible weak solutions. Since the entropy
condition \eqref{E:ED} already implies that the total energy is
non-increasing in time, it appears promising to strengthen this
condition by requiring that total energy be dissipated \emph{at
maximal rate}, as suggested by Dafermos \cite{Dafermos1973}. It was
shown in \cite{ChiodaroliKreml2014}, however, that this entropy
condition seems to favor the highly oscillatory solutions of the
isentropic Euler equations, which are non-unique.

Instead of decreasing the total energy, we will balance the
dissipation of internal energy and minimizing the work done by the
system (defined in terms of acceleration of the fluid elements), which
amounts to changing the velocity a little as possible. We partition a
given time interval into subintervals of length $\tau>0$. The updates
in each timestep are obtained as the solutions of the above
minimization problem. A similar approach has been studied for
polyconvex elasticity in \cites{DemouliniStuartTzavaras2001,
DemouliniStuartTzavaras2012}.

Recall that by the first law of thermodynamics $T \,dS = dU - W$,
where $T$ denotes temperature, $dS$ and $dU$ are the infinitesimal
changes in thermodynamical entropy and internal energy, and $W$ is the
work done \emph{on} the system by its surroundings. Classically, the
work is given by the formula $W = -p \,dV$, with $p$ the pressure and
$dV$ the infinitesimal change of volume. Instead, we will utilize the
\emph{minimal work} functional, which we will introduce in the next
paragraph. Our method boils down to minimizing the sum $W+U$,
depending on some timestep $\tau>0$. Denoting by $W_\tau, U_\tau$ the
corresponding minimizers, we obtain the inequality
\[
	U_\tau + W_\tau \LS U_0 + 0,
\]
where the index $0$ refers to the fluid state obtained by not doing
any work. Defining $dU := U_0-U_\tau$, we observe that we are trying
to \emph{maximize} $dU-W$, which can formally be interpreted as
maximizing $T \,dS$, thus maximizing the entropy production. Similar
ideas have been explored in the recent paper
\cite{BrezinaFeireisl2017}.

\medskip

\textbf{Minimal Work Functional}
\nopagebreak

\medskip

Given $\RHO \in \SP_2(\R^d)$, we denote by $\PR$ the space of Borel
probability measures $\MU \in \SP_2(\R^{2d})$ whose first marginal
$\PP^1\#\MU = \RHO$. Here $\PP^1(x,\xi) := x$ for all $(x,\xi) \in
\R^{2d}$, and $\#$ is the push-forward. Any measure $\MU\in\PR$
describes the state of some fluid with density $\RHO \in \SP_2(\R^d)$
and velocity distribution $\mu_x$ for $\RHO$-a.e.\ $x\in\R^d$, where
$\MU(dx,d\xi) =: \mu_x(d\xi) \,\RHO(dx)$ denotes the disintegration of
$\MU$ with respect to $\RHO$. The special case $\MU(dx,d\xi) =
\delta_{\BU(x)}(d\xi) \,\RHO(dx)$ represents a monokinetic state where
all fluid elements located at the position $x\in\R^d$ have the same
velocity $\BU(x)\in\R^d$ and are therefore indistinguishable. The
velocity field $\BU \in \L^2(\R^d,\RHO)$, by construction. We will
occasionally use bold letters to denote elements in $\R^{2d}$ such as
\[
  \BX = (x,\xi),
  \quad
  \BY = (y,\upsilon),
  \quad\text{and}\quad
  \BZ = (z,\zeta),
\]
where $x,y,z \in \R^d$ represent positions and $\xi,\upsilon,\zeta \in
\R^d$ velocities.

In order to measure the ``distance'' between two state measures
$\MU^1,\MU^2 \in \PR$, we will use the minimal acceleration cost
introduced in \cite{GangboWestdickenberg2009}. It is defined as
follows: For a given timestep $\tau > 0$, consider a fluid element
with initial position/velocity $\BX \in \R^{2d}$. Assume that the
fluid element transitions into a new state $\BZ \in \R^{2d}$. The
transition is described by a smooth curve $X(\cdot| \BX,\BZ) \colon
[0,\tau] \longrightarrow \R^d$ such that
\[
  (X,\dot{X})(0) = (x,\xi)
  \quad\text{and}\quad
  (X,\dot{X})(\tau) = (z,\zeta)
\]
(with $X := X(\cdot|\BX,\BZ)$). Among all such curves there are the
ones that minimize the acceleration $\int_0^\tau |\ddot{X}(t)|^2
\,dt$. They are uniquely determined and given by
\[
  X(t|\BX,\BZ) = x + t\xi
    + \Big( 3(z-x)-\tau(\zeta+2\xi) \Big) \frac{t^2}{\tau^2} 
    - \Big( 2(z-x)-\tau(\zeta+\xi) \Big) \frac{t^3}{\tau^3} 
\]
for all $t\in[0,\tau]$. The minimal acceleration can be computed
explicitly, which allows us to define a cost measuring the
``distance'' between the two end states:
\begin{equation}\label{E:COSTF}
  a_\tau(\BX,\BZ)^2 
    := 3\bigg| \frac{z-x}{\tau}-\frac{\zeta+\xi}{2}\bigg|^2
      + \frac{1}{4} |\zeta-\xi|^2
\end{equation}
for all $\BX,\BZ \in \R^{2d}$. Note that $a_\tau(\BX,\BZ) = 0$ if and
only if $z=x+\tau\xi$ and $\zeta=\xi$.

The cost function \eqref{E:COSTF} can be rewritten in the following
form:
\begin{equation}\label{E:COSTF2}
  a_\tau(\BX,\BZ)^2 
    = \frac{3}{4\tau^2} |(x+\tau\xi)-z|^2 
    + \bigg| \zeta-\bigg( \xi-\frac{3}{2\tau} \Big( (x+\tau\xi)-z \Big)
      \bigg) \bigg|^2
\end{equation} 
for every $\BX, \BZ \in \R^{2d}$. The first term measures how much the
final position $z$ differs from $x+\tau\xi$, which would be the
position of the fluid element after a \emph{free transport}. The
second term measures the difference between $\zeta$ and the velocity
\begin{equation}
  V_\tau(\BX,z) := \xi - \frac{3}{2\tau}\Big( (x+\tau\xi)-z \Big),
\label{E:ZETOP}
\end{equation}
which is the velocity that minimizes the acceleration among all curves
that connect the initial position/velocity $\BX \in \R^{2d}$ to the
final position $z \in \R^d$. Notice that by minimizing the final
velocity $\zeta$ for fixed $(\BX, z)$, setting $\zeta$ equal to
\eqref{E:ZETOP}, we closely link velocity and transport. The minimal
work functional is then defined as an optimal transport problem with
cost function \eqref{E:COSTF} and for pairs of $\MU_1, \MU_2 \in
\SP_2(\R^{2d})$, in analogy to the Wasserstein distance; see
Definition~\ref{D:WAS}.

\medskip

\textbf{Non-Coercive Energy Functional and Monotone Maps.} 
\nopagebreak

\medskip

Our variational method amounts to minimizing the sum of minimal work
functional plus internal energy of the \emph{final} state (the one
after transport). We will be particularly interested in transport
\emph{maps}, where a fluid element at location $x\in\R^d$ is
transported wholly to a new position $\BT(x) \in \R^d$, without being
split up. More precisely, we consider $\R^d$-valued transport maps
$\BT \in \L^2(\R^d, \RHO)$. Taking into account Definition~\ref{D:INT}
and assuming that the specific thermodynamical entropy $S$ is simply
transported along with the flow (recall \eqref{E:ENTCHAR}), we can
formally write
\begin{equation}
	\INT[\BT\#\RHO,\BT\#\sigma] = \int_{\R^d} U\big( \RHO(x), S(x) \big)
		\, \det\big( \nabla\BT(x) \big)^{1-\gamma} \,dx
\label{E:INTFALSE}
\end{equation}
for the internal energy of the final state, using the change of
variables formula. Here we have identified $\RHO$ with its Lebesgue
density. Moreover, we have assumed that the transport map $\BT$ is
sufficiently regular and invertible. Our minimization scheme will
ensure that the internal energy of the new fluid state is finite. But
since the map $\BT \mapsto \INT[\BT\#\RHO,\BT\#\sigma]$ is not
coercive, it does not suggest any natural function space setting to
formulate the minimization problem. In order to be able to use the
direct method of the calculus of variations, we need compactness of
sublevel sets of the internal energy functional in a suitable
topology, and its lower semicontinuity with respect to this topology.
To achieve this, we make two choices:
\begin{itemize}
\item we assume that the transport maps are \emph{monotone}, and
\item in \eqref{E:INTFALSE} we replace the gradient $\nabla\BT(x)$
  by its \emph{symmetric part};
\end{itemize}
see Proposition~\ref{P:LSCCON2}. A monotone map is locally of bounded
variation. In particular, its variation (the total variation of its
derivative) over any convex set can be controled in terms of its
oscillation (the size of its range). This provides us with the
necessary compactness of sublevel sets. On the other hand, by using
only the symmetric part of the derivative $\nabla\BT(x)$, the
resulting energy functional becomes convex and lower semicontinuous
with respect to weak* convergenve in the space of functions with
bounded variation. We refer the reader to Section~\ref{S:VTD3} for
details.

\medskip

\textbf{The Geometry of Monotone Maps.} 
\nopagebreak

\medskip

In continuum mechanics, a configuration is a function that assigns to
each point of the body manifold (called the reference configuration)
its position in physical space $\R^d$, at any given time. These maps
are required to be injective because matter must not interpenetrate.
The space of configurations therefore cannot be a vector space
(subtracting a configuration from itself, we obtain the zero map,
which is not injective). Our assumption of monotonicity of the
transport maps is consistent with these considerations, but a bit
stronger than mere injectivity. Note, however, that we require
monotonicity of the transport maps in the limit of small timesteps
$\tau>0$ only, so that the minimizing $\BT$ will be a perturbation of
the identity map, which is monotone. Recall also that in the theory of
generalized gradient flows on the space of probability measures, which
utilizes the Wasserstein distance to determine the local geometry of
the problem, the optimal transport maps are \emph{cyclically monotone}
(gradients of convex functions), which is a stronger condition than
monotonicity. We do not wish to work with cyclically monotone maps
since the induced velocity fields (obtained as limits of difference
quotients between optimal transport maps and the identity) inherit the
gradient property. For the compressible Euler equations, this would
result in the regime of potential flows; see also
\cite{Westdickenberg2010}.

Our variational methods is phrased as a convex minimization problem
over the closed convex cone of monotone transport maps. As is usual in
optimization, such a constraint may result in the appearance of
Lagrange multipliers in the optimality conditions. For the
monotonicity constraint under consideration here, it turns out that
the representation of such Lagrange multipliers fits very neatly into
the overall structure of our problem. In fact, elements in the closed
convex cone that is polar to the cone of monotone maps can be
represented as divergences of measure fields taking values in the
positive semidefinite symmetric matrices; see Section~\ref{SS:PC}.
This is precisely the form of the flux terms in the gas dynamics
equations \eqref{E:FULL}.

In continuum mechanics, admissible velocities are elements of the
tangent cone to the manifold of configurations. Consequently, if we
consider the map $t \mapsto \RHO_t$ as a curve on the manifold of
probability measures, then the corresponding velocity $\BV$ should
represent a curve in the tangent bundle. In our variational time
discretization we update the velocity as follows: we first move the
current velocity using the optimal transport map, then we project onto
a suitably defined tangent cone to the cone of monotone maps at the
new configuration. This projection will turn out to be trivial in the
cases with pressure; in the pressureless case, the projection of
velocity will be related to the sticky particle condition. This
two-step update for the velocity is similar to the construction of the
parallel transport of tangent vector fields along the space of
probability measures, as developed in \cites{Gigli2004,
AmbrosioGigli2008}.

\medskip

\textbf{Measure-Valued Solutions.}
\nopagebreak

\medskip

The Young measures of Theorem~\ref{T:GLOBAL} are obtained as weak*
limits
\[
	\nu \in \L^\infty_w\big( [0,\infty);
		\M_+(\dot\R^d\times\BFX) \big)
\]
of sequences of analogous maps constructed from approximate solutions
$(\RHO_n, \BV_n, S_n)$ of \eqref{E:FULL}. The $(t,x)$-marginal of such
$\nu$ is the weak* limit $t \mapsto \mu_t(dx)$ of
\[
	\mu_{n,t} := \RHO_{n,t} + \Big( \HA \RHO_{n,t} |\BV_{n,t}|^2 
		+ U(\RHO_{n,t},S_{n,t}) \Big)
\]
(see \eqref{E:FEN}/\eqref{E:UPS}), which captures the space-time
distribution of mass and total energy. The Young measure $\nu$
captures both oscillations and concentration in the approximating
sequence. Notice that the concept of measure-valued solutions to
hyperbolic balance laws is fairly weak. On the other hand, in view of
the non-unique\-ness results by De Lellis and Sz{\'e}kelyhidi one may
wonder whether a distinguished weak solution of \eqref{E:FULL} can be
identified at all and what sets it apart from the other solutions. It
has therefore been suggested by some researchers that the solution
concept for \eqref{E:FULL} must be reconsidered, for example in favor
of measure-valued or statistical solutions; see
\cites{LimYuGlimmLiSharp2008, LimIwerksGlimmSharp2010,
FjordholmKaeppeliMishraTadmor2014}. It would be interesting to
investigate whether the non-linear iteration techniques introduces by
De Lellis, Sz{\'e}kelyhidi, and others can be used to promote
measure-valued solutions to weak ones, at least in regions where
the flow is expected to be laminar instead of turbulent/non-unique.
 
Our variational time discretization decreases the total energy, while
preserving the entropy. This may seem backwards from the physical
point of view. We would like to point out, however, that in turbulence
it is standard to assume that solutions of the incompressible
Navier-Stokes equations converge (in the high Reynolds number limit)
to velocity fields that dissipate kinetic energy, even though they
formally solve the incompressible Euler equations. Therefore the
incompressible Euler equations seem to only give an incomplete
description of the actual physical phenomena. It is natural to expect
that similar effects occur in the compressible models. 


\section{Notation}

In the following, we will always assume that $\R^D$ is equipped with
the Euclidean inner product, for which we write $x\cdot y$ or $\langle
x,y\rangle$ with $x,y\in\R^D$.

Let $\MAT{d}$ be the space of real $(d\times d)$-matrices and
\[
	\MAT[\square]{d} := \Big\{ A \in \MAT{d} \colon 
		\text{$v\cdot (Av) \REL 0$ for all $v\in\R^d$} \Big\}
\]
where $\REL$ stands for either $\GS$ or $>$. We will refer to the
elements of $\MAT[\GS]{d}$ (resp. $\MAT[>]{d}$) as positive
semi-definite (resp. positive definite) matrices. Notice that these
matrices are not assumed to be symmetric. The analogous spaces of
symmetric matrices will be denoted by $\SYM{d}$ and $\SYM[\REL]{d}$.
We have $A \in \MAT[\REL]{d}$ if and only if $A^\S \in \SYM[\REL]{d}$
where $A^\S := (A+A^\T)/2$ is the symmetric part of $A$. The
antisymmetric part of $A$ is defined as $A^\A := (A-A^\T)/2$ and we
will denote by $\SKEW{d}$ the space of antisymmetric real $(d\times
d)$-matrices. Recall that the Frobenius inner product of matrices is
defined as
\[
	A : B := \TRACE(A^\T B)
	\quad\text{for all $A,B \in \MAT{d}$.}
\]

The norms on these spaces will be the ones induced by the inner
products.

We denote by $\CB(\R^D)$ the space of bounded continuous functions on
$\R^D$ and by $\SP(\R^D)$ the space of Borel probability measures.
Weak convergence of sequences of probability measures is defined by
testing against functions in $\CB(\R^D)$. For any $1\LS p< \infty$ we
denote by $\SP_p(\R^D)$ the space of Borel probability measures with
finite $p$th moment, so that $\int_{\R^D} |x|^p \,\RHO(dx) < \infty$
for every $\RHO \in \SP_p(\R^D)$.

\begin{definition}[$p$-Wasserstein Distance]\label{D:WAS}
For any $\RHO^1, \RHO^2 \in \SP(\R^D)$ let
\[ 
  \ADM(\RHO^1,\RHO^2) := \Big\{ \GAMMA\in\SP(\R^{2D}) \colon
    \text{$\PP^k\#\GAMMA = \RHO^k$ with $k=1..2$} \Big\}
\]
be the space of admissible transport plans connecting $\RHO^1$ and
$\RHO^2$, where
\[
  \PP^k(x^1, x^2) := x^k
  \quad\text{for all $(x^1,x^2) \in \R^{2D} = (\R^D)^2$}
\]
and $k=1..2$, and $\#$ denotes the push-forward of measures. For any
$1\LS p<\infty$ the $p$-Wasserstein distance $\WAS_p(\RHO^1,\RHO^2)$
between $\RHO^1$, $\RHO^2$ is defined by
\begin{equation}
  \WAS_p(\RHO^1,\RHO^2)^p := \inf_{\GAMMA \in \ADM(\RHO^1,\RHO^2)}
    \int_{\R^{2D}} |x^1-x^2|^p \,\GAMMA(dx^1,dx^2).
\label{E:WASDE}
\end{equation}
\end{definition}

\begin{remark}\label{R:OPTW}
The $\inf$ in \eqref{E:WASDE} is actually attained, so the set
$\OPT(\RHO^1, \RHO^2)$ of transport plans $\GAMMA$ that minimize
\eqref{E:WASDE} (called optimal transport plans) is non-empty. For
$p=2$ the support of each $\GAMMA \in \OPT(\RHO^1, \RHO^2)$ is
contained in the subdifferential of a lower semicontinuous, convex map
(therefore it is cyclically monotone). If $\RHO^1$ is absolutely
continuous with respect to the Lebesgue measure $\LEB^D$, then each
optimal transport plan is induced by a map (its support lies on the
graph of a function):
\[
  \GAMMA = (\ID, \BT)\#\RHO^1
  \quad\text{for suitable $\BT \in \L^2(\R^D, \RHO^1)$.}
\]
We refer the reader to \cite{AmbrosioGigliSavare2008} for further
details.
\end{remark}

For any $n\in\N$ and $k=1\ldots n$, we define projections
\[
  \PP^k(x^1\ldots x^n) := x^k
  \quad\text{for all $(x^1\ldots x^n) \in \R^{nd} = (\R^d)^n$.}
\]
We will also use projections $\XX$ and $\YY^k$ defined by
\[
  \XX(x,y^1\ldots y^n) := x,
  \quad
  \YY^k(x,y^1\ldots y^n) := y^k
\]
for all $(x,y^1\ldots y^n) \in \R^{(n+1)d} = (\R^d)^{n+1}$ and
$k=1\ldots n$, with $n\in\N$. Sometimes it will be convenient to write
$\ZZ^k$ or $\VV^k$ in place of $\YY^k$ (same definition), depending on
whether the symbols represent positions or velocities, which will be
clear from the context. For $n=1$ we will usually write $\YY:=\YY^1$
etc.

\begin{definition}[Distance]
Let $\RHO\in\SP_2(\R^d)$ be given and
\[
  \PR := \Big\{ \GAMMA\in\SP_2(\R^{2d}) \colon
    \XX\#\GAMMA = \RHO \Big\}.
\]
We introduce a distance as follows: for any $\GAMMA^1,\GAMMA^2 \in
\PR$ we define
\[
  \WR(\GAMMA^1,\GAMMA^2)^2
    := \int_{\R^d} \WAS(\gamma^1_x,\gamma^2_x)^2 \,\RHO(dx),
\]
where $\GAMMA^k(dx,dy) =: \gamma^k_x(dy) \,\RHO(dx)$ with $k=1..2$
denotes the disintegration of $\GAMMA^k$, and where $\WAS$ is the
Wasserstein distance on $\SP_2(\R^d)$; see
\cites{AmbrosioGigliSavare2008, Gigli2004}.
\end{definition}

\begin{definition}[Transport Plans]
Let $\RHO\in\SP_2(\R^d)$ be given.
\begin{enumerate}
\renewcommand{\labelenumi}{(\roman{enumi}.)}
\item \emph{Admissible Plans.} For any $\GAMMA^1,\GAMMA^2\in\PR$ we
  define
\[ 
  \AR(\GAMMA^1,\GAMMA^2) := \Big\{ \ALPHA\in\SP_2(\R^{3d}) \colon
    \text{$(\XX,\YY^k)\#\ALPHA = \GAMMA^k$ with $k=1..2$} \Big\}. 
\]
\item \emph{Optimal Plans.} For any $\GAMMA^1,\GAMMA^2\in\PR$ we
  define
\begin{align*}
  \OR(\GAMMA^1,\GAMMA^2) := \Big\{
    & \ALPHA\in\AR(\GAMMA^1,\GAMMA^2) \colon
\\
    & \qquad \WR(\GAMMA^1,\GAMMA^2)^2
      = \int_{\R^{3d}} |y^1-y^2|^2
        \,\ALPHA(dx,dy^1,dy^2) \Big\}.
\end{align*}
\end{enumerate}
\end{definition}

\begin{theorem}\label{T:OPTMIN}
Let $\RHO\in\SP_2(\R^d)$ be given.
\begin{enumerate}
\renewcommand{\labelenumi}{(\roman{enumi}.)}
\item The function $\WR$ is a distance on $\PR$ and lower
  semicontinuous with respect to weak convergence in $\SP_2(\R^{2d})$.
  We have
\[
  \WR(\GAMMA^1,\GAMMA^2)^2
    = \min_{\ALPHA\in\AR(\GAMMA^1,\GAMMA^2)}
      \int_{\R^{3d}} |y^1-y^2|^2 \,\ALPHA(dx,dy^1,dy^2)
\]
for all $\GAMMA^1, \GAMMA^2\in\PR$, and thus $\OR(\GAMMA^1,\GAMMA^2)$
is non-empty.
\item The set $(\PR, \WR)$ is a complete metric space.
\end{enumerate}
\end{theorem}

\begin{proof}
We refer the reader to Section~4.1 in \cite{Gigli2004}.
\end{proof}

\begin{definition}[Barycentric Projection]\label{D:BAR}
For any $\RHO\in\SP_2(\R^d)$ and $\GAMMA \in \PR$ the barycentric
projection $\BAR(\GAMMA) \in \L^2(\R^d,\RHO)$ is defined as
\[
  \BAR(\GAMMA)(x) := \int_{\R^d} y \,\gamma_x(dy)
  \quad\text{for $\RHO$-a.e.\ $x\in\R^d$,}
\]
where $\GAMMA(dx,dy) =: \gamma_x(dy) \,\RHO(dx)$ is the disintegration
of $\GAMMA$.
\end{definition}

An important subset of $\PR$ consists of those measures $\GAMMA$ that
are induced by maps: there exists a $\BT\in\L^2(\R^d, \RHO)$ taking
values in $\R^d$ such that
\[
  \GAMMA(dx,dy) = \delta_{\BT(x)}(dy) \,\RHO(dx).
\]
In this case, the distance $\WR$ reduces to the $\L^2(\R^d,
\RHO)$-distance of the corresponding maps. If $\GAMMA^1,\GAMMA^2 \in
\PR$ and $\GAMMA^1 = (\ID,\BT)\#\RHO$ with $\BT \in \L^2(\R^d,\RHO)$,
then
\begin{equation}
  \WR(\GAMMA^1,\GAMMA^2)^2 = \int_{\R^{2d}} |\BT(x)-y^2|^2 
    \,\GAMMA^2(dx,dy^2);
\label{E:HYBRID}
\end{equation}
see Lemma~5.3.2 in \cite{AmbrosioGigliSavare2008}. If
$\WR(\GAMMA^n,\GAMMA) \longrightarrow 0$ as $n\rightarrow\infty$, with
$\GAMMA^n,\GAMMA^\infty \in \PR$ and $\GAMMA^n = (\ID,\BT^n)\#\RHO$
for some $\BT^n \in \L^2(\R^d,\RHO)$, then $\BT^n \longrightarrow \BT$
strongly in $\L^2(\R^d,\RHO)$ and $\GAMMA = (\ID, \BT)\#\RHO$. Indeed,
our assumption implies that the sequence $\{\BT^n\}_n$ is Cauchy in
$\L^2(\R^d,\RHO)$ and hence converges to a limit $\BT^\infty$, by
completeness. On the other hand, since $(\ID,\BT^n,\BT^\infty) \#\RHO
\in \AR(\GAMMA^n,\GAMMA^\infty)$ with $\GAMMA^\infty :=
(\ID,\BT^\infty)\#\RHO$, we have
\[
  \WR(\GAMMA^n,\GAMMA^\infty) 
    \LS \|\BT^n-\BT^\infty\|_{\L^2(\R^d,\RHO)}
    \longrightarrow 0
  \quad\text{as $n\rightarrow\infty$.}
\]
Then we use that $\WR(\GAMMA,\GAMMA^\infty) \LS \WR(\GAMMA^n,\GAMMA) +
\WR(\GAMMA^n,\GAMMA^\infty)$. We have the estimate
\[
  \|\BAR(\GAMMA^1)-\BAR(\GAMMA^2)\|_{\L^2(\R^d,\RHO)}
    \LS \WR(\GAMMA^1,\GAMMA^2),
\]
as follows easily from Theorem~\ref{T:OPTMIN} (i.) and Jensen
inequality.


\subsection*{Minimal Work}

As outlined in the Introduction, our approach relies on a functional
measuring the work done to the fluid, called minimal work functional.

\begin{definition}[Minimal Work]\label{D:MAC}
For any pair of measures $\MU^1,\MU^2\in\SP_2(\R^{2d})$ we denote by
$\ADM(\MU^1, \MU^2)$ the set of transport plans $\OMEGA \in
\SP(\R^{4d})$ with
\[
  (\PP^1,\PP^2)\#\OMEGA = \MU^1
  \quad\text{and}\quad
  (\PP^3,\PP^4)\#\OMEGA = \MU^2.
\]
The minimal work is the functional $\ACC_\tau$ defined by
\begin{equation}
  \ACC_\tau(\MU^1,\MU^2)^2 := \inf\Bigg\{ \int_{\R^{4d}}
    a_\tau(\BX^1,\BX^2)^2 \,\OMEGA(d\BX^1,d\BX^2) \colon
      \OMEGA\in\ADM(\MU^1,\MU^2) \Bigg\}.
\label{E:ACC}
\end{equation}
\end{definition}

Note that $\ACC_\tau$ is not a distance: It is not symmetric in its
arguments $\MU^1$ and $\MU^2$, which follows from the asymmetry of the
cost function \eqref{E:COSTF}. Moreover, it does not vanish if $\MU^1
= \MU^2$. Instead, we have the following relation:
\[
  \ACC_\tau(\MU^1,\MU^2) = 0
  \quad\Longleftrightarrow\quad
  \MU^2 = F_\tau\#\MU^1,
\]
where $F_\tau\colon\R^{2d} \longrightarrow \R^{2d}$ is the \emph{free
transport map} defined by
\[
  F_\tau(\BX) := (x+\tau\xi, \xi)
  \quad\text{for all $\BX\in\R^{2d}$.}
\]
The minimal work functional measures how much each fluid element
deviates from the straight path determined by its initial velocity;
see \cite{GangboWestdickenberg2009} for more details.

When minimizing the integral in \eqref{E:ACC} over all plans $\OMEGA
\in \SP_2(\R^{4d})$ with
\[
  (\PP^1,\PP^2,\PP^3)\#\OMEGA = : \BETA
  \quad\text{for given $\BETA \in \SP_2(\R^{3d})$,}
\]
then there exists a unique such minimizer, which takes the form
$\OMEGA = H_\tau\#\BETA$, with the map $H_\tau \colon \R^{3d}
\longrightarrow \R^{4d}$ defined for all $\BX\in\R^{2d}$ and
$z\in\R^d$ as
\[
  H_\tau(\BX,z) := \Big( \BX, z, V_\tau( \BX, z) \Big).
\]
This determines the final velocity in terms of the data $\BX$ and the
new position $z$.


\section{Energy Minimization: First Properties}

In preparation of our variational time discretization for
\eqref{E:FULL}, we first consider the metric projection onto the cone
of monotone transport plans.


\subsection{Monotone Transport Plans}

To every subset $\Gamma \subset \R^d\times\R^d$ we can associate a
set-valued map $u_\Gamma \colon \R^d \longrightarrow P(\R^d)$ (where
$P(\R^d)$ is the power set of $\R^d$) by
\[
  u_\Gamma(x) := \Big\{ y\in\R^d \colon (x,y) \in \Gamma \Big\}
  \quad\text{for all $x\in\R^d$.}
\]
For any set-valued map $u \colon \R^d \longrightarrow P(\R^d)$, we
denote by
\begin{align*}
  \DOM(u) &:= \Big\{ x\in\R^d \colon u(x) \neq \varnothing \Big\},
\\
  \GRAPH(u) &:= \Big\{ (x,y)\in\R^d\times\R^d \colon y\in u(x) \Big\}
\end{align*}
its domain and graph. A subset $\Gamma \subset \R^d\times\R^d$ is
called monotone if
\[
  \langle x_1-x_2, y_1-y_2 \rangle \GS 0
  \quad\text{for any pair of $(x_i,y_i) \in \Gamma$.}
\]
Such a set is called maximal monotone if for any monotone set $\Gamma'
\subset \R^d\times\R^d$ with $\Gamma \subset \Gamma'$ we have that
$\Gamma = \Gamma'$. Equivalently, if it is not possible to enlarge
$\Gamma$ without destroying the monotonicity. We will call any
set-valued map $u$ as above (maximal) monotone if the set $\GRAPH(u)$
is (maximal) monotone.

By Zorn's lemma, any monotone set (equivalently, any monotone
set-valued map) can be extended to a maximal monotone set (map).
Typically, this extension is not unique. A maximal monotone extension
can be obtained constructively as follows: Let $\Gamma \in
\R^d\times\R^d$ be monotone. Then (for all $(x,y), (x^*,y^*) \in
\R^d\times\R^d$)
\begin{enumerate}
\item define the Fitzpatrick function
\[
  F_\Gamma(x,y) := \sup \Big\{ \langle y',x \rangle 
    + \langle y,x' \rangle - \langle y',x' \rangle \colon 
      (x',y') \in \Gamma \Big\};
\]
\item compute its Fenchel conjugate
\[
  F_\Gamma^*(y^*,x^*) := \sup \Big\{ \langle y^*,x \rangle 
    + \langle y,x^* \rangle - F_\Gamma(x,y) \colon 
      (x,y) \in \R^d\times\R^d \Big\};
\]
\item compute the proximal average
\begin{multline*}
  N_\Gamma(x,y) := \inf \bigg\{ \HA F_\Gamma(x_1,y_1) 
    + \HA F_\Gamma^*(y_2,x_2) + \HE \|x_1-x_2\|^2 
    + \HE \|y_1-y_2\|^2 \colon
\\
   (x,y) = \HA (x_1,y_1) + \HA (x_2,y_2) \bigg\}.
\end{multline*}
\end{enumerate}
The function $N_\Gamma$ is lower semicontinuous, convex, and proper,
and the set
\begin{equation}
  \bar{\Gamma} := \Big\{ (x,y) \colon 
    N_\Gamma(x,y) = \langle y,x \rangle \Big\}
\label{E:MME}
\end{equation}
is a maximal monotone extension of $\Gamma$. We refer the reader to
\cites{BauschkeWang2009, Ghoussoub2008} for details.

\begin{remark}\label{R:FUNC}
For any maximal monotone set-valued function $u \colon \R^d
\longrightarrow P(\R^d)$ the image $u(x)$ of any $x\in\R^d$ is closed
and convex (possibly empty); see Proposition~1.2 of
\cite{AlbertiAmbrosio1999}. Therefore the dimension $\DIM u(x)$ is
well-defined. The singular sets
\[
  \Sigma^k(u) := \Big\{ x\in\R^d \colon \DIM u(x) \GS k \Big\},
  \quad\text{with $k=1\ldots d$,}
\]
are countably $\HAUS^{d-k}$-rectifiable; see Theorem~2.2 of
\cite{AlbertiAmbrosio1999} for details. Here $\HAUS^n$ denotes the
$n$-dimensional Hausdorff measure. In particular, the set of points
$x\in\DOM(u)$ for which $u(x)$ contains more than one point (that is,
the set $\Sigma^1(u)$) is negligible with respect to the Lebesgue
measure $\LEB^d$. Outside $\Sigma^1(u)$ the function $u$ is
continuous. This observation will allow us to think of a maximal
monotone map $u$ as a Lebesgue measurable, \emph{single-valued}
function (just redefine $u$ on the null set $\Sigma^1(u)$).
\end{remark}

\begin{definition}[Monotone Transport Plans]
For any $\RHO \in \SP_2(\R^d)$, we define
\begin{equation}
  \CR := \Big\{ \GAMMA\in\PR \colon 
    \text{$\SPT\GAMMA$ is a monotone subset of $\R^d\times\R^d$} \Big\}.
\label{E:CONE}
\end{equation}
\end{definition}

Our definition of monotonicity for measures in $\PR$ is motivated by
the optimal transport plans of Definition~\ref{D:WAS}: an
\emph{optimal} transport plan $\GAMMA$ is characterized by the
property that $\SPT\GAMMA$ must be a \emph{cyclically monotone} set;
see Section~6.2.3 in \cite{AmbrosioGigliSavare2008}. Then there exists
a lower semicontinuous, convex, proper function $\varphi$ with
\[
  \varphi(x) + \varphi^*(y) = \langle y,x \rangle
  \quad\text{for $\GAMMA$-a.e.\ $(x,y) \in \R^d\times\R^d$.}
\]
Here $\varphi^*$ denotes the Fenchel conjugate to $\varphi$. In our
setting, the cyclical monotonicity is replaced by monotonicity, and
$N_\Gamma(x,y)$ of \eqref{E:MME} plays the role of $\varphi(x) +
\varphi^*(y)$. In the terminology of \cite{Ghoussoub2008}, the
function $N_\Gamma$ is called a self-dual Lagrangian. The cone $\CR$
contains the set of optimal transport plans defined above.

Since we do not make any assumptions on $\RHO$, its support may be a
proper subset of $\R^d$ and have ``holes.'' Fortunately, the
monotonicity constraint enables us to work with objects that are
defined on a fixed convex open subset of $\R^d$:

\begin{definition}[Associated Maps]\label{D:ASSO}
Let $\RHO \in \SP_2(\R^d)$ be given. For every $\GAMMA \in \CR$ we
call $u$ a maximal monotone map associated to $\GAMMA$ if $u$ is the
maximal monotone set-valued map induced by a maximal monotone
extension of $\Gamma := \SPT\GAMMA$.
\end{definition}

\begin{lemma}\label{L:ASSO}
For any $\RHO \in \SP_2(\R^d)$ and $\GAMMA \in \CR$, the domain of a
maximal monotone map $u$ associated to $\GAMMA$ contains the convex
open set $\Omega := \INTR\CCONV\SPT\RHO$, where $\INTR$ denotes the
interior of a set, $\CONV$ the convex hull, and $\CCONV$ its closure.
\end{lemma}

\begin{proof}
Let $\GAMMA \in \CR$ be given and consider any maximal monotone map
$u$ associated to $\GAMMA$. Then $\GRAPH(u)$ is a maximal monotone
extension of $\Gamma := \SPT\GAMMA$, which implies that the projection
$X := \PP^1(\Gamma)$ of $\Gamma$ onto $\R^d$ is contained in
$\DOM(u)$. Since
\[
  \INTR\CCONV\DOM(u) \subset \DOM(u) \subset \CCONV\DOM(u)
\]
(this is true for every maximal monotone set-valued function; see
Corollary~1.3 of \cite{AlbertiAmbrosio1999}) we conclude that the
convex open set $\INTR\CCONV(X) \subset \DOM(u)$. It therefore
suffices to show that $\INTR\CCONV(X) = \Omega$. Note that $\Omega$ is
independent of $\GAMMA$ and $u$.

To prove the claim, choose any $x \in X$ and $r > 0$. Then we can
estimate
\[
  \RHO\big( B_r(x) \big) 
    = \GAMMA\big( B_r(x) \times \R^d \big)
    \GS \GAMMA\big( B_r(x) \times B_r(y) \big) 
    > 0,
\]
for suitable $y\in\R^d$ with $(x,y) \in \Gamma = \SPT\GAMMA$. Since
$x\in X$ and $r>0$ were arbitrary, we get that $X \subset \SPT\RHO$,
which implies that $\INTR\CCONV(X) \subset \Omega$. 
\DETAIL{ 
The inclusion $\SPT\RHO \subset X$ may be false. Just consider the
case $\RHO := \LEB^1 |_{[0,1)}$ and
\[
  \GAMMA := (\ID,\BT)\#\RHO
  \quad\text{with $\DST\BT(x) := (1-x^2)^{-1/4}$ for $\RHO$-a.e.\ 
    $x\in\R$,}
\]
which is in $\L^2(\R,\RHO)$. Then $\SPT\RHO = [0,1] \neq [0,1) =
\PP^1(\SPT\GAMMA)$.
} 

Conversely, for every $x \in \Omega$ there exists a ball $B_r(x)
\subset \CCONV\SPT\RHO$ for some $r>0$. Pick an open $d$-cube $Q$
centered at $x$ as large as possible with $Q \subset B_{r/2}(x)$. Then
the closure $\overline{Q}$ is the convex hull of its corners $x_i \in
\partial B_{r/2}(x)$, which satisfy
\[
  x_i \in \CCONV\SPT\RHO
  \quad\text{for $i=1\ldots 2^d$.}
\]
Let $\ell>0$ denote the side length of $Q$ and $0 < \EPS < \ell/8$.
Then there exist
\[
  y_i \in B_\EPS(x_i) \cap \CONV\SPT\RHO
  \quad\text{for $i=1\ldots 2^d$.}
\]
Each $y_i$ can be written as a convex combination 
\[
  y_i = \sum_{k=1}^{N_i} \lambda_{i,k} z_{i,k}
  \quad
  \text{with $\lambda_{i,k} \in [0,1]$ and $\DST\sum_{k=1}^{N_i}
    \lambda_{i,k} = 1$,}
\]
for suitable $z_{i,k} \in \SPT\RHO$ and $N_i \in \N$. We now claim
that for any $z \in \SPT\RHO$ and $\EPS>0$ there exists $\bar{z} \in
B_\EPS(z) \cap X$. Assume for the moment that the claim is true. Then
for each $z_{i,k}$ we can find $\bar{z}_{i,k} \in B_\EPS(z_{i,k}) \cap
X$. We define convex combinations
\[
  \bar{y}_i := \sum_{k=1}^{N_i} \lambda_{i,k} \bar{z}_{i,k}
  \quad\text{for all $i=1\ldots 2^d$,}
\]
which satisfy $\|y_i-\bar{y}_i\| \LS \EPS$ and thus $\bar{y}_i \in
B_{2\EPS}(x_i)$ for all $i$. Consequently, the convex hull of these
$\bar{y}_i$ contains a $d$-cube centered at $x$ with side length
$\ell/2$, which in turn contains a ball $B_\delta(x)$ for $\delta>0$
small enough. By construction, this ball is a subset of the convex
hull of the $\bar{z}_{i,k} \in X$ from above, so that $x \in
\INTR\CCONV(X)$. This proves the lemma. To establish the claim, assume
that on the contrary, there exists $\EPS > 0$ with the property that
for all $\bar{z} \in B_\EPS(z)$ we have $\bar{z} \not\in X$. Then
\begin{align*}
  \RHO\big( B_\EPS(z) \big) 
    &= \GAMMA\big( B_\EPS(z)\times\R^d \big)
\\
    &\LS \GAMMA\Big( (\R^d\setminus X)\times\R^d \Big) 
      \LS \GAMMA\Big( (\R^d\times\R^d)\setminus\Gamma \Big) 
      = 0.
\end{align*}
The second equality follows from the fact that $\GAMMA$ (being a
finite Borel measure on a locally compact Hausdorff space with
countable basis) is inner regular; see \cite{Folland1999}. We conclude
that $z \not\in \SPT\RHO$, which is a contradiction.
\end{proof}


\subsection{Minimal Acceleration Cost}\label{SS:MAC}

Suppose that $\RHO \in \SP_2(\R^d)$ and $\MU \in \PR$ are given. For
any timestep $\tau>0$ we would like to minimize the acceleration
$\ACC_\tau(\MU,\GAMMA)$ over all $\GAMMA \in \PR$ with
\begin{enumerate}
\item the transport plan taking $\PP^1\#\MU$ to $\PP^1\# \GAMMA$ is
  monotone,
\item the velocity distribution of $\GAMMA$ is tangent to $\CR$
  at the new configuration
\end{enumerate} 
(where the tangency to $\CR$ is yet to be specified). As mentioned
above, this would be consistent with the usual setting in continuum
mechanics. Unfortunately, tangent cones often do not possess good
continuity properties, as can already be observed in convex polygons
in $\R^2$: the tangent cone at any point on an edge of the polygon is
a half-space. But the tangent cone collapses to a smaller set at a
corner. Consequently, the distance of a fixed point in $\R^2$ to the
tangent cone may jump upwards as the base point of the tangent cone
approaches a corner of the polygon.

We will therefore use an operator splitting: We first search for the
transport that minimizes the acceleration cost, not imposing any
restrictions on the final velocity, which will be determined a
posteriori by formula \eqref{E:ZETOP}. Then we project this new
velocity onto the tangent cone (to be defined) at the new
configuration. The second term in \eqref{E:COSTF2} now measures the
cost of realizing a feasible velocity.

As explained above, if the velocity distribution of the second measure
in \eqref{E:ACC} is not fixed, then the minimal acceleration cost
simplifies. We therefore consider the following minimization problem:
find the minimizer $\BETA_\tau \in \SP_2(\R^{3d})$ of
\begin{equation}
  \BETA \mapsto
    \frac{3}{4\tau^2} \int_{\R^{3d}} |(x+\tau\xi)-z|^2 \,\BETA(d\BX,dz)
\label{E:WOKF}
\end{equation}
among all $\BETA \in \SP_2(\R^{3d})$ with the following two
properties:
\begin{equation}\label{E:LOP1}
  \text{(1.)} \quad (\PP^1,\PP^2)\#\BETA = \MU,
  \qquad
  \text{(2.)} \quad (\PP^1,\PP^3)\#\BETA \in \CR.
\end{equation}
It will be convenient to define $\UPS_\tau := (\XX,\XX+\tau\VV)\#\MU$
and instead to minimize
\[
  \ALPHA \mapsto
   \frac{3}{4\tau^2} \int_{\R^{3d}} |y-z|^2 \,\ALPHA(dx,dy,dz)
\]
over all $\ALPHA \in \SP_2(\R^{3d})$ with the following two
properties:
\begin{equation}\label{E:LOP2}
  \text{(1.)} \quad (\PP^1,\PP^2)\#\ALPHA = \UPS_\tau,
  \qquad
  \text{(2.)} \quad (\PP^1,\PP^3)\#\ALPHA \in \CR.
\end{equation}
Notice that for every $\tau>0$ the push-forward under the map $(x,\xi)
\mapsto (x+\tau\xi,\xi)$ with $(x,\xi) \in \R^{2d}$ is an automorphism
between the spaces of measures $\ALPHA, \BETA \in \SP_2(\R^{3d})$
satisfying \eqref{E:LOP2} and \eqref{E:LOP1}, respectively. We observe
that (modulo the factor $3/4\tau^2$) we obtain exactly the
minimization that defines the distance $\WR$ (see
Theorem~\ref{T:OPTMIN}), where the second measure is allowed to range
freely over the set $\CR$. Therefore the minimization amounts to
finding the element in $\CR$ closest to $\UPS_\tau$ with respect to
the distance $\WR$, i.e., to computing the metric projection onto
$\CR$.


\subsection{Metric Projection}

In order to study the minimization problem introduced in the previous
section, we introduce on $\PR$ the analogues of scalar multiplication
and vector addition in Hilbert spaces. This will allow us to define
convexity of subsets of $\PR$ and metric projections onto such sets.

\begin{definition}[Addition/Multiplication]\label{D:AM}
Let $\RHO\in\SP_2(\R^d)$ be given.
\begin{enumerate}
\renewcommand{\labelenumi}{(\roman{enumi}.)}
\item \emph{Scaling.} For any $\GAMMA\in\PR$ and $s\in\R$
  let
\[
  s\GAMMA := (\XX,s\YY)\#\GAMMA \in \PR.
\]
\item \emph{Sum.} For any $\GAMMA^1,\GAMMA^2\in\PR$ let
\[
  \GAMMA^1\oplus\GAMMA^2 := \Big\{ (\XX,\YY^1+\YY^2)\#
    \ALPHA \colon \ALPHA\in\AR(\GAMMA^1,\GAMMA^2) \Big\}
      \subset\PR.
\]
\end{enumerate}
\end{definition}

If the plans are induced by functions, then the operations in
Definition~\ref{D:AM} reduce to the usual vector space structures on
the Hilbert space $\L^2(\R^d,\RHO)$. Note also that for all
$\GAMMA^1,\GAMMA^2\in\PR$ and $s\in\R$ we have the useful equality
\[
  \WR(s\GAMMA^1,s\GAMMA^2) = |s| \WR(\GAMMA^1,\GAMMA^2).
\]
We refer the reader to Section~4.1 in \cite{Gigli2004} for a proof.

\begin{definition}[Closed Convex Cone]\label{D:CCC}
A non-empty subset $C \subset \PR$ will be called a closed convex set
if it has the following two properties:
\begin{enumerate}
\item[(i.)] \textbf{Closed.} Consider $\GAMMA^k\in C$ and
  $\GAMMA\in\PR$ with
\[
  \WR(\GAMMA^k,\GAMMA) \longrightarrow 0
  \quad\text{as $k\rightarrow\infty$.}
\]
Then also $\GAMMA\in C$.
\item[(ii.)] \textbf{Convex.} For all $\GAMMA^1,\GAMMA^2 \in  C$ and
  $s\in[0,1]$ we have
\begin{equation}
  (1-s)\GAMMA^1 \oplus s\GAMMA^2 \subset C.
\label{E:CCIC}
\end{equation}
\end{enumerate}
The set $C$ is a closed convex cone if it also has the following
property:
\begin{enumerate}
\item[(iii.)] \textbf{Cone.} For all $\boldsymbol{\GAMMA}\in C$ and
  $s\GS 0$ we have $s\boldsymbol{\GAMMA} \in C$.
\end{enumerate}
\end{definition}

We consider metric projections onto closed convex sets in $\PR$. They
have similar properties like projections in Hilbert spaces.

\begin{proposition}[Metric Projection]\label{P:PROJ}
Let $\RHO\in\SP_2(\R^d)$ be given and $C \subset \PR$ a closed convex
set. For any $\UPS \in \PR$ there is a unique $\PP_C(\UPS) \in C$
with
\[
  \WR\big(\UPS,\PP_C(\UPS) \big) \LS \WR(\UPS,\ETA)
  \quad\text{for all $\ETA \in C$.}
\]
For every $\ETA\in C$ and all $\BETA\in\SP_2(\R^{4d})$ with
\begin{equation}\label{E:POL1}
\begin{aligned}
  (\XX,\YY^1,\YY^2)\#\BETA &\in \OR\big( \UPS,\PP_C(\UPS) \big),
\\
  (\XX,\YY^1,\YY^3)\#\BETA &\in \AR(\UPS,\ETA).
\end{aligned}
\end{equation}
we have the inequality
\begin{equation}\label{E:UNN}
  \int_{\R^{4d}} \langle y^1-y^2,y^2-y^3\rangle
    \,\BETA(dx, dy^1, dy^2, dy^3) \GS 0.
\end{equation}

Conversely, assume that there exists a $\ZETA \in \C$ with the
following property: for all $\ETA\in C$ there exists $\BETA \in
\SP_2(\R^{4d})$ with
\begin{equation}\label{E:POL2}
\begin{aligned}
  (\XX,\YY^1,\YY^2)\#\BETA &\in \AR(\UPS,\ZETA),
\\
  (\XX,\YY^1,\YY^3)\#\BETA &\in \OR(\UPS,\ETA),
\end{aligned}
\end{equation}
such that inequality \eqref{E:UNN} holds true. Then $\ZETA =
\PP_C(\UPS)$.

For any $\UPS^1, \UPS^2 \in \PR$ and any $\OMEGA \in \SP_2(\R^{5d})$
such that
\begin{equation}\label{E:OM1}
\begin{aligned}
  (\XX,\YY^1,\YY^2)\#\OMEGA &\in \OR\big( \UPS^1,\PP_C(\UPS^1) \big),
\\
  (\XX,\YY^3,\YY^4)\#\OMEGA &\in \OR\big( \UPS^2,\PP_C(\UPS^2) \big),
\end{aligned}
\end{equation}
we can estimate as follows:
\begin{equation}\label{E:COMM}
  \int_{\R^{5d}} |y^2-y^4|^2 \,\OMEGA(dx,dy^1,\ldots,dy^4)
    \LS \int_{\R^{5d}} |y^1-y^3|^2 \,\OMEGA(dx,dy^1,\ldots,dy^4). 
\end{equation}
In particular, we have the contraction $\WR(\PP_C(\UPS^1),
\PP_C(\UPS^2)) \LS \WR(\UPS^1, \UPS^2)$.
\end{proposition}

\begin{proof}
The proof is similar to the one of Proposition~4.30 in
\cite{Gigli2004}.
\medskip

\textbf{Step~1.} Let $d := \inf\{ \WR(\UPS,\ETA) \colon \ETA\in C \}
\GS 0$ and consider a sequence of plans $\ETA^n\in C$ such that
$\WR(\UPS, \ETA^n) \longrightarrow d$ as $n\rightarrow\infty$. For any
pair of indices $m,n\in\N$ choose $\BETA^{m,n} \in \SP_2(\R^{4d})$
with the property that
\begin{align*}
  (\XX,\YY^1,\YY^2)\# \BETA^{m,n}
    & \in \OR(\UPS,\ETA^m),
\\
  (\XX,\YY^1,\YY^3)\# \BETA^{m,n}
    & \in \OR(\UPS,\ETA^n),
\end{align*}
and define the plans
\begin{alignat*}{2}
  \ALPHA^{m,n}
    & := (\XX,\YY^2,\YY^3)\#\BETA^{m,n}
    && \in\AR(\ETA^m,\ETA^n),
\\
  \ETA^{m,n}
    & := (\XX, \HA\YY^2+\HA\YY^3)\# \BETA^{m,n}
    && \in C.
\end{alignat*}
The last inclusion follows from convexity \eqref{E:CCIC}. We claim
that the sequence $\{\ETA^n\}_n$ is a Cauchy sequence with respect to
$\WR$. Indeed we have
\begin{align*}
  & \frac{1}{2}\WR(\ETA^m,\ETA^n)^2
\\
  & \quad
    \LS \int_{\R^{4d}} \frac{1}{2}|y^2-y^3|^2
      \,\BETA^{m,n}(dx,dy^1 \ldots dy^3)
\\
  & \quad
    = \int_{\R^{4d}} \Bigg( |y^1-y^2|^2 + |y^1-y^3|^2
      - 2\bigg| y^1-\frac{y^2+y^3}{2} \bigg|^2 \Bigg)
        \,\BETA^{m,n}(dx,dy^1 \ldots dy^3)
\\
  & \vphantom{\int_{\R^{4d}}\bigg(} \quad
    \LS \WR(\UPS,\ETA^m)^2
      + \WR(\UPS,\ETA^n)^2
      - 2\WR(\UPS,\ETA^{m,n})^2
\end{align*}
for all $m,n\in\N$. Notice that $\WR(\UPS,\ETA^{m,n}) \GS d$
because $\ETA^{m,n}\in C$. This yields
\begin{equation}
  \frac{1}{2}\WR(\ETA^m,\ETA^n)^2
    \LS \WR(\UPS,\ETA^m)^2
      + \WR(\UPS,\ETA^n)^2 - 2d^2.
\label{E:CAUCHY}
\end{equation}
Since by assumption the sequence $\{\ETA^n\}_n$ is minimizing, the
right-hand side of \eqref{E:CAUCHY} converges to zero as $m,n
\rightarrow\infty$, which proves our claim. Recall that $(\PR, \WR)$
is a complete metric space. It follows that there is a
$\PP_{C}(\UPS) \in C$ with the property that $\WR(\ETA^n,
\PP_{C}(\UPS))\longrightarrow 0$. By lower semicontinuity of the
distance $\WR$, we now have $\WR(\UPS,\PP_{C}(\UPS)) = d$. This
establishes the existence of a minimizer.
\medskip

\textbf{Step~2.} To prove uniqueness, assume that there exists
$\ETA\in C$ with $\WR(\UPS,\ETA) = d$. Now choose a plan
$\BETA\in\SP_2(\R^{4d})$ that satisfies
\begin{align*}
  (\XX,\YY^1,\YY^2)\# \BETA
    & \in \OR\big( \UPS,\PP_C(\UPS) \big),
\\
  (\XX,\YY^1,\YY^3)\# \BETA
    & \in \OR(\UPS,\ETA),
\end{align*}
and define the plans
\begin{alignat*}{2}
  \ALPHA
    & := (\XX,\YY^1,\HA\YY^2+\HA\YY^3)\#\BETA
    && \in\AR(\UPS,\bar{\ETA}),
\\
  \bar{\ETA}
    & := (\XX,\HA\YY^2+\HA\YY^3)\# \BETA
    && \in C.
\end{alignat*}
The last inclusion again follows from convexity \eqref{E:CCIC}. We can
then estimate
\begin{align}
  2\WR(\UPS,\bar{\ETA})^2
    & \LS \int_{\R^{4d}} 2\bigg|y^1-\frac{y^2+y^3}{2}
      \bigg|^2 \,\BETA(dx,dy^1 \ldots dy^3)
\nonumber\\
    & = \int_{\R^{4d}} \bigg( |y^1-y^2|^2 + |y^1-y^3|^2
      - \frac{1}{2}|y^2-y^3|^2 \bigg)
        \,\BETA(dx,dy^1 \ldots dy^3)
\nonumber\\
    & \vphantom{\int_{\R^{4d}}\bigg(}
      \LS \WR\big( \UPS,\PP_{C}(\UPS) \big)^2
        + \WR(\UPS,\ETA)^2
        - \frac{1}{2}\WR\big( \PP_{C}(\UPS),\ETA \big)^2.
\label{E:WASEST}
\end{align}
By our choice of $\ETA$, we obtain $\WR(\UPS,\bar{\ETA})^2 \LS d^2 -
\frac{1}{4}\WR(\PP_{C}(\UPS),\ETA)^2$, which shows that if
$\PP_{C}(\UPS)$ and $\ETA$ are different, then $\WR
(\UPS,\bar{\ETA}) < d$. This contradicts the definition of $d$ because
$\bar{\ETA}\in C$. Therefore the minimizer must be unique. 
\medskip

\textbf{Step~3.} For any $\ETA\in C$ consider now $\BETA \in
\SP_2(\R^{4d})$ with \eqref{E:POL1}. For every $s>0$ we define $\ETA_s
:= (\XX,(1-s)\YY^2+s\YY^3)\# \BETA \in C$; see \eqref{E:CCIC}. Then
\begin{align*}
  & \int_{\R^{4d}} |y^1-y^2|^2 \,\BETA(dx,dy^1,dy^2,dy^3)
    = \WR\big( \UPS,\PP_{C}(\UPS) \big)^2
\\
  & \qquad
    \LS \WR(\UPS,\ETA_s)^2
      \LS \int_{\R^{4d}} \big| y^1-\big( (1-s)y^2+sy^3 
        \big)\big|^2 \,\BETA(dx,dy^1,dy^2,dy^3),
\end{align*}
which implies the estimate
\begin{align}
  0 & \GS -2s \int_{\R^{4d}} \langle y^1-y^2,
        y^2-y^3 \rangle \,\BETA(dx,dy^1, dy^2,dy^3)
\nonumber\\
    & \quad -s^2 \int_{\R^{4d}} |y^2-y^3|^2 
        \,\BETA(dx,dy^1,dy^2,dy^3).
\label{E:SECOND}
\end{align}
Notice that the second integral on the right-hand side of
\eqref{E:SECOND} is finite. Dividing the inequality \eqref{E:SECOND} by
$-2s<0$ and letting $s\rightarrow 0$, we obtain \eqref{E:UNN}.

Conversely, let $\ZETA \in C$. Assume that for every $\ETA \in C$
there exists $\BETA \in \SP_2(\R^{4d})$ with \eqref{E:POL2} satisfying
\eqref{E:UNN}. Then we can estimate as follows:
\begin{align*}
  & \WR(\UPS,\ZETA)^2 - \WR(\UPS,\ETA)^2
    \LS \int_{\R^{4d}} \Big( |y^1-y^2|^2-|y^1-y^3|^2 \Big)
      \,\BETA(dx,dy^1,dy^2,dy^3)
\\
  & \qquad
    = -2 \int_{\R^{4d}} \langle y^1-y^2,y^2-y^3 \rangle
        \,\BETA(dx,dy^1,dy^2,dy^3)
\\
  & \qquad\quad
      -\int_{\R^{4d}} |y^2-y^3|^2 \,\BETA(dx,dy^1,dy^2,dy^3),
\end{align*}
which is non-positive, by assumption. Since $\ETA \in C$ was
arbitrary, the plan $\ZETA$ must be equal to the uniquely determined
metric projection $\PP_{C}(\UPS)$.
\medskip

\textbf{Step~4.} Consider now $\UPS^1, \UPS^2 \in \PR$ and their metric
projections onto $C$. For all $\ALPHA \in \AR(\UPS^1, \UPS^2)$ there
exists $\OMEGA \in \SP_2(\R^{5d})$ with $(\XX,\YY^1,\YY^3)\#\OMEGA =
\ALPHA$ and \eqref{E:OM1}. Since $(\XX,\YY^1,\YY^4)\#\OMEGA \in
\AR(\UPS^1, \PP_{C}(\UPS^2))$, we apply \eqref{E:UNN} and obtain
\begin{equation}
  \int_{\R^{5d}} \langle y^1-y^2,y^2-y^4 \rangle 
    \,\OMEGA(dx,dy^1,\ldots,dy^4) \GS 0.
\label{E:QW1}
\end{equation}
Similarly, since $(\XX,\YY^3,\YY^2)\#\OMEGA \in \AR(\UPS^2,
\PP_C(\UPS^1))$, we have
\begin{equation}
  \int_{\R^{5d}} \langle y^3-y^4,y^4-y^2 \rangle 
    \,\OMEGA(dx,dy^1,\ldots,dy^4) \GS 0.
\label{E:QW2}
\end{equation}
Adding \eqref{E:QW1} and \eqref{E:QW2} and using the Cauchy-Schwarz
inequality, we get \eqref{E:COMM}. The left-hand side of
\eqref{E:COMM} is always bigger than or equal to $\WR(\PP_C(\UPS^1),
\PP_C(\UPS^2))^2$. The right-hand side equals $\WR(\UPS^1,\UPS^2)^2$
whenever $\ALPHA \in \OR(\UPS^1,\UPS^2)$.
\end{proof}


\subsection{Non-Splitting Projections}\label{SS:RTM}

Under a suitable assumption on the closed convex set, the projections
in Proposition~\ref{P:PROJ} can be expressed in terms of maps.

\begin{assumption}\label{A:COM}
For any $\ETA \in \PR$ and $\ZETA \in C$ we consider the
disintegrations $\ETA(dx,dy) =: \eta_x(dy) \,\RHO(dx)$ and
$\ZETA(dx,dy) =: \zeta_x(dy) \,\RHO(dx)$. We assume that if
\begin{equation}\label{E:ASSUM}
  \SPT\eta_x \subset \CCONV(\SPT\zeta_x)
  \quad\text{for $\RHO$-a.e.\ $x\in\R^d$,}
\end{equation}
then also $\ETA \in C$. 
\end{assumption}

\begin{proposition}[Properties of $\PP_C(\UPS)$]\label{P:BOREL}
Let the closed convex set $C \subset \PR$ satisfy
Assumption~\ref{A:COM} and let $\PP_C(\UPS)$ be the metric
projection of $\UPS \in \PR$ onto $C$; see Proposition~\ref{P:PROJ}.
Then there exists a unique $\ORZ_\UPS \in \L^2(\R^{2d},\UPS)$ with
\begin{equation}\label{E:OPM}
  \OR\big( \UPS,\PP_C(\UPS) \big) 
    = \Big\{ (\XX,\YY,\ORZ_\UPS)\#\UPS \Big\}.
\end{equation}
\end{proposition}

\begin{proof}
The proof is similar to the one of Propositions~4.32 in \cite{Gigli2004}.
\medskip

\textbf{Step~1.} Fix any $\ALPHA \in \OR(\UPS, \PP_C(\UPS))$ and
consider the disintegration
\begin{equation}\label{E:DIS}
  \ALPHA(dx,dy,dz) =: \alpha_{(x,y)}(dz) \,\UPS(dx,dy).
\end{equation}
Then we define the function
\begin{equation}
  \ORZ_\UPS(x,y) := \int_{\R^d} z \,\alpha_{(x,y)}(dz)
  \quad\text{for $\UPS$-a.e.\ $(x,y)\in\R^{2d}$,}
\label{E:DEFT}
\end{equation}
and the plans
\begin{alignat*}{2}
  \bar{\ALPHA}
    & := (\XX,\YY,\ORZ_\UPS)\# \UPS
    && \in\AR(\UPS,\bar{\UPS}),
\\
  \bar{\UPS}
    & := (\XX,\ORZ_\UPS)\# \UPS.
\end{alignat*}
We claim that $\bar{\UPS}\in C$. Notice first that clearly
\[
  \ORZ_\UPS(x,y) \subset \CCONV(\SPT\alpha_{(x,y)})
  \quad\text{for $\UPS$-a.e.\ $(x,y)\in\R^{2d}$.}
\]
Now consider the disintegrations
\begin{align*}
  \bar{\UPS}(dx,dz)
    & = : \bar{\upsilon}_x(dz) \,\RHO(dx),
\\
  \UPS(dx,dy)
    & = : \upsilon_x(dy) \,\RHO(dx),
\\
  \PP_C(\UPS)(dx,dz)
    & = : \PP_C(\upsilon)_x(dz) \,\RHO(dx).
\end{align*}
It follows that
\[
  \PP_C(\upsilon)_x(dz)
    = \int_{\R^d} \alpha_{(x,y)}(dz) \,\upsilon_x(dy)
  \quad\text{for $\RHO$-a.e.\ $x\in\R^d$,}
\]
and so $\SPT\alpha_{(x,y)} \subset \SPT\PP_C(\upsilon)_x$ for
$\upsilon_x$-a.e.\ $y\in\R^d$. This yields
\[
  \SPT\bar{\upsilon}_x \subset \CCONV\big( \SPT\PP_C(\upsilon)_x \big)
  \quad\text{for $\RHO$-a.e.\ $x\in\R^d$.}
\]
By Asssumption~\ref{A:COM}, this implies that $\bar{\UPS} \in C$.

Using that $\bar{\ALPHA} \in \AR(\UPS, \bar{\UPS})$, we now
estimate
\begin{align*}
  \WR(\UPS,\bar{\UPS})^2
    & \LS \int_{\R^{2d}} |y-\ORZ_\UPS(x,y)|^2 \,\UPS(dx,dy)
\\
    & = \int_{\R^{2d}} \bigg| \int_{\R^d} (y-z)
      \,\alpha_{(x,y)}(dz) \bigg|^2 \,\UPS(dx,dy)
\\
    & \LS \int_{\R^{2d}} \int_{\R^d} |y-z|^2
      \,\alpha_{(x,y)}(dz) \,\UPS(dx,dy)
     = \WR(\UPS,\PP_{C}(\UPS))^2.
\end{align*}
The first equality follows from \eqref{E:DEFT} and the second one from
\eqref{E:DIS}. For the second inequality we have used Jensen's
inequality. Recall that Jensen's inequality is strict unless the
probability measure is a Dirac measure, which implies that if
$\alpha_{(x,y)}$ is not a Dirac measure for $\UPS$-a.e.\
$(x,y)\in\R^{2d}$, then $\WR(\UPS,\bar{\UPS}) < \WR(\UPS,
\PP_C(\UPS))$. This contradicts the definition of $\PP_C(\UPS)$
since $\bar{\UPS}\in C$. We conclude that
\[
  \alpha_{(x,y)}(dz) = \delta_{\ORZ_\UPS(x,y)}(dz)
  \quad\text{for $\UPS$-a.e.\ $(x,y)\in\R^{2d}$,}
\]
and thus $\ALPHA=\bar{\ALPHA}$. The same argument works for all
$\ALPHA \in\OR(\UPS,\PP_C(\UPS))$, and so all optimal
transport plans between $\UPS$ and $\PP_C(\UPS)$ are induced
by maps.
\medskip

\textbf{Step~2.} To prove uniqueness, assume there exist two maps
$\ORZ^1, \ORZ^2 \in \L^2(\R^{2d},\UPS)$ such that $\ALPHA^k := (\XX,
\YY, \ORZ^k)\#\UPS \in \OR(\UPS, \PP_C(\UPS))$ for $k=1..2$.
Let
\begin{align*}
  \bar{\BETA}
    & := (\XX,\YY,\ORZ^1,\ORZ^2)\#\UPS,
\\
  \bar{\ALPHA}
    & := (\XX,\YY^1,\HA\YY^2+\HA\YY^3)\#\bar{\BETA}.
\end{align*}
We claim that $\bar{\ALPHA} \in \OR(\UPS,\PP_C(\UPS))$. If
this is true, and if $\ORZ^1$ and $\ORZ^2$ are different, then
$\bar{\ALPHA}$ is not induced by a map, in contradiction to what we
proved in Step~1. We can therefore define $\ORZ_\UPS$ unambiguously by
the property \eqref{E:OPM}.

To prove the claim, let $\bar{\UPS} := (\XX,\YY^2)\# \bar{\ALPHA} =
(\XX,\HA\YY^2+\HA\YY^3)\# \bar{\BETA}$ and note that
\[
  (\XX,\YY^2,\YY^3)\# \bar{\BETA}
    \in \AR\big( \PP_C(\UPS),\PP_C(\UPS) \big).
\]
Then $\bar{\UPS}\in C$ because of convexity \eqref{E:CCIC}. We have
$\bar{\ALPHA} \in \AR(\UPS,\bar{\UPS})$ and
\begin{align*}
  (\XX,\YY^1,\YY^2)\# \bar{\BETA}
    & \in \OR\big( \UPS,\PP_C(\UPS) \big),
\\
  (\XX,\YY^1,\YY^3)\# \bar{\BETA}
    & \in \OR\big( \UPS,\PP_C(\UPS) \big).
\end{align*}
Arguing as in estimate \eqref{E:WASEST}, we obtain that $\WR(\UPS,
\bar{\UPS}) \LS \WR(\UPS, \PP_C(\UPS)) = d$, which shows that
$\bar{\UPS}=\PP_C(\UPS)$, by uniqueness of the minimizer.
\end{proof}

\begin{remark}
Under Assumption~\ref{A:COM}, the third part of
Proposition~\ref{P:PROJ} simplifies as follows: for any plans $\UPS^1,
\UPS^2 \in \PR$ let $\ORZ^k \in \L^2(\R^{2d}, \UPS^k)$ be defined as
in \eqref{E:OPM} for $k=1..2$. Then we have the following inequality:
\begin{align*}
  & \int_{\R^{3d}} |\ORZ^1(x,y^1)-\ORZ^2(x,y^2)|^2 \,\ALPHA(dx,dy^1,dy^2)
\\
  & \qquad
    \LS \int_{\R^{3d}} |y^1-y^2|^2 \,\ALPHA(dx,dy^1,dy^2)
  \quad\text{for all $\ALPHA\in\AR(\UPS^1,\UPS^2)$.}
\end{align*}
\end{remark}

\begin{remark}\label{R:CONE}
If $C$ is a closed convex cone, then \eqref{E:UNN} implies the
following statement: for every $\ETA\in C$ and all $\ALPHA \in
\AR(\UPS,\ETA)$ we have that
\begin{gather}
  \int_{\R^{3d}} \langle y-\ORZ_\UPS(x,y),
    \ORZ_\UPS(x,y) \rangle \,\UPS(dx,dy) = 0,
\label{E:INDREI}\\
  \int_{\R^{3d}} \langle y-\ORZ_\UPS(x,y),
    z \rangle \,\ALPHA(dx,dy,dz) \LS 0.
\label{E:INVIER}
\end{gather}
Indeed note first that because of \eqref{E:OPM}, the inequality
\eqref{E:UNN} reads as follows:
\begin{equation}
  \int_{\R^{3d}} \langle y-\ORZ_\UPS(x,y), \ORZ_\UPS(x,y)-z
      \rangle \,\ALPHA(dx,dy,dz)
    \GS 0
\label{E:INE1}
\end{equation}
for all $\ETA, \ALPHA$ as above. We have $\ETA_0 := (\ID,0)\#\RHO \in
C$ since $C$ is a cone. Then
\[
  \ALPHA^1 := (\XX,\YY,0)\#\UPS 
    \in \AR(\UPS,\ETA_0).
\]
Using $\ALPHA^1$ in \eqref{E:INE1}, we obtain the inequality
\begin{equation}
  \int_{\R^{2d}} \langle y-\ORZ_\UPS(x,y), \ORZ_\UPS(x,y)
      \rangle \,\UPS(dx,dy)
    \GS 0.
\label{E:INEINS}
\end{equation}
On the other hand, we have $2\PP_{C}(\UPS)\in C$ since $C$ is a cone.
Then
\[
  \ALPHA^2 := (\XX,\YY,2\ORZ_\UPS)\#\UPS 
    \in \AR\big( \UPS,2\PP_C(\UPS) \big).
\]
Using $\ALPHA^2$ in \eqref{E:INE1}, we obtain the inequality
\begin{equation}
  \int_{\R^{2d}} \langle y-\ORZ_\UPS(x,y), -\ORZ_\UPS(x,y)
      \rangle \,\UPS(dx,dy)
    \GS 0.
\label{E:INZWEI}
\end{equation}
We now combine \eqref{E:INEINS} and \eqref{E:INZWEI}, and get
\eqref{E:INDREI} and thus \eqref{E:INVIER}.
\end{remark}


\subsection{Monotone Transport Plans}

Propositions~\ref{P:PROJ} and \ref{P:BOREL} apply to $\CR$.

\begin{proposition}[Monotone Transport Plans]\label{P:CONE}
Let $\RHO\in\SP_2(\R^d)$. Then
\[
  \CR := \Big\{ \GAMMA \in \PR \colon 
    \text{$\SPT\GAMMA$ is a monotone subset of $\R^d\times\R^d$} \Big\}
\]
(which is \eqref{E:CONE}) is a closed convex cone and
Assumption~\ref{A:COM} is satisfied.
\end{proposition}

\begin{proof}
We proceed in three steps.
\medskip

\textbf{Step~1.} Consider first plans $\GAMMA^k$ and $\GAMMA$ as in
Definition~\ref{D:CCC}~(i.). Since
\[
    \WAS(\GAMMA^k,\GAMMA) \LS \WR(\GAMMA^k,\GAMMA)
\]
we have that $\GAMMA^k \longrightarrow \GAMMA$ with respect to the
Wasserstein distance, and thus narrowly; see Proposition~7.1.5 in
\cite{AmbrosioGigliSavare2008}. One can check that $\GAMMA \in \PR$.
Fix $(x_i, y_i)\in\SPT\GAMMA$ with $i=1..2$. Since narrow convergence
of probability measures implies Kuratowski convergence of their
supports (see Proposition~5.1.8 in \cite{AmbrosioGigliSavare2008}),
there exist
\[
  (x_i^k,y_i^k) \in \SPT\GAMMA^k
  \quad\text{such that}\quad
  \lim_{k\rightarrow\infty} (x_i^k,y_i^k) = (x_i,y_i)
\]
for $i=1..2$. Since $\SPT\GAMMA^k$ is monotone for all $k$, we obtain
that
\[
  \langle x_1-x_2, y_1-y_2 \rangle
    = \lim_{k\rightarrow\infty} \langle x_1^k-x_2^k, y_1^k-y_2^k 
      \rangle \GS 0.
\]
Since the $(x^i,y^i)$ were arbitrary, we conclude that $\GAMMA \in
\CR$. \medskip

\textbf{Step~2.} Let now $\GAMMA^1, \GAMMA^2 \in \CR$ and $s\in[0,1]$ be
given. For any $\ALPHA \in \AR(\GAMMA^1, \GAMMA^2)$ we define the
interpolation transport plan $\GAMMA_s := (\XX, (1-s)\YY^1+s\YY^2)\#
\ALPHA \in \PR$. Consider now any point $(x,y) \in \SPT\GAMMA_s$. By
the definition of support of a measure, for all $\EPS>0$ there exists
$(\hat{x}, \hat{y}^1, \hat{y}^2) \in \SPT\ALPHA$ such that
\begin{equation}
\label{E:BALL}
  \hat{x} \in B_\EPS(x)
  \quad\text{and}\quad
  (1-s)\hat{y}^1+s\hat{y}^2 \in B_\EPS(y).
\end{equation}
Indeed, assume there exists an $\EPS>0$ with the property that for all
$(\hat{x}, \hat{y}^1, \hat{y}^2) \in \SPT\ALPHA$ statement \eqref{E:BALL}
is wrong. Then $\GAMMA_s(B_\EPS(x)\times B_\EPS(y)) = 0$, which is a
contradiction to our choice $(x,y) \in \SPT\GAMMA_s$. Now $(\hat{x},
\hat{y}^1, \hat{y}^2) \in \SPT\ALPHA$ implies that
\[
  0 < \ALPHA\Big( B_r(\hat{x})\times B_r(\hat{y}^1)
      \times B_r(\hat{y}^2) \Big) 
    \LS \GAMMA^k\Big( B_r(\hat{x})\times B_r(\hat{y}^k) \Big)
\]
for all $r>0$, with $k=1..2$. We conclude that $(\hat{x},\hat{y}^k) \in
\SPT\GAMMA^k$.

We can now apply the above argument to a pair of points $(x_i,y_i) \in
\SPT\GAMMA_s$, with $i=1..2$. For any $\EPS>0$ we find $(\hat{x}_i,
\hat{y}_i^k) \in \SPT\GAMMA^k$, $i=1..2$, such that
\[
  \hat{x}_i \in B_\EPS(x_i)
  \quad\text{and}\quad
  (1-s)\hat{y}_i^1+s\hat{y}_i^2 \in B_\EPS(y_i).
\]
Since $\SPT\GAMMA^k$ is monotone, we obtain the estimate
\begin{align*}
  \langle x_1-x_2, y_1-y_2 \rangle
    & \GS (1-s) \langle \hat{x}_1-\hat{x}_2, 
        \hat{y}_1^1-\hat{y}_2^1 \rangle
      + s \langle \hat{x}_1-\hat{x}_2, 
        \hat{y}_1^2-\hat{y}_2^2 \rangle
      - 4(M+\EPS)\EPS
\\
    & \GS -4(M+\EPS)\EPS,
\end{align*}
with $M := \max_i\{|x_i|,|y_i|\}$. Since $\EPS>0$ and $(x_i,y_i) \in
\SPT\GAMMA_s$ were arbitrary, we get that $\SPT\GAMMA_s$ is monotone.
Since $\ALPHA \in \AR(\GAMMA^1, \GAMMA^2)$ was arbitrary, we obtain
\eqref{E:CCIC}. In a similar way, one proves that if $\GAMMA \in \CR$,
then also $s\GAMMA \in \CR$ for all $s\GS 0$. 
\medskip

\textbf{Step~3.} In order to prove Assumption~\ref{A:COM}, note
that if $\ZETA \in \CR$, then its support is contained in the graph of
a maximal monotone set-valued map $u$ (we may consider a suitable
extension if necessary). For $\RHO$-a.e.\ $x\in\R^d$ we have $\SPT
\zeta_x \subset u(x)$, which is a closed and convex set; see
\cite{AlbertiAmbrosio1999}. Then $\SPT\eta_x \subset u(x)$ as well
because of assumption \eqref{E:ASSUM}, which implies that the support
of $\ETA$ is monotone and hence $\ETA \in \CR$.
\end{proof}

\begin{remark}\label{R:REMQ}
Proposition~\ref{P:BOREL} implies that whenever $\UPS \in \PR$ is
induced by a map, i.e., there exists a $\BT \in \L^2(\R^d,\RHO)$
taking values in $\R^d$ such that
\[
  \UPS(dx,dy) = \delta_{\BT(x)}(dy) \,\RHO(dx),
\]
then the projection $\PP_\CR(\UPS)$ is induced by a map as well:
\[
  \PP_\CR(\UPS)(dx,dz)
    = \delta_{\ORZ_\UPS(x,\BT(x))}(dz) \,\RHO(dx).
\]
Notice that if $\RHO$ is absolutely continuous with respect to the
Lebesgue measure, then all monotone transport plans in $\CR$ are in
fact induced by maps. This follows in the same way as for optimal
transport plans (which are contained in the subdifferentials of convex
functions, thus monotone): the set of points where a (maximal)
monotone set-valued map is multi-valued is a Lebesgue null set; see
\cite{AlbertiAmbrosio1999}.
\end{remark}


\section{Energy Minimization: Pressureless Gases}

For our variational time discretization of the pressureless gas
dynamics equations \eqref{E:PGD}, we divide the time interval $[0,T]$
into subintervals of length $\tau>0$. For every timestep, we minimize
the work \ref{D:MAC} over the cone of monotone transport plans. As
explained in Section~\ref{SS:MAC}, this reduces to a metric
projection, which further simplifies to a minimization over a closed
convex cone in a Hilbert space: we may consider monotone transport
\emph{maps} instead of plans because of Proposition~\ref{P:BOREL}.


\subsection{Configuration Manifold}

Going back to our original setup, we will consider monotone transport
maps that are defined on measures $\MU \in \PR$ representing the
distribution of mass and \emph{velocity}, not representing transport
plans.

\begin{definition}[Configurations]\label{D:CONFIG}
For any $\RHO \in \SP_2(\R^d)$ and $\MU \in \PR$, let
\[
  \CMU := \Big\{ \ORT \in \L^2(\R^{2d},\MU) \colon
    (\XX,\ORT)\#\MU \in \CR \Big\}.
\]
\end{definition}

\begin{lemma}[Closed Convex Cone]
$\CMU$ is a closed convex cone in $\L^2(\R^{2d},\MU)$.
\end{lemma}

\begin{proof}
We observe first that for any $\ORT^1, \ORT^2 \in \CMU$, we have
that
\begin{equation}\label{E:ADMA}
  (\XX,\ORT^1,\ORT^2)\#\MU \in \AR(\GAMMA^1,\GAMMA^2)
\end{equation}
where $\GAMMA^n := (\XX,\ORT^n)\#\MU \in \CR$ with $n=1..2$. This
implies the estimate
\begin{equation}\label{E:INDUC}
  \WR(\GAMMA^1,\GAMMA^2) \LS \|\ORT^1-\ORT^2\|_{\L^2(\R^{2d},\MU)}.
\end{equation}
Consider now a sequence $\ORT^k \longrightarrow \ORT$ in
$\L^2(\R^{2d},\MU)$ and define $\GAMMA^k := (\XX,\ORT^k)\#\MU$. Let
$\GAMMA := (\XX,\ORT)\#\MU$ and notice that $\WR(\GAMMA^k,\GAMMA)
\longrightarrow 0$ because of \eqref{E:INDUC}. If now $\ORT^k \in
\CMU$ and thus $\GAMMA^k \in \CR$ for all $k$, then also $\GAMMA \in
\CR$ since $\CR$ is closed with respect to $\WR$; see
Proposition~\ref{P:CONE}. This proves that $\ORT \in \CMU$. For any
$s \in [0,1]$ we have
\[
  \GAMMA_s := \big( \XX,(1-s)\ORT^1+s\ORT^2 \big)\#\MU
    = \big( \XX,(1-s)\YY^1+s\YY^2 \big)\#\ALPHA^{1,2},
\]
where $\ALPHA^{1,2} := (\XX,\ORT^1,\ORT^2)\#\MU$ with $\ORT^n \in
\CMU$ and $n=1..2$. Using \eqref{E:ADMA}, we conclude that $\GAMMA_s
\in (1-s)\GAMMA^1\oplus s\GAMMA^2$ (see Definition~\ref{D:AM}), which
is in $\CR$, by Proposition~\ref{P:CONE}. Hence $(1-s)\ORT^1+s\ORT^2
\in \CMU$. The proof that $\CMU$ is a cone is analogous.
\end{proof}

\begin{remark}\label{R:SIM}
For every $\tau>0$ let $\ORZ_\tau$ be the unique map defined in
Proposition~\ref{P:BOREL} representing the metric projection of
$\UPS_\tau := (\XX,\XX+\tau\VV)\#\MU$ onto the closed convex cone
$\CR$. We define $\ORT_\tau(x,\xi) := \ORZ_\tau(x, x+\tau\xi)$
for $\MU$-a.e.\ $(x,\xi) \in \R^{2d}$ so that
\begin{equation}\label{E:ZTOT}
  (\XX,\YY,\ORZ_\tau)\#\UPS_\tau 
    = (\XX,\XX+\tau\VV,\ORT_\tau)\#\MU.
\end{equation}
Then $\ORT_\tau$ must be the uniquely determined metric projection of
$\XX+\tau\VV \in \L^2(\R^{2d},\MU)$ onto the cone $\CMU$ (see
\cite{Zarantonello1971} for more information about metric projections
in Hilbert spaces). Indeed for any $\ORS \in \CMU$ (for which
$\GAMMA_\ORS := (\XX,\ORS)\#\MU \in \CR$) we have
\begin{align*}
  \|(\XX+\tau\VV)-\ORT_\tau\|_{\L^2(\R^{2d},\MU)}   
    &= \|\YY-\ORZ_\tau\|_{\L^2(\R^{2d},\UPS_\tau)}
\\
    &= \WR\big( \UPS_\tau,\PP_\CR(\UPS_\tau) \big)
\\
    &\LS \WR(\UPS_\tau,\GAMMA_\ORS)
      \LS \|(\XX+\tau\VV)-\ORS\|_{\L^2(\R^{2d},\MU)}.
\end{align*}
The first equality follows from definition \eqref{E:ZTOT} and the
second one from \eqref{E:OPM}. The subsequent inequality is true
because $\PP_\CR(\UPS_\tau)$ is closest to $\UPS_\tau$ in $\CR$ with
respect to $\WR$, by definition. Finally, we have used that
$(\XX,\XX+\tau\VV,\ORS)\#\MU \in \AR(\UPS_\tau, \GAMMA_\ORS)$. We will
write $\ORT_\tau = \PP_\CMU(\XX+\tau\VV)$. The map $\ORT_\tau$
is uniquely determined by
\begin{align}
  \int_{\R^{2d}} \langle (x+\tau\xi)-\ORT_\tau(x,\xi), 
    \ORT_\tau(x,\xi) \rangle \,\MU(dx,d\xi) &= 0,
\label{E:OPI1}\\
  \int_{\R^{2d}} \langle (x+\tau\xi)-\ORT_\tau(x,\xi),
    \ORS(x,\xi) \rangle \,\MU(dx,d\xi) &\LS 0
  \quad\text{for all $\ORS \in \CMU$.}
\label{E:OPI2}
\end{align}
We just need to combine Lemma~1.1 in \cite{Zarantonello1971} with
Remark~\ref{R:CONE}.
\end{remark}

\begin{remark}\label{R:EQUIV}
A map $\ORT \in \L^2(\R^{2d},\MU)$ is in $\CMU$ if and only if the
following statement is true: There exists a Borel set $N_\ORT \subset
\R^{2d}$ with $\MU(N_\ORT) = 0$ such that
\begin{equation}\label{E:EQMONC}
  \langle \ORT(x_1,\xi_1)-\ORT(x_2,\xi_2), x_1-x_2 \rangle \GS 0
  \quad\text{for all $(x_i,\xi_i) \in \R^{2d}\setminus N_\ORT$}
\end{equation}
with $i=1..2$. Indeed consider any $\ORT \in \CMU$ and $\GAMMA_\ORT
:= (\XX,\ORT)\#\MU \in \CR$. Then
\begin{gather*}
  N_\ORT := \Big\{ (x,\xi)\in\R^{2d} \colon \big( x,\ORT(x,\xi) \big) 
    \not\in \SPT\GAMMA_\ORT \Big\}
  \quad\text{satisfies}
\\
  \MU(N_\ORT) 
    = \MU\Big( (\XX,\ORT)^{-1} ( \R^{2d}\setminus\SPT\GAMMA_\ORT ) \Big)
    = \GAMMA_\ORT( \R^{2d}\setminus\SPT\GAMMA_\ORT ) 
    = 0
\end{gather*}
since $\GAMMA_\ORT$ is inner regular (being a finite Borel measure on
a locally compact Hausdorff space with countable basis; see
\cite{Folland1999}). As $\SPT\GAMMA_\ORT$ is monotone,
\eqref{E:EQMONC} follows.

Conversely, suppose $\ORT \in \L^2(\R^{2d},\MU)$ satisfies
\eqref{E:EQMONC}. Let $\GAMMA_\ORT := (\XX,\ORT)\#\MU$. For any
$(x_i,y_i) \in \SPT\GAMMA_\ORT$ with $i=1..2$ and any $\EPS>0$ we have
\begin{align*}
  0 &< \GAMMA_\ORT\big( B_\EPS(x_i)\times B_\EPS(y_i) \big)
\\
    &= \MU\bigg( \Big\{ (x,\xi)\in\R^{2d} \colon \big( x,\ORT(x,\xi) \big)
      \in B_\EPS(x_i)\times B_\EPS(y_i) \Big\} \bigg),
\end{align*}
by definition of support. Therefore there exist $(\hat{x}_i,\hat{\xi}_i)
\in \R^{2d}$ such that
\[
  \big( \hat{x}_i, \ORT(\hat{x}_i,\hat{\xi}_i) \big) 
    \in B_\EPS(x_i)\times B_\EPS(y_i)
  \quad\text{for $i=1..2$.}
\]
We may assume that $(\hat{x}_i, \hat{\xi}_i) \not\in N_\ORT$, where
$N_\ORT$ is the null set in \eqref{E:EQMONC}. Then
\[
  \langle y_1-y_2,x_1-x_2 \rangle
    \GS \langle \ORT(\hat{x}_1,\hat{\xi}_1)-\ORT(\hat{x}_2,\hat{\xi}_2), 
      \hat{x}_1-\hat{x}_2 \rangle - M\EPS,
\]
with $M := 4\max_{i=1..2}\{|x_i|,|y_i|\} + 2\EPS$. Since $(x_i,y_i)
\in \SPT\GAMMA_\ORT$ and $\EPS>0$ are arbitrary, we conclude that
$\SPT\GAMMA_\ORT$ is monotone, and therefore $\ORT \in \CMU$.
\end{remark}

The cone $\CMU$ is the set of all possible configurations, with a
reference configuration determined by $\MU \in \PR$. Configurations do
not permit any interpenetration of matter since the maps are monotone.
They do admit, however, the concentration of mass if the transport is
not strictly monotone. The fluid element at location/velocity $(x,\xi)
\in \R^d\times\R^d$ will never split because its final position is a
\emph{function} of $(x,\xi)$.

\begin{definition}[Tangent Cone]
Let $\RHO \in \SP_2(\R^d)$ and $\MU \in \PR$ be given. The tangent
cone of $\CMU$ at the configuration $\ORT \in \CMU$ is defined
as
\[
  \TAN_\ORT \CMU := \overline{ \Big\{ \ORV \in \L^2(\R^{2d},\MU) \colon
    \text{there exists $\EPS>0$ with $\ORT+\EPS\ORV \in \CMU$} \Big\}
      }^{\L^2(\R^{2d},\MU)}.
\]
%
\end{definition}

The set $\TAN_\ORT \CMU$ is a closed convex cone with vertex at the
origin (that is, the zero map) containing $\CMU-\ORT$. We refer the
reader to \textsection{2} in \cite{Zarantonello1971} for additional
information on tangent cones (also: support cones) to closed convex
sets in Hilbert spaces.

For any $\ORW \in \L^2(\R^{2d},\MU)$ there exists a unique metric
projection onto the closed convex cone $\TAN_\ORT \CMU$, which we will
denote by $\PP_{\TAN_\ORT \CMU}(\ORW)$. It can be characterized by the
following property: for all $\ORU \in \L^2(\R^{2d},\MU)$ we have
\begin{align*}
  & \qquad\qquad \ORU = \PP_{\TAN_\ORT \CMU}(\ORW)
    \quad\text{if and only if}
\\
  & \int_{\R^{2d}} \langle \ORW(x,\xi)-\ORU(x,\xi), 
      \ORU(x,\xi)-\ORV(x,\xi) \rangle \,\MU(dx,d\xi) \GS 0
    \quad\text{for all $\ORV \in \TAN_\ORT \CMU$.}
\end{align*}
Since $\TAN_\ORT \CMU$ is a cone, the latter condition is equivalent
to
\begin{align*}
  \int_{\R^{2d}} \langle \ORW(x,\xi)-\ORU(x,\xi), \ORU(x,\xi) \rangle
      \,\MU(dx,d\xi) &= 0
\\
  \int_{\R^{2d}} \langle \ORW(x,\xi)-\ORU(x,\xi), \ORV(x,\xi) \rangle
      \,\MU(dx,d\xi) &\LS 0
    \quad\text{for all $\ORV \in \TAN_\ORT \CMU$.}
\end{align*}
We also recall the following fact: for any $\ORT \in \CMU$ we have
\begin{equation}\label{E:INITIL}
  \|\ORU\|_{\L^2(\R^{2d},\MU)}^{-1}
    \Big\| \Big( \PP_\CMU(\ORT+\ORU) - \ORT \Big) 
      - \PP_{\TAN_\ORT \CMU}(\ORU) 
    \Big\|_{\L^2(\R^{2d},\MU)} \longrightarrow 0
\end{equation}
as $\ORU\rightarrow 0$ over any locally compact cone of increments;
see Lemma~4.6 of \cite{Zarantonello1971}. The metric projection
$\PP_{\TAN_\ORT \CMU}$ is therefore the differential of
$\PP_\CMU$ at $\ORT \in \CMU$.

Let us collect some additional
properties of the tangent cone.

\begin{proposition}[Tangent Cone]\label{P:TANG}
With $\ORT \in \CMU$ and $\GAMMA_\ORT := (\XX,\ORT)\#\MU$, assume
\begin{equation}
  \langle y_1-y_2, x_1-x_2 \rangle \GS \alpha |x_1-x_2|^2
  \quad\text{for all $(x_i,y_i) \in \SPT\GAMMA_\ORT$}
\label{E:STRICT}
\end{equation}
and $i=1..2$, where $\alpha>0$ is some constant. Then we have:
\begin{enumerate}
\renewcommand{\labelenumi}{(\roman{enumi}.)}
\item Every $\BU \in \L^2(\R^d,\RHO)$ is contained in $\TAN_\ORT
  \CMU$.
\item For all $\ORV \in \TAN_\ORT \CMU$ we also have $-\ORV \in
  \TAN_\ORT \CMU$.
\end{enumerate}
In particular, the tangent cone $\TAN_\ORT \CMU$ is a closed
subspace of $\L^2(\R^{2d},\MU)$.
\end{proposition}

\begin{proof}
We divide the proof into three steps.
\medskip

\textbf{Step~1.} Note first that any finite Borel measure $\nu$ on a
locally compact Hausdorff space $\Omega$ with countable base is inner
regular. Therefore the space of all continuous functions with compact
support is dense in $\L^2(\Omega,\nu)$. We refer the reader to
\cite{Folland1999} for further details. If $\Omega$ is also a vector
space, then the same statement is true for \emph{smooth} functions with
compact support. For every $\BU \in \L^2(\R^d,\RHO)$ there exists thus a
sequence of smooth functions $\BU^m$ with compact support with $\BU^m
\longrightarrow \BU$ strongly in $\L^2(\R^d,\RHO)$ and therefore in
$\L^2(\R^{2d},\MU)$. We claim that $\ORT+\EPS\BU^m \in \CMU$ for
$\EPS>0$ sufficiently small. To prove this, let $N_\ORT$ be the null
set of Remark~\ref{R:EQUIV}. Then
\begin{align*}
  & \big\langle \big( \ORT(x_1,\xi_1)+\EPS\BU^m(x_1) \big) 
    - \big( \ORT(x_2,\xi_2)+\EPS\BU^m(x_2) \big), x_1-x_2 \big\rangle
\\
  & \vphantom{\Big(} \qquad 
    = \langle \ORT(x_1,\xi_1)-\ORT(x_2,\xi_2), x_1-x_2 \rangle
      - \EPS \langle \BU^m(x_1)-\BU^m(x_2), x_1-x_2 \rangle
\\
  & \qquad
    \GS \Big( \alpha-\EPS \|D\BU^m\|_{\L^\infty(\R^d)} \Big) 
      |x_1-x_2|^2 
    \GS 0
\end{align*}
for all $(x_i,\xi_i) \in \R^{2d}\setminus N_\ORT$ with $i=1..2$, for
$\EPS>0$ small. This proves part (i.). 
\medskip

\textbf{Step~2.} We now show that for every $\ORS \in \CMU$ we
also have $-\ORS \in \TAN_\ORT \CMU$. The argument is a modification
of the proof of Proposition~4.28 of \cite{Gigli2004}. We first define
the plan $\GAMMA := (\XX,\ORS)\#\MU \in \CR$. Since $\SPT\GAMMA$ is a
monotone subset of $\R^d\times \R^d$, there exists a maximal monotone
extension of it, which we denote by $\Gamma$. Let $u$ be the
corresponding maximal monotone set-valued map, defined as
\[
  u(x) := \{ y\in\R^d \colon (x,y)\in\Gamma \}
  \quad\text{for all $x\in\R^d$.}
\]
It is well-known that for every $x\in\R^d$ the image $u(x)$ is a
closed and convex subset of $\R^d$; see \cite{AlbertiAmbrosio1999}.
Consider the disintegration of the transport plan
\[
  \GAMMA(dx,dy) =: \gamma_x(dy) \,\RHO(dx).
\]
Then we have that $\gamma_x = \ORS(x,\cdot)\#\mu_x$ for $\RHO$-a.e.\
$x\in\R^d$, with $\MU(dx,d\xi) = \mu_x(d\xi) \,\RHO(dx)$ the
disintegration of $\MU$. Let $\check{\gamma}_x := (-\ORS(x,
\cdot))\#\mu_x$ and $-\GAMMA = (\XX,-\ORS)\#\MU$ so that
\[
  (-\GAMMA)(dx,dy) = \check{\gamma}_x(dy) \,\RHO(dx).
\]

For $\RHO$-a.e.\ $x\in\R^d$ we denote by $A_x \subset u(x)$ the closed
convex hull of $\SPT\gamma_x$. For such an $x$ there are two
possibilities: either $\gamma_x$ is a Dirac measure and $A_x =
\{\BAR(\GAMMA)(x)\}$ (recall Definition~\ref{D:BAR}), or $A_x$ (and
therefore $u(x)$) contains $\BAR(\GAMMA)(x)$ as an interior point with
respect to the relative topology. In the latter case, the subspace
\[
  L_x := \bigcup_{n\in\N} n\big( -\BAR(\GAMMA)(x) + u(x) \big)
\]
has the property that $\SPT\gamma_x \subset \BAR(\GAMMA)(x) + L_x$,
and hence $\SPT\check{\gamma}_x \subset -\BAR(\GAMMA)(x) + L_x$. Let
$\PP^n_x$ be the metric projection of $\R^d$ onto the closed convex
set
\begin{equation}\label{E:RANGE}
  -(n+1)\BAR(\GAMMA)(x) + nu(x)
\end{equation}
for $\RHO$-a.e.\ $x\in\R^d$ and $n\in\N$. 
Since projections are contractions, we have
\begin{align*}
  |\PP^n_x(y)-y|
    & \LS \big|\PP^n_x(y) - \PP^n_x\big( -\BAR(\GAMMA)(x) \big)\big|
      + \big|\big( -\BAR(\GAMMA)(x) \big) - y \big|
\\
    & \LS 2 \big|y - \big( -\BAR(\GAMMA)(x) \big)\big|
\end{align*}
for all $y\in\R^d$. We used that $-\BAR(\GAMMA)(x)$ is contained in
\eqref{E:RANGE}. We have
\begin{align*}
  & \int_{\R^d} \bigg( \int_{\R^d} \big|y - \big( -\BAR(\GAMMA)(x) 
      \big)\big|^2 \,\check{\gamma}_x(dy) \bigg) \,\RHO(dx)
\\
  & \qquad
    \LS 4 \int_{\R^d} \bigg( \int_{\R^d} |y|^2 
      \,\check{\gamma}_x(dy) \bigg) \,\RHO(dx)
    = 4 \int_{\R^{2d}} |\ORS(x,\xi)|^2 \,\MU(dx,d\xi),
\end{align*}
by definition of $\BAR(\GAMMA)(x)$ and Jensen's inequality. We now
define the maps
\[
  \ORS^n(x,\xi) := \PP^n_x\big(-\ORS(x,\xi) \big) 
  \quad\text{for $\MU$-a.e.\ $(x,\xi) \in \R^{2d}$.}
\]
Using dominated convergence, we get for $n\rightarrow\infty$ that
\begin{equation}\label{E:CONV1}
  \|\ORS^n-(-\ORS)\|_{\L^2(\R^{2d},\MU)}^2 =
    \int_{\R^d} \bigg( \int_{\R^d} |\PP^n_x(y)-y|^2 
      \,\check{\gamma}_x(dy) \bigg) \,\RHO(dx)
    \longrightarrow 0,
\end{equation}
because $\PP^n_x(y) \longrightarrow y$ for all $y\in -\BAR(\GAMMA)(x)
+ L_x$ and for $\RHO$-a.e.\ $x\in\R^d$ such that $\gamma_x$ is not a
Dirac measure. We used again that $\PP^n_x(-\BAR(\GAMMA)(x)) =
-\BAR(\GAMMA)(x)$.

As discussed in Step~1, there exists a sequence of smooth, compactly
supported functions $\BT^m$ such that $\BT^m \longrightarrow
\BAR(\GAMMA)$ in $\L^2(\R^d,\RHO)$. We now define
\begin{equation}\label{E:SNM}
  \ORS^{n,m}(x,\xi) := \ORS^n(x,\xi) 
    + (n+1) \big( \BAR(\GAMMA)(x)-\BT^m(x) \big).
\end{equation}
for $\MU$-a.e.\ $(x,\xi) \in \R^{2d}$. We get for $m\rightarrow\infty$
(with $n$ fixed) that
\begin{equation}\label{E:CONV2}
  \|\ORS^{n,m}-\ORS^n\|_{\L^2(\R^{2d},\MU)}^2
    = (n+1)^2 \int_{\R^d} |\BAR(\GAMMA)(x)-\BT^m(x)|^2 \,\RHO(dx)
    \longrightarrow 0.
\end{equation}
Combining \eqref{E:CONV1} and \eqref{E:CONV2}, we find
$\|\ORS^{n,m}-(-\ORS)\|_{\L^2(\R^{2d},\MU)} \longrightarrow 0$. We
claim that $\ORT+\EPS\ORS^{n,m} \in \CMU$ for $\EPS>0$ small. To
prove this, we observe first that
\[
  \ORS^{n,m}(x,\xi) \subset -(n+1)\BT^m(x) + nu(x)
  \quad\text{for $\MU$-a.e.\ $(x,\xi)\in\R^{2d}$,}
\]
by definition of $\PP^n_x$ and \eqref{E:SNM}. With $N_\ORT$ the null
set of Remark~\ref{R:EQUIV}, we have
\begin{align*}
  & \big\langle \big( \ORT(x_1,\xi_1)-\EPS(n+1)\BT^m(x_1) \big)
    - \big( \ORT(x_2,\xi_2)-\EPS(n+1)\BT^m(x_2) \big), x_1-x_2 \big\rangle
\\
  & \vphantom{\Big(} \qquad
    = \langle \ORT(x_1,\xi_1)-\ORT(x_1,\xi_1), x_1-x_2 \rangle
      -\EPS (n+1) \langle \BT^m(x_1)-\BT^m(x_2), x_1-x_2 \rangle
\\
  & \qquad
    \GS \Big( \alpha - \EPS(n+1)\|D\BT^m\|_{\L^\infty(\R^d)} \Big) 
      |x_1-x_2|^2 
    \GS 0
\end{align*}
for all $(x_i,\xi_i) \in \R^{2d}\setminus N_\ORT$ with $i=1..2$, for
$\EPS>0$ small. Since $u$ is monotone, the support of $(\XX,
\ORT+\EPS\ORS^{n,m})\#\MU$ is contained in a monotone subset of
$\R^d\times\R^d$.
\medskip

\textbf{Step~3.} We prove that if $\ORV \in \TAN_\ORT \CMU$ then
also $-\ORV \in \TAN_\ORT \CMU$. There exists a sequence of $\ORV^n
\in \L^2(\R^{2d}, \MU)$ with $\|\ORV^n-\ORV\|_{\L^2(\R^{2d},\MU)}
\longrightarrow 0$ as $n\rightarrow \infty$, and such that
$\ORT+\EPS^n\ORV^n \in \CMU$ for $\EPS^n>0$ small. We have the
following identity:
\[
  -\ORV^n = -\frac{1}{\EPS^n} (\ORT+\EPS^n\ORV^n) + \frac{1}{\EPS^n} \ORT.
\]
The first term on the right-hand side is in $\TAN_\ORT \CMU$ because
of Step~2; the second one is in $\CMU \subset \TAN_\ORT \CMU$.
Since the tangent cone is a closed convex cone, we conclude that
$-\ORV^n \in \TAN_\ORT \CMU$. Then we use that
$\|(-\ORV^n)-(-\ORV)\|_{\L^2(\R^{2d},\MU)} \longrightarrow 0$.
\end{proof}

\begin{remark}
We emphasize that, unlike the tangent cone built from optimal
transport maps/plans, which basically consists of \emph{gradient}
vector fields (see \cites{AmbrosioGigliSavare2008, Gigli2004}), the
tangent cone derived from monotone maps contains all of
$\L^2(\R^d,\RHO)$ if $\ORT$ is strictly monotone in the sense of
inequality\eqref{E:STRICT}. This condition is satisfied when $\ORT =
\XX$, for example. Generically, it can happen that the tangent cone is
a proper subset of $\L^2(\R^d,\RHO)$: If $d=1$ and $\ORT \in \CMU$
depends only on the spatial variable $x\in\R$ (so that $\ORT \in
\L^2(\R,\RHO)$), then a velocity $\BV \in \L^2(\R,\RHO)$ belongs to
$\TAN_\ORT \CMU$ only if $\BV$ is non-decreasing on each open interval
on which $\ORT$ is constant; see Lemma~3.6 in
\cite{CavallettiSedjroWestdickenberg2014} for more details.
\end{remark}


\subsection{Minimization Problem}\label{SS:EM}

We now introduce the main minimization problem for \eqref{E:PGD}. Both
mass and momentum will be conserved, by construction. But since
transport maps $\ORT \in \CMU$ are not required to be strictly
monotone (hence injective), it may happen that fluid elements with
distinct velocities are transported to the same location. We will then
use the barycentric projection to select an admissible velocity that
are consistent with the monotonicity constraint. This results in fluid
elements sticking together to form larger compounds.

\begin{definition}
For any $\RHO \in \SP_2(\R^d)$ and $\MU \in \PR$, and any $\ORT \in
\CMU$ let
\[
  \H_\MU(\ORT) := \Big\{ \BU\circ\ORT \colon 
    \BU \in \L^2(\R^d,\RHO_\ORT) \Big\},
  \quad
  \RHO_\ORT := \ORT\#\MU.
\]
\end{definition}

One can check that $\H_\MU(\ORT)$ is a closed subspace of
$\L^2(\R^{2d}, \MU)$ because
\[
  \|\BU\circ\ORT\|_{\L^2(\R^{2d},\MU)} 
    = \|\BU\|_{\L^2(\R^d,\RHO_\ORT)}
\]
for all $\BU \circ \ORT \in \H_\MU(\ORT)$; see Section~5.2 in
\cite{AmbrosioGigliSavare2008}. Consequently, there exists an
orthogonal projection onto this subspace, which we will denote by
$\PP_{\H_\MU(\ORT)}$.

\begin{definition}[Energy Minimization]\label{D:EM}
Let $\RHO \in \SP_2(\R^d)$, $\MU \in \PR$, and $\tau>0$ be given.
Then we consider the following three-step scheme:
\begin{enumerate}
\item Compute the metric projection $\ORT_\tau :=
  \PP_{\CMU}(\XX+\tau\VV)$ and define
\begin{equation}
  \ORW_\tau(x,\xi) := V_\tau\big( x,\xi, \ORT_\tau(x,\xi) \big)
  \quad\text{for $\MU$-a.e.\ $(x,\xi) \in \R^{2d}$;}
\label{E:ZEROT2}
\end{equation}
see Remark~\ref{R:SIM} and \eqref{E:ZETOP} for the definition of
$V_\tau$.
\item Compute the orthogonal projection $\ORU_\tau :=
  \PP_{\H_\MU(\ORT_\tau)}(\ORW_\tau)$.
\item Define the updated fluid state
\[
  \RHO_\tau := \ORT_\tau\#\MU,
  \quad 
  \MU_\tau := (\ORT_\tau, \ORU_\tau)\#\MU.
\]
\end{enumerate}
\end{definition}

Notice that $\MU_\tau$ is well-defined for any choice of $(\RHO, \MU,
\tau)$, and that $\ORU_\tau$ determines an Eulerian velocity field
$\BU_\tau \in \L^2(\R^d, \RHO_\tau)$ via $\ORU_\tau =: \BU_\tau \circ
\ORT_\tau$. We observe that $\BU_\tau$ is just the barycentric
projection of $\MU_* := (\ORT_\tau, \ORW_\tau)\#\MU$: Indeed we have
\[
  \int_{\R^{2d}} \big| \ORW_\tau(x,\xi)-\BU\big( \ORT_\tau(x,\xi) \big)
      \big|^2 \,\MU(dx,d\xi) 
    = \int_{R^{2d}} |\zeta-\BU(z,\zeta)|^2 \,\MU_*(dz,d\zeta)
\]
for all $\BU \in \L^2(\R^d,\RHO_\tau)$, and the barycentric projection
$\BAR(\MU_*)$ is the unique element in $\L^2(\R^d,\RHO_\tau)$ closest
to $\MU_*$ with respect to $\W_{\RHO_\tau}$ (recall \eqref{E:HYBRID}).
From Proposition~\ref{P:TANG}, we deduce that $\L^2(\R^d, \RHO_\tau)
\subset \TAN_\XX \C_{\MU_*}$. Step~(2) of Definition~\ref{D:EM} can
therefore be interpreted as the projection of the updated state
$\MU_*$ onto (a subspace of) the tangent cone at the new
configuration; see also Remark~\ref{R:EXACTSL}. A similar combination
of transporting the vector field, then projecting it onto the tangent
cone was used in \cite{AmbrosioGigli2008} to construct the parallel
transport along curves in $\SP_2(\R^D)$; see also \cite{Brenier2009}.

When $\RHO_\tau$ is absolutely continuous with respect to the Lebesgue
measure so that there is no concentration (no sticking together of
fluid elements), then the tangent cone at $\RHO_\tau$ consists only of
monokinetic states, as follows from Remark~\ref{R:REMQ}.

\begin{remark}
We emphasize that the minimization of work links the transport
$\ORT_\tau$ to the intermediate velocity $\ORW_\tau$ through the
optimal velocity \eqref{E:ZEROT2}. This makes it possible to express
the work equivalently in different form: Defining
\[
  \ORV_\tau(x,\xi) := \frac{\ORT_\tau(x,\xi)-x}{\tau}
  \quad\text{for $\MU$-a.e.\ $(x,\xi) \in \R^{2d}$,}
\]
we have the following identities, which will be used frequently:
\begin{equation}
  (x+\tau\xi)-\ORT_\tau(x,\xi)
    = \tau \big( \xi-\ORV_\tau(x,\xi) \big)
    = \frac{2\tau}{3} \big( \xi-\ORW_\tau(x,\xi) \big).
\label{E:IDTIES}
\end{equation}
In particular, the transport velocity can be written as a convex
combination
\begin{equation}
  \ORV_\tau(x,\xi) = \frac{2}{3} \ORW_\tau(x,\xi) + \frac{1}{3} \xi
  \quad\Longleftrightarrow\quad
  \ORW_\tau(x,\xi) = \frac{3}{2} \ORV_\tau(x,\xi) - \frac{1}{2} \xi
\label{E:WTQQ}
\end{equation}
for $\MU$-a.e.\ $(x,\xi)\in\R^{2d}$. Inserting \eqref{E:IDTIES} into
the work functional \eqref{E:WOKF} with $z=\ORT(x,\xi)$ (recall that
projections are non-splitting; see Section~\ref{SS:RTM}), we observe
that the minimi\-zation over $\CMU$ can be reformulated as a
minimization over velocities.
\end{remark}

\begin{remark}\label{R:EXACTSL}
We will be mostly interested in situations where the initial state is
\emph{monokinetic}, for which $\MU(dx, d\xi) = \delta_{\BU(x)}(d\xi)
\,\RHO(dx)$ for some velocity $\BU \in \L^2(\R^d, \RHO)$. In this
case, the transport maps in $\CMU$ are functions of $x\in\R^d$ alone
because
\begin{equation}
  \ORT(x,\xi) := \ORT\big( x,\BU(x) \big)
  \quad\text{for $\MU$-a.e.\ $(x,\xi)\in\R^{2d}$.}
\label{E:REDUC}
\end{equation}

For $d=1$, it was shown in \cites{NatileSavare2009,
CavallettiSedjroWestdickenberg2014} that the family of transport maps
\[
  \ORT_s := \PP_{\CMU}(\XX+s\VV)
  \quad\text{for all $s\GS 0$,}
\]
determines a weak solution of the pressureless gas dynamics system
\eqref{E:PGD} with initial data $(\RHO,\BU)$ in the following way: The
density at time $s$ is given by the push-forward $\RHO_s :=
\ORT_s\#\MU$. The (Lagrangian) velocity is defined by the formula
\[
  \ORV_s(x,\xi) := \lim_{h\rightarrow 0+} \frac{\ORT_{s+h}(x,\xi)-\ORT_s(x,\xi)}{h} 
  \quad\text{for $\MU$-a.e.\ $(x,\xi)\in\R^{2d}$.}
\]
Again this is a function of $x\in\R^d$ alone because of
\eqref{E:REDUC}. Using \eqref{E:INITIL}, we observe that $\ORV_s$ is,
in fact, the metric projection of the \emph{initial velocity} $\BU$
onto the tangent cone $\TAN_{\ORT_s} \CMU$. Here we used that $\xi =
\BU(x)$ for $\MU$-a.e.\ $(x,\xi) \in\R^{2d}$. Moreover, since the map
$s \mapsto \ORT_s$ is Lipschitz continuous in $\L^2(\R^{2d},\MU)$ and
therefore differentiable for a.e.\ $s\in\R$, we conclude for such $s$
that $\ORV_s$ can also be obtained as
\[
  \ORV_s(x,\xi) = \lim_{h\rightarrow 0+} \frac{\ORT_{s-h}(x,\xi)-\ORT_s(x,\xi)}{-h} 
  \quad\text{for $\MU$-a.e.\ $(x,\xi)\in\R^{2d}$.}
\]
and so $\ORV_s \in -\TAN_{\ORT_s}\CMU$ as well. This implies that 
\[
  \ORV_s \in \H_\MU(\ORT_s)
  \quad\text{for a.e.\ $s\in\R$,}
\]
and, in fact, $\ORV_s$ is the orthogonal projection of the initial
velocity onto this subspace; see \cites{NatileSavare2009,
CavallettiSedjroWestdickenberg2014}. Our minimization problem
preserves this structure: Note that
\begin{align*}
  & \int_{\R^2} \big\langle \xi-\BU_\tau\big( \ORT_\tau(x,\xi) \big), 
    \BV\big( \ORT_\tau(x,\xi) \big) \big\rangle \,\MU(dx,d\xi)
\\
  & \qquad
      = \int_{\R^2} \big\langle \xi-\ORW_\tau(x,\xi), 
        \BV\big( \ORT_\tau(x,\xi) \big) \big\rangle \,\MU(dx,d\xi)
\end{align*}
for all $\BV \in \D(\R)$. Using the identity \eqref{E:WTQQ} for
$\ORW_\tau$, for any $\alpha>0$ we can write
\begin{align*}
  & \int_{\R^2} \big\langle \xi-\ORW_\tau(x,\xi), 
      \BV\big( \ORT_\tau(x,\xi) \big) \big\rangle \,\MU(dx,d\xi)
\\
  & \qquad
    = \frac{3}{2\tau} \int_{\R^2} \big\langle (x+\tau\xi)-\ORT_\tau(x,\xi),
      (\BV+\alpha\,\ID)\circ\ORT_\tau(x,\xi) \big\rangle \,\MU(dx,d\xi)
\\
  & \qquad\quad
    - \frac{3\alpha}{2\tau} \int_{\R^2} \langle (x+\tau\xi)-\ORT_\tau(x,\xi),
      \ORT_\tau(x,\xi) \rangle \,\MU(dx,d\xi).
\end{align*}
The last integral vanishes because of \eqref{E:OPI1} in
Remark~\ref{R:SIM}. Since for large enough $\alpha$ the map
$\BV+\alpha\,\ID$ is strictly increasing, one can check that
$(\ID,(\BV+\alpha\,\ID)\circ \ORT_\tau)\#\MU \in \CMU$. Here we used
the assumption that $d=1$ (since compositions of monotone maps are
again monotone in one space dimension). Using inequality
\eqref{E:OPI2}, we have
\[
  \int_{\R^2} \big\langle \xi-\ORW_\tau(x,\xi), 
    \BV\big( \ORT_\tau(x,\xi) \big) \big\rangle \,\MU(dx,d\xi) \LS 0
  \quad\text{for all $\BV \in \D(\R)$,}
\]
which in particular implies equality. Since $\D(\R)$ is dense in
$\L^2(\R,\RHO_\tau)$ we get
\[
  \int_{\R} \big\langle \BU(x)-\BU_\tau\big( \BT_\tau(x) \big),
    \BV\big( \BT_\tau(x) \big) \big\rangle \,\RHO(dx) = 0
  \quad\text{for all $\BV \in \L^2(\R,\RHO_\tau)$,}
\]
writing $\BT_\tau := \ORT_\tau(x,\BU(x))$ for $\RHO$-a.e.\ $x\in\R^d$.
As explained above, this orthogonality, in combination with the
definition of $\BT_\tau$ as the metric projection of $\ID+\tau\BU$
onto mono\-tone maps, characterizes solutions of \eqref{E:PGD}
satisfying a stickyness condition, for a.e.\ $\tau>0$. So our
discretization already generates the \emph{exact} solution, not just
an approximations, for $d=1$ and monokinetic initial state $\MU =
(\ID,\BU)\#\RHO$.
\end{remark}

\begin{remark}
In \cite{BressanNguyen2014} the authors prove the non-existence of
sticky particle solutions to \eqref{E:PGD} for well-designed initial
data. In their construction the number of collisions grows unboundedly
the closer one gets to the initial time, and so the dynamics has
arbitrarily small time scales. Using our discretization, we can
construct a sequence of approximate solutions to \eqref{E:PGD}
starting from the initial data in \cite{BressanNguyen2014}. We will
show below that this approximation converges to a measure-valued
solution of \eqref{E:PGD}. The timestep $\tau>0$ in our discretization
introduces a minimal time scale below which the dynamics is not
completely resolved but is ``smeared out.'' It would be interesting to
know to which solution our discretization converges in the limit
$\tau\rightarrow 0$.
\end{remark}

\begin{remark}\label{R:MOMMY}
The constant map $\ORS(x,\xi)=\pm b$ for all $(x, \xi) \in\R^{2d}$,
where $b\in\R^d$ is some vector, is an element of $\CMU$. Using this
function in \eqref{E:OPI2}, we get
\begin{align*}
  \int_{\R^{2d}} \langle b, \zeta \rangle \,\MU_\tau(dz,d\zeta)
    &= \int_{\R^{2d}} \langle b, \ORU_\tau(x,\xi) \rangle \,\MU(dx,d\xi)
\\
    &= \int_{\R^{2d}} \langle b, \ORW_\tau(x,\xi) \rangle \,\MU(dx,d\xi)
      = \int_{\R^{2d}} \langle b, \xi \rangle \,\MU(dx,d\xi);
\end{align*}
see \eqref{E:WTQQ}. Recall that $\ORU_\tau$ is the orthogonal
projection of $\ORW_\tau$ onto $\H_\MU(\ORT)$, which contains $\ORS$.
We conclude that the minimization preserves the total momentum.

Similarly, we can use the test functions $\ORS(x,\xi) = \pm Ax$ for
all $(x,\xi) \in \R^{2d}$ with $A \in \SKEW{d}$ in \eqref{E:OPI2}
because these functions are monotone. It follows that
\[
  \int_{\R^{2d}} \langle \ORW_\tau(x,\xi), Ax \rangle \,\MU(dx,d\xi)
    = \int_{\R^{2d}} \langle \xi,Ax \rangle \MU(dx,d\xi);
\]
recall \eqref{E:IDTIES}. We decompose the left-hand side using $x =
\ORT_\tau(x,\xi) - \tau \ORV_\tau(x,\xi)$. The corresponding second
integral can be estimated as
\begin{align*}
  & \bigg| -\tau \int_{\R^{2d}} \langle \ORW_\tau(x,\xi), A\ORV_\tau(x,\xi)
    \rangle \,\MU(dx,d\xi) \bigg|
 \LS \tau \|A\| \|\ORW_\tau\|^{1/2}_{\L^2(\R^{2d},\MU))}
    \|\ORV_\tau\|^{1/2}_{\L^2(\R^{2d},\MU))}.
\end{align*}
We will see below that both $\L^2(\R^{2d},\MU)$-norms can be estimated
against the kinetic energy $\int_{\R^{2d}} |\xi|^2 \,\MU(dx,d\xi)$ of
the initial state, uniformly in $\tau$. Moreover, we get
\begin{align*}
  \int_{\R^{2d}} \langle \ORW_\tau(x,\xi), A\ORT_\tau(x,\xi) \rangle \,\MU(dx,d\xi)
    & = \int_{\R^{2d}} \langle \zeta,Az\rangle \,\MU_*(dz,d\zeta)
\\
    & = \int_{\R^d} \langle \BU_\tau(z), Az \rangle \,\RHO_\tau(dz),
\end{align*}
where we have used that $\BU_\tau$ is the barycentric projection of
$\MU_*$. We conclude that our minimization preserves total angular
momentum up to order $\tau$.
\end{remark}


\subsection{Polar Cone}\label{SS:PC}

In this section, we will give a representation of the elements in the
polar cone of $\CMU$. As we will see later, such elements appear as
stress tensors. In Remark~\ref{R:MOMOS}, we have defined the space
$\C_*(\R^d;\R^D)$ of all continuous functions $f \colon \R^d
\longrightarrow \R^D$ for which $\lim_{|x|\rightarrow\infty} f(x) \in
\R^D$ exists. We identify $\C_*(\R^d;\R^D)$ with the space
$\C(\dot\R^d; \R^D)$ of continuous functions on the one-point
compactification $\dot\R^d$ of $\R^d$: We adjoin to $\R^d$ a point
$\infty$ and define a distance (see \cite{Mandelkern1989})
\[
  d(x,y) := \begin{cases}
     \min\{ |x-y|, h(x)+h(y) \}
       & \text{if $x,y \in \R^d$,}
\\
     h(x) 
       & \text{if $x\in\R^d$ and $y=\infty$,}
\\
     0
       & \text{if $x,y=\infty$,}
   \end{cases}
\]
where $h(x) := 1/(1+|x|)$ for all $x\in\R^d$. Then $|x|\rightarrow
\infty$ is equivalent to $d(x,\infty) \rightarrow 0$. To any
$g\in\C_*(\R^d;\R^D)$ we associate $\dot g \in \C(\dot\R^d; \R^D)$
defined as
\[
  \dot g(x) := \begin{cases}
      g(x) 
        & \text{if $x\in\R^d$,}
\\
      \lim_{|x|\rightarrow\infty} g(x)
        & \text{if $x=\infty$.}
    \end{cases}
\]
Conversely, the restriction of any function in $\C(\dot\R^d; \R^D)$
to $\R^d$ induces a function in $\C_*(\R^d; \R^D)$. We will hence not
distinguish between the two spaces. Similarly, we define $\C_*(\R^d;
\MAT{l})$, $\C_*(\R^d; \SYM{l})$, and $\C_*(\R^d; \SYM[\GS]{l})$.

For any $u\in\C^1(\R^d;\R^d)$ we refer to the symmetric part $\nabla
u(x)^\S$ for all $x\in\R^d$ as its deformation tensor, which is an
element of $\C(\R^d;\SYM{d})$. Let
\begin{gather*}
  \C^1_*(\R^d;\R^d) := \big\{ u\in\C^1(\R^d;\R^d) \colon
    \text{$\nabla u \in \C_*\big( \R^d;\MAT{d} \big)$} \big\},
\\
  \MON(\R^d) := \big\{ u\in\C^1_*(\R^d;\R^d) \colon
    \text{$u$ is monotone} \big\}.
\end{gather*}
The cone $\MON(\R^d)$ contains, in particular, all linear maps
$u(x):=Ax$ for all $x\in \R^d$, with $A\in\MAT[\GS]{d}$. We will use
the following result from \cite{CavallettiWestdickenberg2014}:

\begin{theorem}[Stress Tensor]\label{T:STRESS}
Assume that there exist a measure $\F \in \M(\R^d;\R^d)$ with finite
first moment and a measure $\PK \in \M(\R^d;\SYM[\GS]{d})$ with
\begin{equation}
\label{E:EULA}
  G(u) := -\int_{\R^d} \langle u(x), \F(dx) \rangle
    - \int_{\R^d} \TRACE\big( \nabla u(x) \PK(dx) \big) \GS 0
\end{equation}
for all $u \in \MON(\R^d)$. Then there exists $\RES \in \M(\dot\R^d;
\SYM[\GS]{d})$ such that
\begin{gather}
  G(u) = \int_{\dot\R^d} \TRACE\big( \nabla u(x) \RES(dx) \big)
  \quad\text{for all $u \in \C^1_*(\R^d;\R^d)$,}
\nonumber\\
  \int_{\dot\R^d} \TRACE\big( \RES(dx) \big) 
    = -\int_{\R^d} \langle x, \F(dx) \rangle
      -\int_{\R^d} \TRACE\big( \PK(dx) \big).
\label{E:CONTROL2}
\end{gather}
\end{theorem}

Notice that the integral in \eqref{E:EULA} is finite for any choice of
$u\in \C^1_*(\R^d;\R^d)$ since the first moment of $\F$ is finite, by
assumption. Recall that the trace of a symmetric matrix equals the sum
of its eigenvalues, which in the case of a positive semidefinite
matrix are all non-negative. Therefore \eqref{E:CONTROL2} controls the
size of $\RES$.

\begin{remark}\label{R:FINI}
The stress tensor $\RES$ does not actually assign any mass to the
remainder $\dot\R^d \setminus \R^d$, so Theorem~\ref{T:STRESS}
remains true if the compactification $\dot\R^d$ is replaced by
$\R^d$. In fact, recall that $\R^d$ (being a separable metric space)
is a Radon space, so that any finite Borel measure is inner regular.
Consider a non-negative, radially symmetric test function $\varphi \in
\D(\R^d)$ with $\int_{\R^d} \varphi(x) \,dx = 1$ and define
\[
  u_R := \nabla(\phi_R \star \varphi),
  \quad\text{with}\quad
  \phi_R(x) := \HA\max\{ |x|^2-R^2, 0\}
\]
for $x\in\R^d$ and $R>0$. The map $\phi_R \star \varphi$ is convex and
smooth (since the convolution preserves convexity), hence $u_R$ is
monotone and smooth. Notice that $u_R(x) = 0$ for all $|x| \LS R-c$,
with $c$ the (finite) diameter of $\SPT\varphi$. Moreover, we have
\[
  \int_{\R^d} \varphi(x-y) |y|^2 \,dy
    = |x|^2 + \bigg( \int_{\R^d} |z|^2 \varphi(z) \,dz \bigg)
\]
for all $x\in\R^d$, which implies that $u_R(x) = x$ and $Du_R(x) =
\ONE$ for $|x| \GS R+c$. In particular, we observe that $u_R \in
\MON(\R^d)$ for all $R>0$. Then
\begin{equation}
  \int_{|x|\GS R+c} \TRACE\big( \RES(dx) \big)
    \LS C \bigg( \int_{|x|\GS R-c} |x| \,|\F(dx)|
      + \int_{|x|\GS R-c} \TRACE\big( \PK(dx) \big) \bigg),
\label{E:ABSCH}
\end{equation}
with $C$ some finite constant depending on $\varphi$. The right-hand
side of \eqref{E:ABSCH} converges to zero as $R\rightarrow\infty$
since both measures $|\F|$ and $\TRACE(\PK)$ are inner regular and
the first moment of $\F$ is finite. We conclude that $\TRACE(\RES)
(\dot\R^d \setminus \R^d) = 0$.
\end{remark}

For $\MU \in \PR$ and $\tau>0$ let $\ORT_\tau$ be given by
Definition~\ref{D:EM}. Then
\[
  -\frac{3}{2\tau^2} \int_{\R^{2d}} \langle (x+\tau\xi)-\ORT_\tau(x,\xi), 
    \ORS(x,\xi) \rangle \,\MU(dx,d\xi) \GS 0
  \quad\text{for all $\ORS \in \CMU$,}
\]
which is \eqref{E:OPI2}. In particular, this inequality is true for
$\ORS = u$ with $u \in \MON(\R^d)$. Functions in $\MON(\R^d)$ have at
most linear growth and are therefore in $\L^2(\R^d,\RHO)$. Applying
Theorem~\ref{T:STRESS} (with $\PK \equiv 0$), we get $\RES_\tau
\in \M(\R^d; \SYM[\GS]{d})$ with
\begin{align}
  \int_{\R^d} \TRACE\big( \nabla u(x) \RES_\tau(dx) \big)
    &= -\frac{3}{2\tau^2} \int_{\R^{2d}} \langle (x+\tau\xi)
      -\ORT_\tau(x,\xi), u(x) \rangle \,\MU(dx,d\xi),
\label{E:STW1}\\
  \int_{\R^d} \TRACE\big( \RES_\tau(dx) \big) 
    &= -\frac{3}{2\tau^2} \int_{\R^{2d}} \langle (x+\tau\xi)
      -\ORT_\tau(x,\xi), x \rangle \,\MU(dx,d\xi);
\label{E:STW2}
\end{align}
see Remark~\ref{R:FINI}. The representation in Theorem~\ref{T:STRESS}
generalizes a similar description of the polar cone of monotone
maps obtained in \cite{NatileSavare2009} in one space dimension.
Using the identity \eqref{E:WTQQ}, we obtain the following identities:
\begin{align*}
  \int_{\R^d} \TRACE\big( \nabla u(x) \RES_\tau(dx) \big)
    &= -\frac{3}{2\tau} \int_{\R^{2d}}\langle \xi
      -\ORV_\tau(x,\xi), u(x) \rangle \,\MU(dx,d\xi)
\\
    &= -\frac{1}{\tau} \int_{\R^{2d}} \langle \xi
      -\ORW_\tau(x,\xi), u(x) \rangle \,\MU(dx,d\xi),
\end{align*}
with transport velocity $\ORV_\tau(x,\xi) := (\ORT_\tau(x,\xi)-x)
/\tau$ for $\MU$-a.e.\ $(x,\xi) \in \R^{2d}$, and
\begin{align*}
  \int_{\R^d} \TRACE\big( \RES_\tau(dx) \big) 
    &= -\frac{3}{2\tau} \int_{\R^{2d}} \langle \xi
      -\ORV_\tau(x,\xi), x \rangle \,\MU(dx,d\xi)
\\
    &= -\frac{1}{\tau} \int_{\R^{2d}} \langle \xi
      -\ORW_\tau(x,\xi), x \rangle \,\MU(dx,d\xi).
\end{align*}
%

\begin{remark}
In order to explore the significance of $\RES_\tau$, we consider
\[
  \MU(dx,d\xi) = \OF \delta_0(d\xi) \,\LEB^1 |_{(-1,1)}(dx)  
    + \OH \delta_1(d\xi) \,\delta_0(dx).
\]
For any $\tau>0$ the support of the transport plan $(\XX,
\XX+\tau\VV)\#\MU$ is not monotone. Then $\GAMMA_\tau := (\XX,
\ORT_\tau)\#\MU$, with $\ORT_\tau$ given by Definition~\ref{D:EM}, can
be computed as
\begin{align*}
  \GAMMA_\tau(dx,dy)
    &= \OF \Big( \delta_{\beta(\tau)\tau}(dy) 
      \,\LEB^1 |_{[0,\beta(\tau)\tau]}(dx) + \delta_x(dy) 
      \,\LEB^1 |_{(-1,1)\setminus[0,\beta(\tau)\tau]}(dx) \Big)
\\
    &\vphantom{\Big(}\quad
      + \OH \delta_{\beta(\tau)\tau}(dy) \,\delta_0(dx),
\end{align*}
where $\beta(\tau) \in [0,1]$ is the minimizer of the following
function:
\[
  \varphi_\tau(\beta) := \OH |1-\beta|^2 \tau^2 
    + \OF \int_0^{\beta\tau} |\beta\tau-x|^2 \,dx,
\]
which represents the $\L^2(\R^{2d},\MU)$-distance of $\XX+\tau\VV$ to
some map in $\CMU$ parameterized by $\beta$. One can check that
$\beta(\tau) := \frac{2}{\tau} (\sqrt{1+\tau}-1)$ and
$\beta(\tau)\longrightarrow 1$ as $\tau\rightarrow 0$. The induced
velocity distribution $\MU_\tau := (\XX,(\YY-\XX)/\tau)\# \GAMMA_\tau$
equals
\begin{align*}
  \MU_\tau(dx,d\xi) 
    &= \OF \Big( \delta_{\beta(\tau)-x/\tau}(d\xi) 
      \,\LEB^1 |_{[0,\beta(\tau)\tau]}(dx) + \delta_0(d\xi) 
        \,\LEB^1 |_{(-1,1)\setminus[0,\beta(\tau)\tau]}(dx) \Big)
\\
    &\vphantom{\Big(}\quad
      + \OH \delta_{\beta(\tau)}(d\xi) \,\delta_0(dx).
\end{align*}
%
%
The first $\xi$-moments of $\MU$ and $\MU_\tau$ determine the corresponding
momenta:
\begin{align*}
  \BM(dx) &\vphantom{\Big(}
    := \OH \delta_0(dx),
\\
  \BM_\tau(dx) &:= \OF \Big( \beta(\tau)-\frac{x}{\tau} \Big) 
     \,\LEB^1 |_{[0,\beta(\tau)\tau]}(dx) + \OH \beta(\tau) 
       \,\delta_0(dx).
\end{align*}
Therefore the change in momentum (which represents an acceleration) has
two parts: The velocity of the fluid element with mass $1/2$ located
at $x=0$ decreases, so the momentum is getting smaller. This momentum
is \emph{transfered} to fluid elements in the interval $[0, \beta(\tau)
\tau]$, which pick up speed. The transfer is described by the
derivative of the non-negative measure from Theorem~\ref{T:STRESS}. Let
$\RES_\tau := R_\tau\LEB^1$ with
\[
  R_\tau(x) := \begin{cases}
      \OH \big( 1-\beta(\tau) \big) 
          -\OF \Big( \beta(\tau)x-\DST\frac{x^2}{2\tau} \Big)
        & \text{if $x \in [0,\beta(\tau)\tau]$,}
\\
      \vphantom{\Big(} 0 
        & \text{otherwise.}
    \end{cases}
\]
Since $\OH(1-\beta(\tau)) = \OE\beta(\tau)^2\tau$ the measure
$\RES_\tau$ is non-negative, supported in $[0,\beta(\tau)\tau]$, and
it satisfies $\BM-\BM_\tau = \partial_x\RES_\tau$ in $\D'(\R)$. Note
further that $R_\tau$ vanishes as $\tau\rightarrow 0$, in any
$\L^p(\R)$ with $1\LS p<\infty$. Theorem~\ref{T:STRESS} suggests that a
similar structure can be found in higher space dimensions: the metric
projection onto $\CR$ may cause the transfer of momentum to
neighboring fluid elements, captured by the distributional divergence
$\nabla\cdot\RES_\tau$ of the stress tensor field $\RES_\tau$.
This transfer manifests itself also in the kinetic energy balance,
which we will consider next.
\end{remark}

\begin{proposition}[Energy Balance]\label{P:EB}
For any $(\RHO,\MU,\tau)$ as in Definition~\ref{D:EM} consider the
quantities $(\ORT_\tau, \ORW_\tau, \ORU_\tau, \MU_\tau)$ specified
there. Let $\RES_\tau \in \M(\R^d; \SYM[\GS]{d})$ be the stress tensor
field satisfying \eqref{E:STW1}/\eqref{E:STW2}. Then we have
\begin{align*}
  \E[\MU_\tau]
    & + \int_{\R^{2d}} \Big( {\TST\frac{1}{6}} |\ORW_\tau-\xi|^2
      + \HA |\ORU_\tau-\ORW_\tau|^2 \Big) \,\MU(dx,d\xi)
\\
    & + \int_{\R^d} \TRACE\big( \RES_\tau(dx) \big)
      = \E[\MU],
\end{align*}
with total/kinetic energy
\[
  \E[\MU] := \int_{\R^{2d}} \OH |\xi|^2 \MU(dx,d\xi).
\]
Recall that $\MU_\tau = (\ID,\BU_\tau)\#\RHO_\tau$ for some $\BU_\tau
\in \L^2(\R^d,\RHO_\tau)$ with $\RHO_\tau = \ORT_\tau\#\MU$.
\end{proposition}

\begin{proof}
Since $\ORU_\tau := \PP_{\H_\MU(\ORT_\tau)}(\ORW_\tau)$
(orthogonal projection), we have
\begin{align*}
  & \int_{\R^{2d}} \HA |\ORU_\tau(x,\xi)|^2 \,\MU(dx,d\xi)
    + \int_{\R^{2d}} \HA |\ORU_\tau(x,\xi)-\ORW_\tau(x,\xi)|^2 \,\MU(dx,d\xi)
\\
  & \qquad
    = \int_{\R^{2d}} \HA |\ORW_\tau(x,\xi)|^2 \,\MU(dx,d\xi).
\end{align*}
On the other hand, using definition \eqref{E:ZEROT2} of $\ORW_\tau$
and \eqref{E:IDTIES} we can write
\begin{align}
    \int_{\R^{2d}} \HA |\ORW_\tau(x,\xi)|^2 \,\MU(dx,d\xi)
  & + \frac{1}{6} \int_{\R^{2d}} |\ORW_\tau(x,\xi)-\xi|^2 
      \,\MU(dx,d\xi)
\nonumber\\
    = \int_{\R^{2d}} \HA |\xi|^2 \,\MU(dx,d\xi)
  & - \frac{3}{2\tau^2} \int_{\R^{2d}} \langle (x+\tau\xi)-\ORT_\tau(x,\xi),
      \ORT_\tau(x,\xi) \rangle \,\MU(dx,d\xi)
\nonumber\\
  & + \frac{3}{2\tau^2} \int_{\R^{2d}} \langle (x+\tau\xi)-\ORT_\tau(x,\xi),
      x \rangle \,\MU(dx,d\xi)
\label{E:PO2}
\end{align}
The second integral on the right-hand side of \eqref{E:PO2} vanishes
because of \eqref{E:OPI1}, the last one can be expressed in terms of
the stress tensor field $\RES_\tau$; see \eqref{E:STW2}.
\end{proof}


\section{Energy Minimization: Polytropic Gases}\label{S:VTD3}

We now modify the minimization problem of Section~\ref{SS:EM} for
polytropic gases. In this case, the density $\RHO$ must be absolutely
continuous with respect to the Lebesgue measure since otherwise the
internal energy would be infinite (see Definition~\ref{D:INT2}). We
need a lower semicontinuity result for the internal energy, suitably
redefined as a convex functional on the set of monotone transports.


\subsection{Gradient Young Measures}

We introduce Young measures to capture oscillations and concentrations
of weak* converging sequences of derivatives of functions of bounded
variations. They will be used in Section~\ref{SS:IE} to establish a
lower semicontinuity result for the internal energy. We follow the
presentation of \cites{KristensenRindler2010, Rindler2011}.

Let $\Omega \subset \R^d$ be a bounded Lipschitz domain and $\BT \in
\BVS(\Omega; \R^d)$. Let $B_d$ be the open unit ball in $\MAT{d}$ and
$\partial B_d$ its boundary. We associate to the derivative $D\BT$
(which is a measure) a triple $\upsilon = (\nu,\sigma,\mu)$ with
\begin{equation}
\label{E:DUSP}
  \nu \in \L^\infty_\mathrm{w}\Big( \Omega;\SP\big( \MAT{d} \big) \Big), 
  \quad
  \sigma \in \M_+(\bar{\Omega}),
  \quad
  \mu \in \L^\infty_\mathrm{w}\big( \bar{\Omega},\sigma; 
    \SP(\partial B_d) \big)
\end{equation}
as follows: Consider the Lebesgue-Radon-Nikod\'{y}m decomposition
\begin{equation}
\label{E:LRN}
  D\BT = \nabla\BT \,\LEB^d + D^s \BT,
  \quad
  D^s \BT \perp \LEB^d,
\end{equation}
and define $\nu_x := \delta_{\nabla\BT(x)}$ for a.e.\ $x\in\Omega$ and
$\sigma := |D^s \BT|$. Let further
\[
  D^s \BT = \frac{\DD D^s \BT}{\DD |D^s \BT|} \,|D^s \BT|,
  \quad
  p:=\frac{\DD D^s \BT}{\DD |D^s \BT|} \in \L^1(\Omega,|D^s \BT|;\partial B_d).
\]
be the polar decomposition of $D^s \BT$ and define $\mu_x =
\delta_{p(x)}$ for $|D^s \BT|$-a.e.\ $x\in\Omega$. Here
$\L^\infty_\mathrm{w}(\Omega;\SP(\MAT{d}))$ is the space of weakly
measurable maps from $\Omega$ into the space of probability measures
on $\MAT{d}$ (similar definition for $\L^\infty_\mathrm{w}
(\bar{\Omega}, \sigma; \SP(\partial B_d))$). We call $\upsilon =
(\nu,\sigma,\mu)$ an elementary Young measure associated to $Du$.

Consider now a sequence of uniformly bounded maps $\BT^k \in
\BVS(\Omega;\R^d)$. Extracting a subsequence, we may assume that
$\BT^k \longrightarrow \BT$ in $\L^1(\Omega;\R^d)$ and $D\BT^k \WEAK
D\BT$ weak* in $\M(\Omega; \MAT{d})$, for some $\BT \in
\BVS(\Omega;\R^d)$. In this case, we say that $\BT^k$ converges weak*
to $\BT$ in $\BVS(\Omega;\R^d)$. We denote by $\upsilon^k = (\nu^k,
\sigma^k, \mu^k)$ the elementary Young measure associated to $D\BT^k$
as above. Since the spaces in \eqref{E:DUSP} are contained in the dual
spaces to $\L^1(\Omega;\C_0(\MAT{d}))$, $\C(\bar{\Omega})$, and
$\L^1(\bar{\Omega},\sigma; \C(\partial B_d))$ respectively, one can
show that there exists a subsequence (which we do not relabel, for
simplicity) and a triple $\upsilon = (\nu,\sigma, \mu)$ as in
\eqref{E:DUSP} with the property that the
\begin{align}
  \llbracket f,\upsilon^k \rrbracket
    & := \int_\Omega \lbrack f(x,\cdot), \nu^k_x \rbrack \,dx
      + \int_{\bar{\Omega}} \lbrack f^\infty(x,\cdot), \mu^k_x \rbrack
        \,\sigma^k(dx)
\label{E:LANGLE}\\
    & := \int_\Omega \int_{\MAT{d}} f(x,M) \,\nu^k_x(dM) \,dx
      + \int_{\bar{\Omega}} \int_{\partial B_d} f^\infty(x,M) \,\mu^k_x(dM)
        \,\sigma^k(dx)
\nonumber
\end{align}
converge to $\llbracket f,\upsilon \rrbracket$ (defined analogously) as $k
\rightarrow \infty$, for $f \in \mathscr{R}(\Omega;\MAT{d})$ with
\[
  \mathscr{R}\big( \Omega;\MAT{d} \big) := \left\{
    \begin{aligned}
      f \colon & \bar{\Omega}\times\MAT{d} \longrightarrow \R \;\colon
\\
      & \begin{minipage}{45ex} 
          the map $f$ is a Carath\'{e}odory function with 
\\
         linear growth at infinity, and there exists 
\\
         $f^\infty \in \C\big( \bar{\Omega}\times\MAT{d} \big)$
        \end{minipage} 
    \end{aligned} \right\};
\]
see Corollary~2 and Proposition~2 in \cite{KristensenRindler2010}.
Recall that the map $f\colon \bar{\Omega}\longrightarrow\R^d$ is
called a Carath\'{e}odory function if it is $\LEB^d\times
\B(\MAT{d})$-measurable and if $M\mapsto f(x,M)$ is continuous for
a.e.\ $x\in\bar{\Omega}$. It is enough to check the measurability of
$x \mapsto f(x,M)$ for all $M\in\MAT{d}$ fixed; see Proposition~5.6 in
\cite{AmbrosioFuscoPallara2000}. The map $f$ has linear growth at
infinity if there exists $c\GS 0$ such that $|f(x,M)| \LS c(1+\|M\|)$
for a.e.\ $x\in\bar{\Omega}$ and all $M\in\MAT{d}$. We denote by
$f^\infty$ the recession function of $f$, defined as
\begin{equation}
\label{E:DEFREC}
  f^\infty(x,M) := \lim_{\substack{x' \rightarrow x \\ M' \rightarrow M \\
    t \rightarrow \infty}} \frac{f(x',tM')}{t}
  \quad\text{for a.e.\ $x\in\bar{\Omega}$ and all $M\in\MAT{d}$.}
\end{equation}
Note that the recession function is positively $1$-homogeneous in $M$,
if it exists. We call a triple $\upsilon = (\nu,\sigma,\mu)$ obtained
as a limit as above a gradient Young measure and denote the space of
gradient Young measures by $\mathscr{G}(\Omega;\MAT{d})$. Then
\[
  D\BT = \lbrack \ID,\nu \rbrack \,\LEB^d + \lbrack \ID,\mu \rbrack \,\sigma,
\]
by construction (cf.\ \eqref{E:LANGLE}). Moreover, we have
\[
  \|\nabla\BT^k\| \;\LEB^d + \bigg\| \frac{\DD D^s\BT^k}{\DD |D^s\BT^k|} \bigg\|
      \; |D^s\BT^k| 
    \WEAK \lbrack \|\cdot\|,\nu \rbrack \,\LEB^d 
      + \lbrack \|\cdot\|,\mu \rbrack \,\sigma
\]
weak* in $\M(\bar{\Omega})$ as $k\rightarrow\infty$, which implies
that $\lbrack \|\cdot\|,\nu \rbrack \in \L^1(\Omega)$. We used the
fact that the recession function of $f(x,M) := \varphi(x)\|M\|$ with
$\varphi\in \C(\bar{\Omega})$ coincides with $f$. We refer the reader
to \cite{KristensenRindler2010} for further information. 
\medskip

We apply this framework to sequences of \emph{monotone} functions
$\BT^k \in \BVS(\Omega;\R^d)$ (see Remark~\ref{R:FUNC}), in which case
the derivatives $D\BT^k$ are positive (that is, matrix-valued and
locally finite) measures; see Theorem~5.3 in
\cite{AlbertiAmbrosio1999}. Since the map $(M,v) \mapsto v\cdot(Mv)$
is continuous, the set $\MAT[>]{d}$ is open and convex; the set
$\MAT[\GS]{d}$ is a closed convex cone. Recall that a matrix $M$ is an
element of $\MAT[\GS]{d}$ (resp.\ $\MAT[>]{d}$) if and only if its
symmetric part $M^\S \in \SYM[\GS]{d}$ (resp.\ $\SYM[>]{d}$).

\DETAIL{ 
In fact, for all $A\in\MAT{d}$ we have the following identities:
\begin{align*}
  2 \langle v,Av\rangle 
    &= \langle v,Av\rangle + \langle A^\T v,v\rangle
    = \langle v,Av\rangle + \langle v,A^\T v\rangle
    = \langle v,(A+A^\T)v\rangle,
\\
  0 &= \langle v,Av\rangle - \langle A^\T v,v\rangle
    = \langle v,Av\rangle - \langle v,A^\T v\rangle
    = \langle v,(A-A^\T)v\rangle.
\end{align*}
} 

\begin{proposition}[Gradient Young Measures]\label{P:GYMP}
Let $\Omega \subset \R^d$ be a bounded Lipschitz domain and suppose
that $\BT^k \WEAK \BT$ weak* in $\BVS(\Omega;\R^d)$ with $\BT^k,\BT \in
\BVS(\Omega;\R^d)$ monotone. For all $k\in\N$ we denote by
$\upsilon^k$ the elementary gradient Young measure associated to
$D\BT^k$, as introduced above. Then there exists a subsequence (which we
do not relabel, for simplicity) and a gradient Young measure $\upsilon
\in \mathscr{G}(\Omega;\MAT{d})$ with the property that $\llbracket
f,\upsilon^k \rrbracket \longrightarrow \llbracket f,\upsilon \rrbracket$
for all $f \in \RP(\Omega;\MAT{d})$, where
\[
  \RP\big( \Omega;\MAT{d} \big) := \left\{
    \begin{aligned}
      f \colon & \bar{\Omega}\times\MAT{d} \longrightarrow \R \;\colon
\\
      & \begin{minipage}{45ex} 
          the map $f$ is a Carath\'{e}odory function with 
\\
         linear growth at infinity, and there exists 
\\
         $f^\infty \in \C\big( \bar{\Omega}\times\MAT[\GS]{d} \big)$
        \end{minipage} 
    \end{aligned} \right\}.
\]
\end{proposition}

\begin{proof}
It suffices to check continuity of the recession function $f^\infty$
on the smaller set $\MAT[\GS]{d}$ because all gradient Young measures
considered above vanish outside of $\Omega\times\MAT[\GS]{d}$. Indeed,
consider any test function $f\in\mathscr{R}(\Omega; \MAT{d})$ of the
form $f(x,M) = \varphi(x) h(M)$, with $\varphi \in \CC(\Omega)$
non-negative, $h(M) := \DIST(M,\MAT[\GS]{d})$ for all $M\in\MAT{d}$.
Then the map $h$ is positively $1$-homogeneous. This follows
immediately from the fact that $\MAT[\GS]{d}$ is a cone. It can also
be derived from the following observation: Notice first that symmetric
and antisymmetric matrices in $\MAT{d}$ are orthogonal to each other
with respect to the Frobenius inner product. For given $M\in\MAT{d}$
let $M^\S=QP$ be a polar decomposition of its symmetric part (so that
$Q^\T Q=\ONE$ and $P=P^\T\GS 0$). Then
\[
  X_M := M^\A + (M^\S+P)/2
\]
is the unique element in $\MAT[\GS]{d}$ closest to $M$ in the Frobenius
norm, and
\[
  \DIST(M,\MAT[\GS]{d})^2 = \sum_{\lambda_i(M^\S)<0} \lambda_i(M^\S)^2,
\]
with $\lambda_i(M^\S)$ the (real) eigenvalues of $M^\S$; see
\cite{Higham1988}. Then the claim follows.

Since $h$ is positively $1$-homogeneous it is sufficient to consider
the limits $x'\rightarrow x$ and $M'\rightarrow M$ in \eqref{E:DEFREC}
to define the recession function of $f$. But $\varphi, h$ are
continuous, and hence $f^\infty$ coincides with $f$. In particular, this
proves that $f \in \mathscr{R}(\Omega; \MAT{d})$. Note that $h(M)=0$ if
and only if $M \in \MAT[\GS]{d}$. If $\BT^k$ is monotone, then
\begin{gather*}
  \nabla \BT^k(x) \in \MAT[\GS]{d}
  \quad\text{and}\quad
  p^k(x) \in \MAT[\GS]{d}
\end{gather*}
for a.e.\ $x\in\Omega$ and $|D^s\BT^k|$-a.e.\ $x\in\Omega$,
respectively, where
\[
  D^s\BT^k = \frac{\DD D^s\BT^k}{\DD |D^s\BT^k|} \,|D^s\BT^k|,
  \quad
  p^k := \frac{\DD D^s\BT^k}{\DD |D^s\BT^k|} 
    \in \L^1(\Omega,|D^s\BT^k|;\partial B_d)
\]
is the polar decomposition of $D^s\BT^k$. If $\upsilon^k$ is the
elementary gradient Young measure associated to $D\BT^k$, then
$\llbracket f,\upsilon^k \rrbracket = 0$ for all $k\in\N$ and $f$ as
above. Then the gradient Young measure $\upsilon$ generated by
$\{\upsilon^k\}_k$ satisfies $\llbracket f,\upsilon\rrbracket = 0$
because $\llbracket f,\upsilon^k \rrbracket \longrightarrow \llbracket
f,\upsilon \rrbracket$ for all $f \in \mathscr{R}(\Omega;\MAT{d})$.
Since $\varphi \in \CC(\Omega)$ non-negative was arbitrary, we get that
the gradient Young measure $\upsilon$ vanishes outside of $\Omega
\times \MAT[\GS]{d}$. A careful inspection of the proof of
Proposition~2 in \cite{KristensenRindler2010} now yields the result:
Convergence of the gradient Young measures follows from the weak*
convergence of
\[
	\nu^k \,\LEB^d + \mu^k \,\sigma^k \WEAK \nu \,\LEB^d + \mu \,\sigma
\]
on (a suitable compactification of) $\Omega\times\MAT{d}$, which
reduces to weak* convergence on $\Omega\times\MAT[\GS]{d}$ whenever
$\BT^k$ and $\BT$ are monotone.
\end{proof}



\subsection{Internal Energy}\label{SS:IE}

We introduce a functional on the space of monotone $\BVS$-vector
fields that represents the internal energy. This functional will be
convex and lower semicontinuous with respect to weak* convergence in
$\BVS_\LOC(\Omega;\R^d)$.

Let us start with two auxiliary results. 

\begin{lemma}\label{L:HA}
For any $\gamma>1$, the map $h\colon \MAT{d} \longrightarrow
[0,\infty]$ defined by
\begin{equation}
\label{E:HA}
  h(M) := \begin{cases}
      \det(M^\S)^{1-\gamma} & \text{if $M \in \MAT[>]{d}$,}
\\ 
      \infty & \text{otherwise},
    \end{cases}
\end{equation}
is lower semicontinuous, proper, and convex. For all $M\in\MAT{d}$, we
have
\begin{equation}
\label{E:RECF}
  h^\infty(M) 
    := \lim_{t\rightarrow\infty} \frac{h(\ONE+tM)-h(\ONE)}{t} 
    = \begin{cases}
      0 
        & \text{if $M \in \MAT[\GS]{d}$,}
\\
      \infty 
        & \text{otherwise.}
    \end{cases}
\end{equation}
\end{lemma}

\begin{proof}
Since $M \mapsto \det(M^\S)$ is continuous, the function $h$ is lower
semicontinuous. It is proper because $h(\ONE) = 1$. In order to prove
the convexity of $h$, we observe that $S \mapsto \det(S)^{1/d}$ is
concave for all symmetric, positive definite $S\in\MAT{d}$. Indeed,
pick any two such matrices $S^0$ and $S^1$. For all $t\in[0,1]$ we can
write
\[
  \det\big((1-t)S^0+t S^1\big)^{1/d} 
    = \big( \det(S^0) \det(\ONE+tB) \big)^{1/d},
\]
where $B := C^{-1}(S^1-S^0)C^{-1}$ and $C:=\sqrt{S^0}$. The matrix $C$
exists and is invertible since $S^0$ is symmetric and positive
definite, by assumption. Then we compute
\begin{gather}
  \frac{d}{dt} \det(\ONE+tB)^{1/d}
    = \det(\ONE+tB)^{1/d}
      \bigg\{ \frac{1}{d} \TRACE(D) \bigg\},
\nonumber\\
  \frac{d^2}{dt^2} \det(\ONE+tB)^{1/d}
    = \det(\ONE+tB)^{1/d}
      \bigg\{ \frac{1}{d^2} \TRACE(D)^2 - \frac{1}{d} \TRACE(D^2) \bigg\},
\label{E:HES}
\end{gather}
where $D := B(\ONE+tB)^{-1}$. The matrix $D$ is symmetric.
Therefore
\[
  \TRACE(D)^2 
    = (\lambda_1+\cdots+\lambda_d)^2
    \LS d (\lambda_1^2+\cdots+\lambda_d^2) 
    = d \TRACE(D^2),
\]
where $\lambda_1, \ldots, \lambda_d$ are the real eigenvalues of $D$.
Hence \eqref{E:HES} is non-positive for every $s\in[0,1]$. The
composition of a concave function with a convex, non-increasing map is
convex. Therefore the map $S \mapsto \det(S)^{1-\gamma}$ is convex for
all symmetric, positive definite $S \in \MAT{d}$. Finally, the
composition of any convex function with the linear map $M \mapsto
M^\S$ is again convex. Then the result follows.

To prove \eqref{E:RECF}, we use that the map $t\mapsto (h(\ONE +
tM)-h(\ONE))/t$ is non-decreasing (hence $\lim_{t\rightarrow\infty} =
\sup_{t>0}$), by convexity of $h$. If $M \not\in \MAT[\GS]{d}$, then
there exists $v \in \R^d$, $\|v\| = 1$, such that $\langle v,Mv\rangle
< 0$. For sufficiently large $t>0$, we get
\[
  \langle v,(\ONE+tM)v \rangle = 1 + t \langle v,Mv \rangle < 0,
\]
and thus $h(\ONE+tM) = \infty$. This proves \eqref{E:RECF} for
$M \not\in \MAT[\GS]{d}$.

If $M \in \MAT[\GS]{d}$, then $\ONE+tM \in \MAT[>]{d}$ for all
$t>0$, because
\[
  \langle v,(\ONE+tM)v\rangle = 1 + t\langle v,Mv\rangle \GS 1
\]
for all $v\in\R^d$, $\|v\| = 1$. By convavity of $M \mapsto
\det(M^\S)^{1/d}$, we obtain
\begin{align*}
  \det(\ONE+tM^\S)^{1/d} 
    & = \det\Bigg( (1+t) \bigg( \frac{1}{1+t}\ONE 
      + \frac{t}{1+t} M^\S \bigg) \Bigg)^{1/d}
\\
    & \GS (1+t) \bigg( \frac{1}{1+t} \det(\ONE)^{1/d} 
      + \frac{t}{1+t} \det(M^\S)^{1/d} \bigg)
    \GS 1
\end{align*}
for $t>0$. Notice that $\det(M^\S) \GS 0$. This implies $\det(\ONE +
tM^\S)^{1-\gamma} \LS 1$ (recall that $\gamma>1$, by assumption), and
so \eqref{E:RECF} follows for $M\in\MAT[\GS]{d}$ as well.
\end{proof}

\begin{lemma}\label{L:INFCONF}
For any $n\in\N$, we define the $\inf$-convolution
\begin{equation}
\label{E:PHIN}
  h_n(M) := \inf_{B \in \MAT{d}} \Big\{ n\|M-B\| + h(B) \Big\}
\end{equation}
for all $M\in\MAT{d}$, which has the following properties:
\begin{enumerate}
\item The map $h_n$ is lower semicontinuous, proper, and convex.
\item For all $M\in\MAT{d}$, we have $h_n(M) \longrightarrow h(M)$
  monotonically from below.
\item The map $h_n$ is Lipschitz continuous with Lipschitz constant
  $n$.
\item The map $h_n$ has linear growth at infinity:
\begin{equation}
\label{E:LINGRO}
  h_n(M) \LS 1+n\sqrt{d} + n\|M\| 
  \quad\text{for all $M\in\MAT{d}$.}
\end{equation}
\item For all $M \in \MAT{d}$, we have that 
\begin{equation}
\label{E:HNI}
  h_n^\infty(M) 
    := \lim_{t\rightarrow\infty} \frac{h_n(\ONE+tM)-h_n(\ONE)}{t}
    = n \, \DIST(M,\MAT[\GS]{d}).
\end{equation}
\end{enumerate}
\end{lemma}

\begin{proof}
Statement~(1) follows from Corollary~9.2.2 in \cite{Rockafellar1997}:
Notice first that the norm and $h$ are lower semicontinuous, convex,
and proper; see Lemma~\ref{L:HA}. The recession function of the norm
is the norm itself, and it holds
\[
  n\|M\| + h^\infty(-M) > 0
  \quad\text{for all $M\in\MAT{d}, M\neq 0$.}
\]

Statements~(2) and (3) follow from Lemma~1.61 in
\cite{AmbrosioFuscoPallara2000}.

To prove \eqref{E:LINGRO}, we just choose $B = \ONE$ in
\eqref{E:PHIN} and use the triangle inequality.

Finally, statement~(5) follows from Corollary~9.2.1 in
\cite{Rockafellar1997}. We must prove that for all pairs of matrices
$M_1,M_2\in\MAT{d}$ with the property that
\begin{equation}
\label{E:ROCKCON}
  n\|M_1\| + h^\infty(M_2) \LS 0
  \quad\text{and}\quad
  n\|-M_1\| + h^\infty(-M_2) > 0,
\end{equation}
it holds $M_1+M_2\neq 0$. The first condition in \eqref{E:ROCKCON} is
only satisfied if 
\[
  M_1=0
  \quad\text{and}\quad
  M_2 \in \MAT[\GS]{d},
\]
because of \eqref{E:RECF}. Then the second condition requires $-M_2
\not\in \MAT[\GS]{d}$, and thus there exists $v\in\R^d$, $v\neq 0$, with
$\langle v,-M_2v\rangle < 0$ (consistent with $M_2 \in \MAT[\GS]{d}$).
This is only possible if $M_2 \neq 0$, so the claim follows. We then
obtain that the recession function \eqref{E:HNI} of the
$\inf$-convolution \eqref{E:PHIN} is given by
\[
  h_n^\infty(M) = \inf_{B \in \MAT{d}} \Big\{ n\|M-B\| + h^\infty(B) \Big\}
\]
for all $M \in \MAT{d}$, which implies the result because of
\eqref{E:RECF}.
\end{proof}

We can now prove the following lower semicontinuity result.

\begin{proposition}[Internal Energy]\label{P:LSCCON2}
Let $\Omega \subset \R^d$ be open and convex, and $h$ given by
\eqref{E:HA}. For $U \in \L^1(\Omega)$ non-negative and $\BT \in
\BVS_\LOC(\Omega; \R^d)$ we define
\begin{equation}
\label{E:DEFU}
  \INT[\BT] := \begin{cases}
      \DST \int_\Omega U(x) h\big( \nabla\BT(x) \big) \,dx
        & \text{if $\BT$ monotone,}
\\
      +\infty 
        & \text{otherwise,}
  \end{cases}
\end{equation}
using again the decomposition \eqref{E:LRN}. Then the following is
true:
\begin{enumerate}
\item The functional $\INT$ is convex.
\item For any $\BT^k \WEAK \BT$ weak* in $\BVS_\LOC(\Omega;\R^d)$ with
  $\BT^k,\BT \in \BVS_\LOC(\Omega;\R^d)$ monotone, there exists a
  subsequence (not relabeled) such that
\[
  \INT[\BT] \LS \liminf_{k\rightarrow\infty} \INT[\BT^k].
\]
\end{enumerate}
\end{proposition}

\begin{remark}
Notice that in \eqref{E:DEFU} we only consider the part of $D\BT$ that
is absolutely continuous with respect to $\LEB^d$ and disregard the
singular component. The intuition is that (for each direction) only
increasing jumps are allowed in the transport map $\BT$, which
correspond to the formation of vacuum, which is admissible.
\end{remark}

\begin{proof}[Proof of Proposition~\ref{P:LSCCON2}]
We proceed in two steps.
\medskip

\textbf{Step~1.} Consider $\BT^k \in \BVS_\LOC(\Omega;\R^d)$ with
$k=0..1$. For any $s\in(0,1)$ we define $\BT^s := (1-s) \BT^0 + s \BT^1 \in
\BVS_\LOC(\Omega;\R^d)$. If $\INT[\BT^k] = +\infty$ for $k=0$ or $k=1$,
then there is nothing to prove, so we may assume that both terms are
finite. This requires that both $\BT^k$ are monotone and $\nabla\BT^k(x)
\in \MAT[>]{d}$ for $U\LEB^d$-a.e.\ $x\in\Omega$. It follows that $\BT^s$
is monotone as well and $\nabla\BT^s(x) \in \MAT[>]{d}$ for
$U\LEB^d$-a.e.\ $x\in\Omega$. Then
\[
  h\big( \nabla\BT^s(x) \big) 
    \LS (1-s) h\big( \nabla\BT^0(x) \big) + s h\big( \nabla\BT^1(x) \big);
\]
see Lemma~\ref{L:HA}. Multiplying by $U(x)$ and integrating in
$\Omega$, we obtain
\[
  \INT[\BT^s] \LS (1-s) \INT[\BT^0] + s \INT[\BT^1]
\]
for all $s\in[0,1]$. This proves the convexity of the functional.
\medskip

\textbf{Step~2.} We first introduce a sequence of bounded open convex
sets
\[
  \Omega_n := \Big\{ x\in B_n(0) \colon 
    \DIST(x,\R^d\setminus\Omega)>1/n \Big\},
\]
which are bounded Lipschitz domains. We have $\Omega_{n-1} \subset
\Omega_n$ for all $n \in \N$.
\DETAIL{ 
To prove the convexity of $\Omega_n$ choose points $x_0,x_1 \in
\Omega_n$ and $\lambda \in [0,1]$, and define $x := (1-\lambda)x_0 +
\lambda x_1$. Let $\delta := \min\{ \DIST(x_i, \R^d\setminus\Omega)
\colon i=0..1 \}$, which is larger than $1/n$. By convexity of
$\Omega$, any point in the set $[x_0,x_1] + B_\delta(0)$ is contained
in $\Omega$. In particular, we have $B_\delta(x) \subset \Omega$.
Therefore $\DIST(x,\R^d\setminus \Omega) > 1/n$. \\[-1ex]
} 

We then choose a sequence of cut-off functions $\varphi_n \in
\CC(\Omega;[0,1])$ with $\varphi_n(x) = 1$ for all $x\in\Omega_{n-1}$
and $\varphi_n(x) = 0$ for all $x\not\in\Omega_n$. For all $n\in\N$ we
define
\begin{equation}
\label{E:DEFFN}
  f_n(x,M) := \big( U(x)\wedge n \big) \varphi_n(x) h_n(M)
  \quad\text{for all $(x,M) \in \Omega \times \MAT{d}$,}
\end{equation}
where $h_n$ is given by \eqref{E:PHIN}. Because of
Lemma~\ref{L:INFCONF}, the map $f_n$ is a Carath\'{e}odory function
with linear growth at infinity. In fact, we can estimate
\[
  0 \LS f_n(x,M) \LS n(1+n\sqrt{d} + n\|M\|)
  \quad\text{for all $(x,M) \in \Omega \times \MAT{d}$.}
\]
We prove that $f_n^\infty(x,M) = 0$ for all $(x,M) \in
\Omega\times\MAT[\GS]{d}$: Note first that
\[
  \bigg| \frac{f_n(x',tM')}{t}-0 \bigg| 
     \LS n h_n(tM')/t
     \LS n \Big\{ h_n(tM)/t + n\|M'-M\| \Big\},
\]
uniformly in $x'\in\Omega$. Recall that $h_n$ is Lipschitz continuous
with Lipschitz constant $n$. Since $h_n(M) < \infty$ for all $M \in
\MAT{d}$, by Theorem~8.5 in \cite{Rockafellar1997} we have
\[
  h_n(tM)/t \longrightarrow h_n^\infty(M)
  \quad\text{as $t\rightarrow \infty$,}
\]
which vanishes for $M\in\MAT[\GS]{d}$; see Lemma~\ref{L:INFCONF}. Then
$f_n^\infty \in \C(\Omega\times\MAT[\GS]{d})$, and so $f_n \in
\RP(\Omega_n; \MAT{d})$ for all $n\in\N$. By construction, it holds
\begin{equation}
\label{E:SUPFN}
  f_n(x,M) \LS f_{n+1}(x,M)
  \quad\text{and}\quad
  U(x)h(M) = \sup_n f_n(x,M)
\end{equation}
for all $(x,M) \in \Omega \times \MAT{d}$. We used again
Lemma~\ref{L:INFCONF}.

Let us fix $n\in\N$ for a moment. Extracting a subsequence if
necessary, we may assume that the sequence of elementary gradient
Young measures $\upsilon^k$ generated by $D\BT^k|\Omega_n$ converges to
$\upsilon=(\nu,\sigma,\mu) \in \mathscr{G}(\Omega_n;\MAT{d})$ in the
sense that
\begin{equation}
\label{E:CONVO}
  \llbracket f,\upsilon^k \rrbracket \longrightarrow \llbracket f,\upsilon \rrbracket
  \quad\text{for all $f\in\RP\big( \Omega_n;\MAT{d} \big)$;}
\end{equation}
see Proposition~\ref{P:GYMP}. It holds $D\BT = \lbrack \ID,\nu \rbrack
\,\LEB^d + \lbrack \ID,\mu \rbrack \,\sigma$. Comparing this identity
with the Lebesgue-Radon-Nikod\'{y}m decomposition \eqref{E:LRN}, we
find
\[
  \nabla\BT = \lbrack \ID,\nu\rbrack 
    + \lbrack \ID,\mu\rbrack \frac{\DD \sigma}{\DD \LEB^d}
  \quad\text{a.e.}\quad\text{and}\quad
  D^s u = \lbrack \ID,\mu\rbrack \,\sigma^s,
\]
where $\sigma^s \perp \LEB^d$ is the singular part of $\sigma$. Note
that $\lbrack \ID,\mu_x \rbrack \in \MAT{d}$ may not have unit length for
$\sigma^s$-a.e.\ $x\in\Omega$. The polar decomposition of $D^s u$ is
given by
\[
  |D^s u| = |\lbrack\ID,\mu\rbrack| \,\sigma^s
  \quad\text{and}\quad
  \frac{\DD D^s u}{\DD |D^s u|} 
    = \frac{\lbrack\ID,\mu\rbrack}{|\lbrack\ID,\mu\rbrack|}
  \quad\text{$|D^s u|$-a.e.}
\]

We now apply the convergence \eqref{E:CONVO} to the function $f_n$
defined in \eqref{E:DEFFN}, whose restriction to $\Omega_n$ belongs to
$\RP(\Omega_n; \MAT{d})$. We observe that
\[
  f_n^\infty(\cdot,\lbrack\ID,\mu\rbrack) \,\sigma^s
    = f_n^\infty\bigg( \cdot,\frac{\DD D^s u}{\DD |D^s u|} \bigg) \,|D^s u|
\]
because the map $M \mapsto f_n^\infty(x,M)$ is positively
$1$-homogeneous for $x\in\Omega_n$. Then the following Jensen-type
inequalities hold (see Theorem~9 in \cite{KristensenRindler2010}):
\begin{gather*}
  f_n(\cdot, \nabla u) 
    \LS \lbrack f_n,\nu \rbrack
    + \lbrack f_n^\infty,\mu\rbrack \frac{\DD\sigma}{\DD\LEB^d} 
  \quad\text{a.e.,}
\\
  \vphantom{\frac{1}{2}}
  f_n^\infty(\cdot,\lbrack \ID,\mu\rbrack) 
    \LS \lbrack f_n^\infty,\mu\rbrack
  \quad\text{$\sigma^s$-a.e.}
\end{gather*}
because the map $M \mapsto f_n(x,M)$ is convex for $x\in\Omega$. We
can then estimate
\begin{align*}
  & \lim_{k\rightarrow\infty} \int_{\Omega_n} f_n(\cdot,\nabla u^k) 
    + \int_{\Omega_n} f_n^\infty\bigg( \cdot,\frac{\DD D^s u^k}{\DD |D^s u^k|} 
      \bigg) \,|D^s u^k|
\\
  & \qquad
    = \int_{\Omega_n} \lbrack f_n,\nu\rbrack
    + \int_{\Omega_n} \lbrack f_n^\infty,\mu\rbrack \,\sigma
\\
  & \qquad
    = \int_{\Omega_n} \bigg( \lbrack f_n,\nu\rbrack
      + \lbrack f_n^\infty,\mu\rbrack \frac{\DD \sigma}{\DD \LEB^d} \bigg)
    + \int_{\Omega_n} \lbrack f_n^\infty,\mu\rbrack \,\sigma^s
\\
  & \qquad 
    \GS \int_{\Omega_n} f_n(\cdot,\nabla u) 
    + \int_{\Omega_n} f_n^\infty\bigg( \cdot,\frac{\DD D^s u}{\DD |D^s u|} 
      \bigg) \,|D^s u|.
\end{align*}
Clearly the integrals can be extended to all of $\Omega$ because $f_n$
vanishes outside of $\Omega_n$. Moreover, we have shown that the
recession function $f_n^\infty$ vanishes. Hence
\begin{equation}
\label{E:LSCN}
  \int_{\Omega} f_n(x,\nabla u(x)) \,dx 
    \LS \liminf_{k\rightarrow\infty} \int_{\Omega} U(x) h(\nabla u^k(x)) \,dx,
\end{equation}
where we also used \eqref{E:SUPFN}. By a standard diagonal argument
(successively extracting subsequences if necessary), we may assume
that \eqref{E:LSCN} holds for all $n\in\N$. We then use
\eqref{E:SUPFN} and the monotone convergence theorem to obtain the
result.
\end{proof}

We finish the section with an estimate on determinants of square
matrices.

\begin{lemma}\label{L:DETS}
Suppose $S$ is a real, positive semidefinite, symmetric $(d\times
d)$-matrix. For any real skew-symmetric $(d\times d)$-matrix $A$ we
have
\[
  \det(S+A) \GS \det S \GS 0.
\]
\end{lemma}

\begin{proof}
We divide the proof into two steps.
\medskip

\textbf{Step~1.} We will first prove that if $\det S = 0$, then
$\det(S+A) \GS 0$. Recall that the determinants of square matrices
equal the product of their eigenvalues. Non-real eigenvalues of $S+A$
can only occur in complex conjugate pairs because $S$, $A$ are real
matrices. Since the product of two complex conjugate numbers is
non-negative, it remains to prove that every \emph{real} eigenvalue of
$S+A$ must be non-negative. Let $\lambda \in \R$ be an eigenvalue with
corresponding eigenvector $v$. Note that if $v$ is complex, then its
complex conjugate $\bar{v}$ is another eigenvector to the same
eigenvalue $\lambda$. Taking the sum $v+\bar{v}$ if necessary, we may
therefore assume that $v \in \R^d$. We have
\[
  (S+A)v = \lambda v,
  \quad
  \|v\| > 0.
\]
We take the inner product with $v$ and obtain (since $A$ is
skew-symmetric)
\[
  \lambda \|v\|^2 = \langle (S+A)v, v\rangle
    = \langle Sv, v\rangle.
\]
The right-hand side is non-negative because $S$ is positive
semidefinite. Hence $\lambda \GS 0$. From this, we conclude that
$\det(S+A) \GS \det S$ whenever $\det S = 0$. 
\medskip

\textbf{Step~2.} Consider now $\det S \neq 0$. Since $S$ is positive
semidefinite and symmetric, all eigenvalues of $S$ (which are real)
are positive. Therefore $\det S > 0$ and $\langle Sv, v\rangle > 0$
for every $v\in\R^d$ with $v\neq 0$. We claim that $\det(S+tA) > 0$
for every $t\in\R$. In fact, assume this is false. Then zero is an
eigenvalue of $S+tA$, with corresponding eigenvector $v \in \R^d$ (see
above). We have $(S+tA)v = 0$ and $v \neq 0$. We get
\[
  0 < \langle Sv, v\rangle = \langle (S+tA)v, v\rangle = 0,
\]
using again that $A$ is skew-symmetric. This contradiction proves the
claim.

For all $t\in\R$, we can now define $f(t) := \log\det(S+tA)$. We
compute
\[
  f'(t) = \TRACE\big( (S+tA)^{-1} A \big).
\]
Notice that $t(S+tA)^{-1}A = \ONE - (S+tA)^{-1}S$. Since $S$ is
symmetric, there exists an orthogonal matrix $Q$ such that $Q^{-1}SQ =
\Lambda$, where $\Lambda := \mathrm{diag}(\lambda_1,\ldots,\lambda_d)$
contains the eigenvalues $\lambda_i>0$ of $S$, $i=1\ldots d$. Let
$e_i$ denote the $i$th standard basis vector of $\R^d$. Since the
trace is invariant under changes of basis, we obtain
\begin{align*}
  t \TRACE\big( (S+tA)^{-1} A \big) 
    &= \TRACE\big( \ONE - Q^{-1}(S+tA)^{-1}SQ \big)
\\
    &= \sum_{i=1}^d \big( 1 - \langle Q^{-1}(S+tA)^{-1}SQ e_i,
      e_i \rangle \big).
\end{align*}
We denote by $v_i$ the $i$th column vector of $Q$ (hence $v_i =
Qe_i$), which is a normalized eigenvector of $S$ corresponding to the
eigenvalue $\lambda_i$. As $Q^{-1}=Q^\T$, we have
\begin{equation}
  t \TRACE\big( (S+tA)^{-1} A \big) 
    = \sum_{i=1}^d ( 1 - \lambda_i \langle w_i, v_i \rangle ),
  \quad
  w_i := (S+tA)^{-1}v_i.
\label{E:TRV}
\end{equation}
Since the eigenvectors $v_1, \ldots, v_d$ form an orthonormal basis of
$\R^d$, there is a unique expansion $w_i = \sum_{k=1}^d \alpha_i^k
v_k$ with $\alpha_i^k := \langle w_i,v_k \rangle$. Using this
expansion, we get
\begin{equation}
  \alpha_i^i 
    = \langle w_i, (S+tA)w_i \rangle
    = \langle w_i, Sw_i \rangle
    = \sum_{k=1}^d \lambda_k (\alpha_i^k)^2
\label{E:GHT}
\end{equation}
for $i=1\ldots d$. Recall that the eigenvalues $\lambda_k$ are all
positive and $A$ is skew-symmetric. We conclude that $\alpha_i^i \GS
0$. Moreover, rewriting \eqref{E:GHT} in the form
\[
  \alpha_i^i (1-\lambda_i \alpha_i^i) 
    = \sum_{k\neq i} \lambda_k(\alpha_i^k)^2 \GS 0,
\]
we obtain that $1-\lambda_i\langle w_i,v_i \rangle \GS 0$ for each $i$.
Using this estimate in \eqref{E:TRV}, we conclude that $f'(t) \GS 0$
for all $t>0$, and so the map $t\mapsto f(t)$ is non-decreasing for
such $t$. In particular, we have that $\det(S+A) = \exp f(1) \GS \exp
f(0) = \det S > 0$.
\end{proof}

\begin{remark}\label{R:DISSC}
Lemma~\ref{L:DETS} can be made more precise if $\det(S)>0$: We first write
\begin{align*}
  \det(S+A)^\beta-\det(S)^\beta
    &= \int_0^1 \frac{d}{dt} \det(S+tA)^\beta \,dt
\\
    &= \beta \int_0^1 \det(S+tA)^\beta \, \TRACE\big( (S+tA)^{-1} A \big) \,dt
\end{align*}
for $\beta\in\R$. Notice that all terms are well-defined, and the
integrand on the right-hand side is non-negative. Since $S$ is positive
definite and symmetric, we can compute its root, which is the unique
$R\in\SYM[>]{d}$ such that $R^2=S$. Then
\begin{align*}
  \det(S+tA) &= \det(S) \det(\ONE+tC),
\\
  \TRACE\big( (S+tA)^{-1}A \big) &= \TRACE\big( (\ONE+tC)^{-1}C \big),
\end{align*}
with $C := R^{-1}AR^{-1}$ skew-symmetric. We obtain the following
identity:
\[
  \bigg( \frac{\det(S+A)}{\det(S)} \bigg)^\beta-1
    = \beta \int_0^1 \det(\ONE+tC)^\beta \, 
      \TRACE\big( (\ONE+tC)^{-1} C \big) \,dt,
\]
where the integral on the right-hand side is non-negative.
\end{remark}


\subsection{Minimization Problem}\label{SS:MP}

We now introduce the main minimization problem for \eqref{E:FULL}. We
represent the state of the fluid by $(\RHO,\MU,\sigma)$, with $\RHO
\in \SP_2(\R^d)$ the density, $\MU \in \PR$ the velocity distribution,
and $\sigma \in \M_+(\R^d)$ the thermodynamic entropy. We assume that
$\INT[\RHO,\sigma] < \infty$, which implies that $\RHO = r\LEB^d$ and
$\sigma = \RHO S$ for suitable Borel functions $r$, $S$; see
Definition~\ref{D:INT}. In the isentropic case, $S$ will be constant
in time and space. We want to minimize the sum of the internal energy
of the transported fluid and the acceleration cost of the transport,
over the cone $C_\MU$ of monotone maps; see Definition~\ref{D:CONFIG}.
The following observation will be useful:

\begin{lemma}\label{L:PROP}
Let $\RHO \in \SP_2(\R^d)$ and $\MU \in \PR$, where $\RHO \ll \LEB^d$.
To every $\ORT \in C_\MU$ we can associate a function $\BT \in
\L^2(\R^d,\RHO)$ defined on all of $\R^d$ that satisfies
\begin{equation}
  \ORT(x,\xi) = \BT(x)
  \quad\text{for $\MU$-a.e.\ $(x,\xi) \in \R^{2d}$.}
\label{E:COINC}
\end{equation}
The map $\BT$ is monotone on $\Omega := \INTR\CCONV\SPT\RHO$ (hence
$\BT \in \BVS_\LOC(\Omega;\R^d)$) so that
\[
  \langle \BT(x_1)-\BT(x_2), x_1-x_2 \rangle \GS 0
  \quad\text{for all $x_1,x_2 \in \Omega$.}
\]
\end{lemma}

\begin{proof}
For $\ORT \in \CMU$ let $u$ be any maximal monotone map associated
to $\GAMMA := (\XX,\ORT)\#\MU$, which is in $C_\RHO$; see
Definition~\ref{D:ASSO}. As shown in Lemma~\ref{L:ASSO}, the domain of
$u$ contains the convex open set $\Omega$. As $\RHO \ll \LEB^d$, the
set $\Omega$ must be non-empty and $\RHO(\R^d\setminus\Omega) = 0$
(since the boundary of $\CCONV\SPT\RHO$ is a Lipschitz manifold of
codimension one, which is a Lebesgue null set and hence
$\RHO$-negligible). Consequently, the maximal monotone map $u$
associated to $\GAMMA$ is defined $\RHO$-a.e. The map $u$ is
single-valued for all $x \in \Omega\setminus\Sigma^1(u)$ (see
Remark~\ref{R:FUNC}), and $\Sigma^1(u)$ is a Lebesgue null set and
hence $\RHO$-negligible. We now define a (single-valued) function
$\BT$ on all of $\R^d$ as follows:
\[
  \BT(x) := \begin{cases}
      z & \text{if $x \in \Omega\setminus\Sigma^1(u)$ and $u(x) =: 
        \{z\}$,}
\\
      \bar{z} & \text{if $x \in \Omega\cap\Sigma^1(u)$ and $\bar{z}$ 
        is the center of mass of $u(x)$,}
\\
      0 & \text{if $x \in \R^d\setminus\Omega$.}
    \end{cases}
\]
Then $\BT$ is monotone on $\Omega$ because $\BT(x) \in u(x)$ for every
$x \in \Omega$. Recall that $u(x)$ is a closed convex set (possibly
empty) for all $x \in \R^d$; see Proposition~1.2 in
\cite{AlbertiAmbrosio1999}.

As shown in Remark~\ref{R:EQUIV}, there exists a Borel set $N_\ORT
\subset \R^{2d}$ such that $\MU(N_\ORT) = 0$ and $(x,\ORT(x,\xi)) \in
\SPT\GAMMA$ for all $(x,\xi) \in \R^{2d}\setminus N_\ORT$. This
implies that $\ORT(x,\xi) \in u(x)$ for such $(x,\xi)$, since
$\GRAPH(u)$ is an extension of $\SPT\GAMMA$. Therefore
\begin{gather*}
  \Big\{ (x,\xi) \in \R^{2d} \colon \ORT(x,\xi) \neq \BT(x) \Big\}
    \subset N_\ORT \cup (E\times\R^d),
\\
  \text{where}\quad
    E := (\R^d\setminus\Omega) \cup \big( \Omega\cap\Sigma^1(u) \big).
\end{gather*}
Since $\MU(N_\ORT) = 0$ and $\MU(E\times\R^d) = \RHO(E) = 0$,
statement \eqref{E:COINC} follows. Now
\[
  \int_{\R^d} |\BT(x)|^2 \,\RHO(dx)
    = \int_{\R^{2d}} |\BT(x)|^2 \,\MU(dx,d\xi)
    = \int_{\R^{2d}} |\ORT(x,\xi)|^2 \,\MU(dx,d\xi),
\]
which is finite. The $\BVS_\LOC(\Omega;\R^d)$-regularity of $\BT$
follows from Theorem~5.3 in \cite{AlbertiAmbrosio1999}.
\end{proof}

Lemma~\ref{L:PROP} shows that instead of minimizing over $\CMU$ it
is sufficient to consider a minimization over the following convex
cone in $\L^2(\R^d,\RHO)$ (we refer the reader to the proof of
Proposition~\ref{P:EXISTENCE} for topological properties):

\begin{definition}[Configurations]\label{D:CONFIG2}
Let $\RHO \in \SP_2(\R^d)$ satisfy $\RHO \ll \LEB^d$. We denote by
$\CR$ the set of all Borel maps $\BT \colon \R^d \longrightarrow
\R^d$ with the following properties:
\begin{enumerate}
\item $\BT$ is monotone on $\Omega := \INTR\CCONV\SPT\RHO$ (hence
  $\BT \in \BVS_\LOC(\Omega;\R^d)$),
\item $\BT \in \L^2(\R^d,\RHO)$. 
\end{enumerate}
\end{definition}

If $\MU \in \PR$ and $\ORT \in \CMU$ are given, and $\tau>0$, then
\begin{align}
  & \int_{\R^{2d}} |(x+\tau\xi)-\ORT(x,\xi)|^2 \,\MU(dx,d\xi)
\label{E:NINDE}\\
  & \qquad
    = \tau^2 \int_{\R^{2d}} |\xi-\BU(x)|^2 \,\MU(dx,d\xi)
    + \int_{\R^d} \big| \big( x+\tau\BU(x) \big)-\BT(x) \big|^2 \,\RHO(dx),
\nonumber
\end{align}
for every map $\BT \in \CR$ satisfying \eqref{E:COINC}. Here $\BU$
is the barycentric projection $\BAR(\MU)$ of $\MU$ (equivalently, the
orthogonal projection of $\MU$ onto the space of functions in
$\L^2(\R^{2d},\MU)$ that depend only on the spatial variable
$x\in\R^d$). Notice that the first term on the right-hand side of
\eqref{E:NINDE} does not depend on $\ORT$ or $\BT$.

For any smooth, strictly monotone map $\BT \colon \R^d \longrightarrow
\R^d$, the internal energy of the fluid transported by $\BT$ is given
(after a change of variables) by
\begin{align}
  \INT[\BT\#\RHO, \BT\#\sigma] 
    &= \int_{\R^d} U\Bigg( \bigg( \frac{r}{\det(\nabla\BT)} \bigg) \circ
      \BT^{-1}(z), S\circ \BT^{-1}(z) \Bigg) \,dz
\nonumber\\
    & \vphantom{\Bigg(}
      = \int_{\R^d} U\bigg( \frac{r(x)}{\det\big( \nabla\BT(x) \big)}, 
        S(x) \bigg) \det\big( \nabla\BT(x) \big) \,dx
\nonumber\\
    & \vphantom{\Bigg(}
      = \int_{\R^d} U\big( r(x),S(x) \big) 
        \; \det\big( \nabla\BT(x) \big)^{1-\gamma} \,dx.
\label{E:INTNEW}
\end{align}
Since the matrix $\nabla\BT$ may not be symmetric, the
functional $\BT \mapsto \INT[\BT\#\RHO, \BT\#\sigma]$ is not convex if
$d\GS 2$. In order to obtain a \emph{convex} minimization problem, we
modify the functional by replacing $\nabla\BT$ by the deformation,
i.e., its symmetric part.

\begin{definition}[Internal Energy]\label{D:INT2} 
Suppose that $(\RHO,\sigma) \in \SP_2(\R^d)\times\M_+(\R^d)$ with
$\RHO = r \LEB^d$, $\sigma = \RHO S$, and $\INT[\RHO, \sigma] <
\infty$. For any $\BT \in \CR$ let
\begin{equation}
  D\BT = \nabla\BT \,\LEB^d + D^s\BT,
  \quad
  D^s\BT \perp \LEB^d
\label{E:DECF}
\end{equation}
be the Lebesgue-Radon-Nikod\'{y}m decomposition of its derivative.
Then
\begin{equation}
  \INT[\BT|\RHO,\sigma] 
    := \int_{\R^d} U\big( r(x),S(x) \big) \;h\big( \nabla\BT(x) \big) \,dx
  \quad\text{for $\BT \in \CR$.}
\label{E:INT2}
\end{equation}
Recall that $h(\nabla\BT)$ only depends on the symmetric part of
$\nabla\BT$; see \eqref{E:HA}.
\end{definition}

\begin{remark}\label{R:ABSO} 
We have $U(r,S) \in \L^1(\R^d)$ as $\INT[\RHO,\sigma] < \infty$. In
\eqref{E:INT2} we may restrict the integration to $\Omega :=
\INTR\CCONV\SPT \RHO$ because the measures $\nu := U(r,S) \LEB^d$ and
$\RHO$ are mutually absolutely continuous, and $\RHO(\R^d
\setminus\Omega) = 0$ if $\RHO \ll \LEB^d$.

\DETAIL{ 
Defining $H(r,S) := \kappa e^S r^{\gamma-1}$, we have $\nu = H(r,S)
\RHO$ with $H(r,S) \in \L^1(\R^d, \RHO)$; see Definition~\ref{D:INT}.
Conversely, for any $A\subset\R^d$ Borel with $\nu(A) = 0$ we get
\begin{align*}
  \RHO(A) 
    &= \lim_{k\rightarrow\infty} \RHO(A\cap B_k(0))
\\
    &\LS \limsup_{k\rightarrow\infty}
      \Big( \LEB^d(B_k(0) )\Big)^{(\gamma-1)/\gamma}
      \bigg( \frac{1}{\kappa} \nu(A\cap B_k(0)) \bigg)^{1/\gamma}
        = 0,
\end{align*}
using H\"{o}lder inequality and the fact that $S(x)\GS 0$ for
$\RHO$-a.e.\ $x\in\R^d$.
} 
\end{remark}

\begin{remark}
Using only the symmetric part of $\nabla\BT$ can be justified by the
expectation that the map $\BT$ will be a perturbation of the identity,
whose derivative is the identity matrix everywhere, which is
symmetric. Using only $\nabla\BT$ instead of the derivative $D\BT$
means that the formation of vacuum does not cost any energy.
\end{remark}

The following lemma allows us to control \eqref{E:INTNEW} in terms of
\eqref{E:INT2}.

\begin{lemma}\label{L:CONTROL}
Suppose that density/entropy $(\RHO,\sigma) \in \SP_2(\R^d) \times
\M_+(\R^d)$ are given with $\RHO =: r \LEB^d$, $\sigma =: \RHO S$, and
$\INT[\RHO,\sigma] < \infty$. For any $\BT \in \CR$ with
$\INT[\BT|\RHO, \sigma] < \infty$ there exists a Borel set $\Sigma
\subset \R^d$ with $\RHO(\Sigma) = 0$ and $\BT |_{\R^d
\setminus\Sigma}$ injective. Then
\begin{equation}
  \INT[\BT\#\RHO,\BT\#\sigma] \LS \INT[\BT|\RHO,\sigma].
\label{E:CONTROL}
\end{equation}
\end{lemma}
 
\begin{proof}
We have $\RHO(\R^d \setminus\Omega) = 0$ with $\Omega :=
\INTR\CCONV\SPT\RHO$. Choose a maximal monotone set-valued map $u$
whose graph is an extension of $\Gamma := (\ID,\BT)\#\RHO$. Then $u(x)
= \{ \BT(x) \}$ for a.e.\ $x\in \Omega$ and $u$ is differentiable
a.e.: there is a $(d\times d)$-matrix $A(x)$ with
\begin{equation}
  \lim_{\substack{x'\rightarrow x\\ y \in u(x')}}
    \frac{y-\BT(x)-A(x)\cdot (x'-x)}{|x'-x|} = 0;
\label{E:DIFFQ}
\end{equation}
see Theorem~3.2 in \cite{AlbertiAmbrosio1999}. It follows that the
function $\BT$ is approximately differentiable a.e.\ in $\Omega$ (see
Definition~3.70 in \cite{AmbrosioFuscoPallara2000}) and $A$ coincides
with the absolutely continuous part $\nabla\BT$ of the derivative
$D\BT$; see Theorem~3.83 in \cite{AmbrosioFuscoPallara2000} and
\eqref{E:DECF}.

\DETAIL{ 
Notice that \eqref{E:DIFFQ} implies that for all $\EPS>0$ there exists
$\delta>0$ with the following property: for all $x'\in\R^d$ with
$|x'-x| \LS \delta$ we can estimate
\[
  |\BT(x')-\BT(x)-A(x)\cdot(x'-x)| \LS \EPS |x'-x|.
\]
We obtain
\[
  \fint_{B_r(x)}
    \frac{|\BT(x')-\BT(x)-A(x)\cdot(x'-x)|}{r} \,dx' \LS \EPS
\]
for all $0<r<\delta$. Since $\EPS>0$ was arbitrary, we get
\[
  \lim_{r\rightarrow 0} \fint_{B_r(x)}
    \frac{|\BT(x')-\BT(x)-A(x)\cdot(x'-x)|}{r} \,dx' = 0.
\]
} 

Let $D$ be the set of $x\in\Omega$ for which $u(x)$ is single-valued
and $u$ is differentiable at $x$ in the sense of \eqref{E:DIFFQ}. Then
$\LEB^d(\Omega \setminus D) = 0$. We define
\[
  N := \Big\{ x\in D \colon \text{there exists $x'\in\Omega$, 
    $x'\neq x$, with $\BT(x) \in u(x')$} \Big\}.
\]
For given $x\in N$ consider any $x'\in\Omega$, $x'\neq x$, such that
$\BT(x) \in u(x')$. By choice of $u$, we get $x, x' \in u^{-1}(y)$
with $y := \BT(x)$. Since the inverse map $u^{-1}$ is also maximal
monotone, the set $u^{-1}(y)$ is closed and convex, containing with
$x$ and $x'$ also the segment connecting the two points. Since $\BT$
is differentiable at $x$, we obtain
\[
  0 = \lim_{\substack{t\rightarrow 0 \\ y \in u(x_t)}}
     \frac{y-\BT(x)-\nabla \BT(x)\cdot(x_t-x)}{|x_t-x|}
    = -\nabla\BT(x) \cdot \xi,
\]
where $x_t := (1-t)x + tx'$ for $t\in[0,1]$ and $\xi :=
(x'-x)/|x'-x|$. Indeed notice that $\BT(x) \in u(x_t)$ for all $t \in [0,1]$.
Hence $\xi \neq 0$ is an eigenvector of the $(d\times d)$-matrix
$\nabla\BT(x)$, to the eigenvalue zero. Since $x\in N$ was arbitrary,
we obtain
\[
  N \subset \Big\{ x\in D \colon \det\big( \nabla\BT(x) \big) = 0 \Big\} =: M.
\]

Let $\nu := U(r,S)\LEB^d$. Since $\nu \ll \RHO$, we have that
$\nu(\R^d \setminus \Omega) = 0$. Since $\INT[\RHO,\sigma] < \infty$
implies that $U(r,S) \in \L^1(\R^d)$, we obtain $\nu(\Omega \setminus
D) = 0$. Finally, the assumption $\INT[\BT|\RHO,\sigma] < \infty$
requires that $\nu(M) = 0$. We conclude that the set
\[
  \Sigma := (\R^d\setminus\Omega) \cup (\Omega\setminus D) \cup M
\]
is $\nu$-negligible, hence $\RHO(\Sigma) = 0$; see Remark~\ref{R:ABSO}.
Then $\BT |_{\R^d\setminus\Sigma}$ is injective, which implies in
particular that $\BT\#\sigma = (S\circ\BT^{-1}) \, \BT\#\RHO$. Applying
Lemma~5.5.3 in \cite{AmbrosioGigliSavare2008} we conclude that the
equality \eqref{E:INTNEW} is true for $\BT$ (with suitable
modifications on sets of measure zero). We now use Lemma~\ref{L:DETS}
to obtain the estimate
\[
  0 < \det\big( \nabla\BT^\S(x) \big) 
    \LS \det\big( \nabla\BT(x) \big)
  \quad\text{for $\nu$-a.e.\ $x\in\R^d$}.
\]
Then inequality \eqref{E:CONTROL} follows from the definition
\eqref{E:INT2}.
\end{proof}

\begin{remark}\label{R:DISSINTQ}
Using Remark~\ref{R:DISSC}, we can give a more precise version of
\eqref{E:CONTROL}:
\begin{align}
  & \INT[\BT\#\RHO,\BT\#\sigma] - \INT[\BT|\RHO,\sigma]
\label{E:DIFIN}\\
  & \qquad = - \int_{\R^d} P(r,S) \,
    \bigg( \det(\nabla\BT^\S)^{1-\gamma} \int_0^1
      \det(\ONE+tC)^{1-\gamma} \mathbf{T}(t, C) \,dt \bigg) \,dx,
\nonumber
\end{align}
where $C(x) := R(x)^{-1} \nabla\BT^\A(x) R(x)^{-1}$ with $R(x) \in
\SYM[>]{d}$ such that
\begin{gather*}
  R(x)^2 = \nabla\BT^\S(x)
  \quad\text{for $\RHO$-a.e. $x\in\R^d$.}
\end{gather*}
For suitable $M \in\MAT{d}$ we defined
\begin{equation}
  \mathbf{T}(t,M) := \TRACE\big( (\ONE+tM)^{-1}M \big)
  \quad\text{for all $t\GS 0$.}
\label{E:DEFM}
\end{equation}
Note that the difference \eqref{E:DIFIN} vanishes if and only if
$\nabla\BT^\A(x) = 0$ for $\RHO$-a.e.\ $x\in\R$, i.e., if $\BT$ is not
only monotone, but \emph{optimal} in the sense of Remark~\ref{R:OPTW}.
\end{remark}

\begin{proposition}[Existence of Minimizers]\label{P:EXISTENCE}
Consider some triple $(\RHO,\MU,\sigma)$, with density $\RHO \in
\SP_2(\R^d)$, velocity distribution $\MU \in \PR$, and entropy $\sigma
\in \M_+(\R^d)$. Assume that $\RHO =: r \LEB^d$, $\sigma =: \RHO S$,
and $\INT[\RHO,\sigma] < \infty$. Given any timestep $\tau>0$, there
exists a unique $\BT_\tau \in \CR$ that minimizes the functional
\begin{equation}
  \Psi_\tau[\BT|\MU,\sigma] := \frac{3}{4\tau^2} \int_{\R^{2d}} 
    |(x+\tau\xi)-\BT(x)|^2 \,\MU(dx,d\xi) + \INT[\BT|\RHO,\sigma]
\label{E:OBJI}
\end{equation}
with $\BT \in \CR$. This minimum is finite, which implies in
particular that $\INT[\BT_\tau| \RHO,\sigma] < \infty$. For all Borel
maps $\BV \colon \R^d\longrightarrow\R^d$ with the property that
$\BT_\tau+\EPS\BV \in \CR$ for some $\EPS>0$, we have the
following inequality: let $P(r,S) := U'(r,S)r-U(r,S)$ for $r,S \GS 0$
(where $'$ denotes differentiation with respect to $r$). Then
\begin{align}
  & -\frac{3}{2\tau^2} \int_{\R^{2d}} \langle (x+\tau\xi)-\BT_\tau(x),
    \BV(x) \rangle \,\MU(dx,d\xi)
\label{E:EL}\\
  & \qquad
    -\int_{\R^d} P\big( r(x),S(x) \big) \; \det\big( \nabla \BT^\S_\tau(x) 
      \big)^{1-\gamma} \TRACE\Big( \big(
        \nabla\BT^\S_\tau(x) \big)^{-1} 
        \nabla \BV(x) \Big) \,dx \GS 0.
\nonumber
\end{align}
In particular, inequality \eqref{E:EL} is true for $\BV \in \CR$ since
$\CR$ is a convex cone.
\end{proposition}

\begin{proof}
We proceed in three steps.
\medskip

\textbf{Step~1.} We observe first that the infimum $\beta := \inf_{\BT
\in \CR} \Psi_\tau[\BT|\MU,\sigma]$ is non-negative. Furthermore
$\beta$ is finite because we may choose $\BT = \ID \in \CR$ to obtain
\[
  0 \LS \beta \LS \frac{3}{4} \int_{\R^{2d}} |\xi|^2 \,\MU(dx,d\xi) 
    + \INT[\RHO,\sigma] < \infty.
\]
We consider a sequence of $\BT^k \in \CR$ such that $\Psi_\tau[\BT^k|
\MU,\sigma] \longrightarrow \beta$ as $k\rightarrow \infty$. Without
loss of generality, we may assume that $\Psi_\tau[\BT^k| \MU,\sigma]
\LS \beta+1$ for all $k\in\N$. Then
\begin{align*}
  & \int_{\R^d} |\BT^k(x)|^2 \,\RHO(dx)
\\
  & \qquad
    \LS 2 \int_{\R^{2d}} |(x+\tau\xi)-\BT^k(x)|^2 \,\MU(dx,d\xi)
      + 2 \int_{\R^{2d}} |x+\tau\xi|^2 \,\MU(dx,d\xi)
\\
  & \qquad
    \LS \frac{8\tau^2}{3} (\beta+1)
      + 4 \bigg\{ \int_{\R^d} |x|^2 \,\RHO(dx)
        + \tau^2 \int_{\R^{2d}} |\xi|^2 \,\MU(dx,d\xi) \bigg\}
      < \infty.
\end{align*}
Therefore the sequence $\{\BT^k\}_k$ is precompact with respect to
weak convergence in $\L^2(\R^d,\RHO)$: there exists a subsequence
(still denoted by $\{\BT^k\}_k$) and $\BT \in \L^2(\R^d,\RHO)$ such
that $\BT^k \WEAK \BT$ weakly. By Mazur's lemma, there exists a map $K
\colon \N\longrightarrow\N$ with $K(n) \GS n$ for all $n\in\N$, and a
sequence of non-negative numbers
\[
  \{ \lambda^n_k \colon k=n\ldots K(n) \}
\]
with $\sum_{k=n}^{K(n)} \lambda_k^n = 1$, with the property that
\[
  \BS^n := \sum_{k=n}^{K(n)} \lambda_k^n \BT^k \longrightarrow \BT
  \quad\text{strongly in $\L^2(\R^d,\RHO)$}
\]
as $n\rightarrow\infty$. Notice that $\BS^n \in \CR$ since $\CR$ is a
convex cone. We apply Proposition~\ref{P:LSCCON2} (the convexity of
the quadratic term in \eqref{E:OBJI} is easy to check) to estimate
\[
  \beta \LS \Psi_\tau[\BS^n|\MU,\sigma] 
    \LS \sum_{k=n}^{K(n)} \lambda^n_k \Psi_\tau[\BT^k|\MU,\sigma]
    \longrightarrow \beta.
\]
Consequently, we obtain a strongly convergent minimizing sequence.
Without loss of generality, we may assume that $\Psi_\tau[\BS^n|
\MU,\sigma] \LS \beta+1$ for all $n\in\N$. Extracting another
subsequence if necessary, we may even assume the existence of a Borel
set $N \subset \R^d$ with $\RHO(\R^d\setminus N)=0$ such that
$\BS^n(x) \longrightarrow \BT(x)$ for all $x \in \R^d\setminus N$.
\medskip

\textbf{Step~2.} It remains to establish the lower semicontinuity of
the functional \eqref{E:OBJI}. The quadratic part is clearly lower
semicontinuous with respect to weak convergence in $\L^2(\R^d,\RHO)$.
For the internal energy part, we will prove that the sequence
$\{\BS^n\}_n$ is weak* precompact in $\BVS_\LOC(\Omega;\R^d)$. Then we
apply Proposition~\ref{P:LSCCON2}.

For
all $m\in\N$, we define the convex compact sets
\[
  \Omega_m := \Big\{ x\in\R^d \colon
    \text{$|x|\LS m$ and $\DIST(x,\R^d\setminus\Omega) \GS 1/m$} \Big\}.
\]
Then $\bigcup_{m\in\N} \Omega_m = \Omega$. Let us fix $m$ for the
moment. For each $x\in\Omega_{m+1}$ there exist finitely many points in
$\Omega$ with the property that $x$ is in the interior of the
convex hull of these points. Therefore we can even find an open ball
centered at $x$ that is contained in the convex hull of these points.
The collection of balls obtained in this way form an open covering of
$\Omega_{m+1}$. By compactness of $\Omega_{m+1}$, we may choose a
finite subcovering. This proves the following statement: there exist
finitely many points $x_m^i \in \Omega$, $i=1\ldots I_m$ for
some $I_m\in\N$, with the property that
\[
  \Omega_{m+1} \subset \HULL X_m,
  \quad\text{where}\quad
  X_m := \{ x_m^i \colon i=1\ldots I_m \}.
\]
By adapting the argument in the proof of Lemma~\ref{L:ASSO}, we can
write each $x_m^i \in X_m$ as a convex combination of points
$z_m^{i,j} \in \Omega\setminus N$ with $j=1\ldots J_m^i$ for some
$J_m^i \in \N$. Recall that $\RHO(\R^d\setminus N) = 0$ and $\BS^n(x)
\longrightarrow \BT(x)$ for all $x\in\R^d\setminus N$. Thus
\begin{equation}
  \Omega_{m+1} \subset \HULL Z_m,
  \quad\text{where}\quad
  Z_m := \{ z_m^{i,j} \colon j=1\ldots J_m^i, i=1\ldots I_m \}.
\label{E:OMEGA}
\end{equation}
Since the sequence $\{\BS^n(z_m^{i,j})\}_n$ converges, it must be
bounded. Let
\[
  \beta_m^n 
    := \max_{i=1\ldots I_m} \max_{j=1\ldots J_m^i} |\BS^n(z_m^{i,j})|.
\]
Then $\{\beta_m^n\}_n$ is uniformly bounded for every $m\in\N$. We now
observe that
\[
  \sup_{x\in\Omega_m} |\BS^n(x)| \LS \frac{\beta_m^n \DIAM(Z_m)}
    {\DIST(\Omega_m, \R^d\setminus\Omega_{m+1})},
\]
which is bounded uniformly in $n$; see Proposition~1.2 in
\cite{AlbertiAmbrosio1999} and \eqref{E:OMEGA}. We conclude that
$\{\BS^n\}_n$ is uniformly bounded in $\L^\infty(\Omega_m; \R^d)$ for
all $m\in\N$. Since
\begin{equation}
  \int_{\Omega_m} |D\BS^n| \LS c_d \, \DIAM(\Omega_m)^{d-1} 
    \OSC(\BS^n,\Omega_m),
\label{E:BVB}
\end{equation}
where $c_d>0$ is a constant depending only on the space dimenension,
and where
\[
  \OSC(\BS^n,A) := \sup_{x_1,x_2\in A} |\BS^n(x_1)-\BS^n(x_2)|
  \quad\text{for all $A\subset\R^d$}
\]
denotes the oscillation of $\BS^n$ over $A$, we obtain that the
sequence $\{\BS^n\}_n$ is uniformly bounded in $\BVS(\Omega_m; \R^d)$
for all $m\in\N$, thus precompact in $\BVS_\LOC(\Omega;\R^d)$. We
refer the reader to Proposition~5.1 and Remark~5.2 in
\cite{AlbertiAmbrosio1999} for a proof of \eqref{E:BVB}.

Extracting another subsequence if necessary (not relabeled), we find
that $\BS^n \WEAK \BS$ weak* in $\BVS_\LOC(\Omega;\R^d)$ for a
suitable function $\BS \in \BVS_\LOC(\Omega;\R^d)$. One can now check that
$\BS$ is again a monotone map on $\Omega$ (possibly after redefining
$\BS$ on a set of measure zero). Moreover, we have $\BS(x) = \BT(x)$
for $\RHO$-a.e.\ $x\in\Omega$, by construction. Defining
$\BT_\tau(x) := \BS(x)$ for $x\in\Omega$, and $\BT_\tau(x) := 0$ for
$x\in\R^d\setminus\Omega$, we have that
\[
  \BT_\tau \in \CR
  \quad\text{and}\quad
  \Psi_\tau[\BT_\tau|\MU,\sigma] \LS \liminf_{n\rightarrow\infty}
    \Psi_\tau[\BS^n|\MU,\sigma].
\]
In particular, we get $\Psi_\tau[\BT_\tau|\MU,\sigma] = \beta$, thus
$\BT_\tau$ is a minimizer. Its uniqueness follows from the strict
convexity of the first term in \eqref{E:OBJI}, which is quadratic in
$\BT$. 
\medskip

\textbf{Step~3.} Consider $\BV \in \L^2(\R^d,\RHO)$ such that
$\BT_\tau+\EPS\BV \in \CR$ for $\EPS>0$ small. Since $\BT_\tau \in
\CR$, we have that $\BV \in \BVS_\LOC(\Omega;\R^d)$ as well; see
Definition~\ref{D:CONFIG2}. Then
\[
  \Psi_\tau[\BT_\tau+\EPS\BV|\MU,\sigma] - \Psi_\tau[\BT_\tau|\MU,\sigma] 
    \GS 0.
\]
We divide by $\EPS>0$ and consider the limit $\EPS\rightarrow 0$. We
obtain that
\begin{align*}
  & \lim_{\EPS\rightarrow 0+} \frac{1}{\EPS} \Bigg\{ 
      \frac{3}{4\tau^2} \int_{\R^{2d}} \big| (x+\tau\xi)-\big( 
        \BT_\tau(x)+\EPS\BV(x) \big) \big|^2 \,\MU(dx,d\xi) 
\\
  & \hspace{8em}
    - \frac{3}{4\tau^2} \int_{\R^{2d}} |(x+\tau\xi)-\BT_\tau(x)|^2 
      \,\MU(dx,d\xi) \Bigg\}
\\
  & \qquad 
    = -\frac{3}{2\tau^2} \int_{\R^{2d}} \langle (x+\tau\xi)-\BT_\tau(x),
      \BV(x) \rangle \,\MU(dx,d\xi).
\end{align*}
Since $\INT[\BT_\tau|\RHO,\sigma] < \infty$, we can further write
\begin{align*}
  & \lim_{\EPS\rightarrow 0+} \frac{\INT[\BT_\tau+\EPS\BV|\RHO,\sigma]
    -\INT[\BT_\tau|\RHO,\sigma]}{\EPS}
\\ 
  & \qquad 
    = \lim_{\EPS\rightarrow 0+} \int_{\R^d} U\big( r(x),S(x) \big) \; 
      \frac{1}{\EPS} \Big\{ \det\big( \nabla\BT^\S_\tau(x)  
        + \EPS\nabla\BV^\S(x) \big)^{1-\gamma}
\\
   & \hspace{21em}
      - \det\big( \nabla\BT^\S_\tau(x) \big)^{1-\gamma} \Big\} \,dx.
\end{align*}
We can restrict the integration to $\Omega$ where $\nabla\BT_\tau$,
$\nabla\BV$ are well-defined; see Remark~\ref{R:ABSO}. Since $A
\mapsto \det(A^\S)^{1-\gamma}$ is convex (see
Proposition~\ref{P:LSCCON2}), the term in curly brackets is
non-decreasing for a.e.\ $x\in\R^d$. By monotone convergence, it
follows that
\begin{align*}
  & \lim_{\EPS\rightarrow 0+} 
    \frac{\INT[\BT_\tau+\EPS\BV|\RHO,\sigma]-\INT[\BT_\tau|\RHO,\sigma]}{\EPS}
\\
  & \qquad 
    = -\int_{\R^d} P\big( r(x),S(x) \big) \; \det\big( \nabla\BT^\S_\tau(x) 
      \big)^{1-\gamma} \TRACE\Big( \big( \nabla\BT^\S_\tau(x) \big)^{-1}
        \nabla\BV^\S(x) \Big) \,dx.
\end{align*}
We now can replace $\nabla\BV^\S(x)$ by $\nabla\BV(x)$ since the
antisymmetric part of the derivative cancels in the inner product with
a symmetric matrix.
\end{proof}

\begin{remark}
Instead of using Mazur's lemma to get strong
$\L^2(\R^d,\RHO)$-convergence (and thus convergence pointwise a.e., up
to a subsequence), in Step~2 we can also use narrow convergence of the
transport plans $(\ID,\BT^k)\#\RHO$ together with Kuratowski
convergence of their supports; see Proposition~5.1.8 in
\cite{AmbrosioGigliSavare2008}.
\end{remark}

\begin{remark}\label{R:INJE}
Since $\INT[\BT_\tau| \RHO,\sigma] < \infty$, we can apply
Lemma~\ref{L:CONTROL} to conclude that $\BT_\tau$ is essentially
injective and $\INT[\RHO_\tau, \sigma_\tau] < \infty$, where
$(\RHO_\tau, \sigma_\tau) := \BT_\tau\# (\RHO, \sigma)$. It follows
that $\RHO_\tau$ must be absolutely continuous with respect to the
Lebesgue measure and $\sigma_\tau = \RHO_\tau S_\tau$ with transported
entropy $S_\tau := S\circ\BT_\tau^{-1}$; recall
Definition~\ref{D:INT}.
\end{remark}

\begin{remark}\label{R:MOM2}
Using the test functions $\BV = \pm\BT_\tau$ in \eqref{E:EL}, we
obtain
\begin{align}
  & \frac{3}{2\tau^2} \int_{\R^{2d}} \langle (x+\tau\xi)-\BT_\tau(x),
    \BT_\tau(x) \rangle \,\MU(dx,d\xi)
\label{E:HGAA}\\
  & \qquad
    +\int_{\R^d} P\big( r(x),S(x) \big) \; \det\big( \nabla\BT^\S_\tau(x) 
      \big)^{1-\gamma} \TRACE\Big( \big( \nabla\BT^\S_\tau(x) \big)^{-1} 
        \nabla\BT_\tau(x) \Big) \,dx = 0.
\nonumber
\end{align}
This is the analogue of equality \eqref{E:OPI1} from the pressureless
case. As a consequence, we can rewrite \eqref{E:EL} in the following
form (cf. \eqref{E:OPI2}): for all $\BS \in \CR$ we have
\begin{align}
  & \frac{3}{2\tau^2} \int_{\R^{2d}} \langle (x+\tau\xi)-\BT_\tau(x),
    \BS(x) \rangle \,\MU(dx,d\xi)
\label{E:ELL}\\
  & \qquad
    +\int_{\R^d} P\big( r(x),S(x) \big) \; \det\big( \nabla\BT^\S_\tau(x) 
      \big)^{1-\gamma} \TRACE\Big( \big( \nabla\BT^\S_\tau(x) \big)^{-1}
        \nabla\BS(x) \Big) \,dx \LS 0.
\nonumber
\end{align}
Using in \eqref{E:ELL} the constant maps $\BS(x)=\pm b$ for all $x \in
\R^d$, where $b\in\R^d$ is some vector, we conclude that the
minimization in Proposition~\ref{P:EXISTENCE} again preserves the
total momentum; see Remark~\ref{R:MOMMY} for more details. Similarly,
using $\BS(x) := \pm Ax$ with $A\in\SKEW{d}$, we obtain global
conservation of angular momentum. Notice that in this case, the trace in
\eqref{E:ELL} vanishes since $(\nabla\BT^\S_\tau(x))^{-1}$ is
symmetric.
\end{remark}

\begin{remark}
In \eqref{E:HGAA} we can replace $\nabla\BT_\tau(x)$ by the
deformation $\nabla\BT^\S_\tau(x)$ since the antisymmetric part cancels
in the trace. By Cramer's rule, we obtain
\begin{align}
  & -\frac{3}{2\tau^2} \int_{\R^{2d}} \langle (x+\tau\xi)-\BT_\tau(x),
    \BT_\tau(x) \rangle \,\MU(dx,d\xi)
\label{E:HGAA2}\\
  & \qquad
    = d \int_{\R^d} P\big( r(x),S(x) \big) \; \det\big( \nabla\BT^\S_\tau(x) 
      \big)^{1-\gamma} \,dx
    = d(\gamma-1) \; \INT[\BT_\tau|\RHO,\sigma].
\nonumber
\end{align}
\end{remark}

\begin{definition}\label{D:NOTAT}
For $(\RHO,\MU,\sigma,\tau)$ as in Proposition~\ref{P:EXISTENCE}, let
$\BT_\tau$ denote the unique minimizer considered there. We define
$\ORT_\tau, \ORW_\tau, \ORU_\tau \in \L^2(\R^{2d}, \MU)$ as follows:
\begin{equation}
  \ORT_\tau(x,\xi) := \BT_\tau(x),
  \quad
  \ORU_\tau(x,\xi) := \ORW_\tau(x,\xi) 
    := V_\tau\big( x,\xi,\BT_\tau(x) \big)
\label{E:NEWD}
\end{equation}
for $\MU$-a.e.\ $(x,\xi)\in\R^{2d}$, with $V_\tau$ given by
\eqref{E:ZETOP}. Then
\[
  (\RHO_\tau, \sigma_\tau) := \BT_\tau\#(\RHO,\sigma),
  \quad 
  \MU_\tau := (\ORT_\tau, \ORU_\tau)\#\MU.
\]
\end{definition}

\begin{remark}\label{R:MONOK}
The definition of $\ORT_\tau$ in \eqref{E:NEWD} is natural in view of
Proposition~\ref{L:PROP}. If $\MU = (\ID,\BU)\#\RHO$ for some $\BU \in
\L^2(\R^d,\RHO)$ and $\MU_* := (\ORT_\tau, \ORW_\tau)\#\MU$, then
\[
  \int_{\R^{2d}} \varphi(z,\zeta) \,\MU_*(dz,d\zeta) 
    = \int_{\R^d} \varphi\big( \BT_\tau(x), W(x) 
      \big) \,\RHO(dx)
\]
for all $\varphi \in \CB(\R^{2d})$, with $W := \frac{3}{2} V -
\frac{1}{2}\BU$ and $V := (\BT_\tau-\ID)/\tau$. Let
\begin{equation}
  \BU_\tau(z) := W\big( \BT_\tau^{-1}(z) \big)
  \quad\text{for $\RHO_\tau$-a.e.\ $z\in\R^d$.}
\label{E:UTAU}
\end{equation}
The velocity $\BU_\tau$ is well-defined because $\BT_\tau$ is
essentially injective; see Remark~\ref{R:INJE}. It follows that $\MU_*
= (\ID, \BU_\tau)\#\RHO_\tau$ and $\BU_\tau \in \L^2(\R^d,\RHO_\tau)$.
We would like to emphasize the fact that the minimization preserves
the monokinetic structure of the fluid (recall that the velocity
update in \eqref{E:NEWD} is a consequence of the minimization of the
work functional). Since the tangent cone over the cone of monotone
maps at $\RHO_\tau$ equals $\L^2(\R^d,\RHO_\tau)$, no additional
projection is necessary (unlike in the pressureless gas case; see
Step~(2) in Definition~\ref{D:EM}). We can therefore put $\ORU_\tau =
\ORW_\tau$.
\end{remark}

\begin{proposition}[Stress Tensor]\label{P:SIGMA}
Suppose that $\tau>0$ and $(\RHO,\MU,\sigma)$ are given, with density
$\RHO \in \SP_2(\R^d)$, velocity distribution $\MU \in \PR$, and
entropy $\sigma \in \M_+(\R^d)$. Assume that $\RHO =: r \LEB^d$,
$\sigma =: \RHO S$, and $\INT[\RHO,\sigma] < \infty$. Consider the
unique minimizer $\BT_\tau \in \CR$ from
Proposition~\ref{P:EXISTENCE}. There exists $\RES_\tau \in
\M(\R^d; \SYM[\GS]{d})$ with
\begin{multline}
  \int_{\R^d} \TRACE\big( \nabla u(x) \RES_\tau(dx) \big)
    = -\frac{3}{2\tau^2} \int_{\R^{2d}} \langle (x+\tau\xi)
      -\BT_\tau(x), u(x) \rangle \,\MU(dx,d\xi)
\label{E:ID2}\\
    -\int_{\R^d} P\big( r(x),S(x) \big) \; \det\big( \nabla\BT^\S_\tau(x)
      \big)^{1-\gamma} \TRACE\Big( \big(
        \nabla\BT^\S_\tau(x) \big)^{-1} 
        \nabla u(x) \Big) \,dx
\end{multline}
for all $u \in \C^1_*(\R^d; \R^d)$. In particular, we have the control
\begin{multline}
  \int_{\R^d} \TRACE\big( \RES_\tau(dx) \big)
    = -\frac{3}{2\tau^2} \int_{\R^{2d}} \langle (x+\tau\xi)
      -\BT_\tau(x), x \rangle \,\MU(dx,d\xi)
\label{E:SIGID2}\\
    -\int_{\R^d} P\big( r(x),S(x) \big) \; \det\big( \nabla\BT^\S_\tau(x)
      \big)^{1-\gamma} \TRACE\Big( \big(
        \nabla\BT^\S_\tau(x) \big)^{-1} \Big) \,dx.
\end{multline}
\end{proposition}

\begin{proof}
Since every $u \in \MON(\R^d)$ has at most linear growth, we have $u
\in \L^2(\R^d, \RHO)$. Thus $\MON(\R^d) \subset \CR$ and $\BV := u
\in \MON(\R^d)$ is admissible in \eqref{E:EL}. Let
\[
  \PK(dx) := P\big( r(x),S(x) \big) \; \det\big( \nabla\BT^\S_\tau(x)
    \big)^{1-\gamma} \big( \nabla\BT^\S_\tau(x) \big)^{-1} \,dx.
\]
The inverse matrix $(\nabla\BT^\S_\tau(x))^{-1}$ is symmetric and
positive definite for a.e.\ $x\in\Omega$ because $\BT_\tau$ is
monotone there. Consequently, its norm can be controlled by the trace.
Using $\BV=\ID$ (which is an element of $\CR$) in \eqref{E:EL}, we
obtain the estimate
\[
  0 \LS \int_{\R^d} \TRACE\big( \PK(dx) \big)
  \LS -\frac{3}{2\tau^2} \int_{\R^{2d}} \langle (x+\tau\xi)-\BT_\tau(x),
    x \rangle \,\MU(dx,d\xi),
\]
which is finite. Thus $\PK \in \M(\R^d;\SYM[\GS]{d})$. If we define
\[
  \F(dx) := -\frac{3}{2\tau^2} \Big( \big( x+\tau\BU(x) \big)
      -\BT_\tau(x) \Big) \,\RHO(dx),
\]
where $\BU := \BAR(\MU)$ denotes the barycentric projection of $\MU$
(which is in $\L^2(\R^d,\RHO)$), then $\F \in \M(\R^d;\R^d)$ has
finite first moment because $\RHO \in \SP_2(\R^d)$. We then apply
Theorem~\ref{T:STRESS} to obtain the representation
\eqref{E:ID2}/\eqref{E:SIGID2}; see also Remark~\ref{R:FINI}.
\end{proof}

\begin{proposition}[Energy Balance]\label{P:EB2}
Let $\tau>0$ and $(\RHO,\BU,\sigma)$ are given, with density $\RHO \in
\SP_2(\R^d)$, Eulerian velocity $\BU \in \L^2(\R^d,\RHO)$, and entropy
$\sigma \in \M_+(\R^d)$. Suppose that $\RHO =: r \LEB^d$, $\sigma =:
\RHO S$, and $\INT[\RHO,\sigma] < \infty$. Let $\BT_\tau \in \CR$
denote the unique minimizer from Proposition~\ref{P:EXISTENCE} (where
$\MU := (\ID,\BU)\#\RHO$) and $\RES_\tau \in \M(\R^d; \SYM[\GS]{d})$
the stress tensor field in Proposition~\ref{P:SIGMA}. Consider
$(\RHO_\tau, \BU_\tau, \sigma_\tau)$ and $\BW_\tau$ as defined in the
Remarks~\ref{R:INJE}/\ref{R:MONOK}. Then the following energy equality
holds:
\begin{align}
  \E[\RHO_\tau, \BU_\tau, \sigma_\tau]
    & + \int_{\R^d} {\TST\frac{1}{6}} \RHO |\BW_\tau-\BU|^2 
\label{E:ENIN2}\\
    & + \int_{\R^d} \bigg( 
      P(r,S) \, \DISS^2\big( \nabla\BT_\tau-\ONE \big) \Big) \,dx 
        + \TRACE\big( \RES_\tau(dx) \big) \bigg)
    = \E[\RHO,\BU,\sigma],
\nonumber
\end{align}
with total energy (recall Definition~\ref{D:INT})
\[
  \E[\RHO,\BU,\sigma]
    := \int_{\R^d} \OH \RHO|\BU|^2 + \INT[\RHO, \sigma].
\]
For all matrices $\ONE+S\in\SYM[>]{d}$ and $A\in\SKEW{d}$ we have
\begin{align*}
  \DISS^2(S+A)
    & := \int_0^1 \det(\ONE+tS)^{1-\gamma} \Big( (\gamma-1) \mathbf{T}(t,S)^2
        + \mathbf{T}_2(t,S) \Big) t \,dt 
\\
    & \hphantom{:} + \det(\ONE+S)^{1-\gamma} \int_0^1 \det(\ONE+tC)^{1-\gamma}
      \mathbf{T}(t,C) \,dt \GS 0.
\end{align*}
Here $C := R^{-1}AR^{-1}$ and $R\in\SYM[>]{d}$ is uniquely determined
by $\ONE+S =: R^2$. Recall \eqref{E:DEFM} for the definition of
$\mathbf{T}$. For suitable $M \in\MAT{d}$ we define
\[
  \mathbf{T}_2(t,M) := \TRACE\Big( \big( (\ONE+tM)^{-1}M \big)^2 \Big)
  \quad\text{for all $t\GS 0$.}
\]
Notice that all terms in curly brackets in \eqref{E:ENIN2} are
non-negative.
\end{proposition}

\begin{proof}
Let us first consider the kinetic energy. Because of
\eqref{E:IDTIES}/\eqref{E:UTAU}, we have
\begin{align*}
  & \int_{\R^d} \OH|\BU_\tau(x)|^2 \,\RHO_\tau(dx)
    + \frac{1}{6} \int_{\R^d} |\BW_\tau(x)-\BU(x)|^2 \,\RHO(dx)
\\
  & \qquad
    = \int_{\R^d} \OH|\BU(x)|^2 \,\RHO(dx)
    - \frac{3}{2\tau^2} \int_{\R^{2d}} \langle (x+\tau\xi)-\BT_\tau(x),
      \BT_\tau(x)-x \rangle \,\MU(dx,d\xi)
\end{align*}
Combining \eqref{E:HGAA} with the representation \eqref{E:SIGID2}, we
find that
\begin{align}
  & - \frac{3}{2\tau^2} \int_{\R^{2d}} \langle (x+\tau\xi)-\BT_\tau(x),
      \BT_\tau(x)-x \rangle \,\MU(dx,d\xi)
\label{E:BHG}\\
  & \qquad
    = \tau \int_{\R^d} P\big( r(x),S(x) \big) \; \det\big( \nabla\BT^\S_\tau(x)
      \big)^{1-\gamma} \TRACE\Big( \big( \nabla\BT^\S_\tau(x) \big)^{-1} 
        \nabla \BV_\tau(x) \Big) \,dx
\nonumber\\
  & \qquad\quad
    - \int_{\R^d} \TRACE\big( \RES_\tau(dx) \big).
\nonumber
\end{align}
Let $\BT(s,x) := x+s\tau\BV_\tau(x)$ for $s\in[0,1]$. Taylor expanding
around $s=1$, we get
\begin{align}
  & \det\big( \nabla\BT^\S_\tau(x) \big)^{1-\gamma} = 1
\nonumber\\
  &\quad
    - \tau(\gamma-1) \det\big( \nabla\BT^\S_\tau(x) \big)^{1-\gamma} \;
      \TRACE\Big( \big( \nabla\BT^\S_\tau(x) \big)^{-1} \nabla\BV_\tau(x) \Big)
\label{E:REPLA}\\
  & \quad
      - \int_0^1
        \det\big( \nabla\BT^\S(s,x) \big)^{1-\gamma}
        \Bigg\{ (\gamma-1)^2 \,
          \bigg( \TRACE\Big( \big( \nabla\BT^\S(s,x)
            \big)^{-1} \tau\nabla\BV^\S_\tau(x) \Big) \bigg)^2
\nonumber\\
  & \quad\hphantom{-\int_0^\tau \det\big( \nabla\BT^\S(s,x) 
        \big)^{-\gamma} } \;
      + (\gamma-1) \, \TRACE\bigg( \Big( \big( \nabla\BT^\S(s,x)
        \big)^{-1} \tau\nabla\BV^\S_\tau(x) \Big)^2 \bigg) \Bigg\}
          s \,ds
\nonumber
\end{align}
for a.e.\ $x\in\Omega$. We now multiply by $U(r(x),S(x))$ and
integrate in $x\in \R^d$. Then the integral of \eqref{E:REPLA} equals
the negative of the first term on the right-hand side of
\eqref{E:BHG}. Combining all terms and using Remark~\ref{R:DISSINTQ},
we conclude the proof.
\end{proof}

\begin{remark}[Bregman Divergence]\label{R:RESA}
We observe that the function 
\begin{align}
  D_\INT(S) 
    & := \Big( 1-\det(\ONE+S)^{1-\gamma} \Big) 
      -(\gamma-1) \det(\ONE+S)^{1-\gamma} \TRACE\Big( (\ONE+S)^{-1}S \Big)
\label{E:DINTP}\\
    & \hphantom{:} = \int_0^1 \det(\ONE+tS)^{1-\gamma} \Big( (\gamma-1) 
      \mathbf{T}(t,S)^2 + \mathbf{T}_2(t,S) \Big) t \,dt \GS 0,
\nonumber
\end{align}
defined for every $S \in \SYM{d}$ with $\ONE+S$ positive definite, is
the Bregman divergence for $0$ and $S$ associated to the convex
function $S \mapsto \det(\ONE+S)^{1-\gamma}$. 
\end{remark}

The following result will be useful to control the momentum equation
of \eqref{E:FULL}.

\begin{lemma}\label{L:COMML}
For every $\EPS>0$ there exists a constant $C_\EPS>0$ with the
following property: For all $S \in \SYM{d}$ with $\ONE+S$ positive
definite, we have
\begin{equation}
  \sup_{z\in\R^d, |z|=1} 
    \big| \big\langle z, \big( \ONE-\det(\ONE+S)^{1-\gamma} (\ONE+S)^{-1}
        \big)z \big\rangle \big|
    \LS \EPS + C_\EPS D_\INT(S),
\label{E:DISSL}
\end{equation}
where $D_\INT$ is defined in \eqref{E:DINTP}. 

Similar, for any $\EPS>0$ there exists $C_\EPS>0$ such that
\begin{equation}
  \sup_{z\in\R^d, |z|=1} 
    \big| \big\langle z, \big( \BW\otimes(\BV-\BW) \big) z \big\rangle \big|
      \LS \EPS |\BW|^2 + C_\EPS D_\K|\BW-\BU|^2
\label{E:DISSK}
\end{equation}
for all $\BV, \BW \in \R^d$ and $\BU := 3\BV-2\BW$.
\end{lemma}

\begin{proof} 
Notice first that the map $S \mapsto \ONE - \det(\ONE+S)^{1-\gamma}
(\ONE+S)^{-1}$ vanishes for $S=0$ and is continuous there.
Consequently, for any $\EPS>0$ there exists a $\delta>0$ such that the
left-hand side of \eqref{E:DISSL} is less than $\EPS$ for all $S \in
\SYM{d}$ with $\|S\| < \delta$.

For any $S \in \SYM{d}$ with $\ONE+S$ positive definite, we rewrite
\begin{align*}
  & \ONE - \det(\ONE+S)^{1-\gamma} (\ONE+S)^{-1}
\\
  & \quad
    = \Big( 1-\det(\ONE+S)^{1-\gamma} \Big) \ONE
    + \det(\ONE+S)^{1-\gamma} \Big( \ONE-(\ONE+S)^{-1} \Big)
\end{align*}
Because of the spectral theorem, there exist real eigenvalues
$\lambda_i$ and a corresponding system of orthonormal eigenvalues $e_i
\in \R^d$, $i=1\ldots d$, such that $S = \sum_{i=1}^d \lambda_i
e_i\otimes e_i$. We also have the identity $\sum_{i=1}^d e_i\otimes
e_i = \ONE$. We can then write
\begin{align*}
  & \ONE - \det(\ONE+S)^{1-\gamma} (\ONE+S)^{-1}
\\
  & \quad
    = \bigg( 1-\prod_{i=1}^d (1+\lambda_i)^{1-\gamma} \bigg) \ONE
    + \prod_{i=1}^d (1+\lambda_i)^{1-\gamma} 
      \sum_{i=1}^d \frac{\lambda_i}{1+\lambda_i} e_i\otimes e_i.
\end{align*}
Multiplying from left and right by a vector $z\in\R^d$ with $|z|=1$,
we obtain
\begin{align*}
  & \big\langle z, \big( \ONE - \det(\ONE+S)^{1-\gamma} (\ONE+S)^{-1} 
    \big) z \big\rangle
\\
  & \quad
    = \bigg( 1-\prod_{i=1}^d (1+\lambda_i)^{1-\gamma} \bigg)
    + \prod_{i=1}^d (1+\lambda_i)^{1-\gamma} 
      \sum_{i=1}^d c_i^2 \frac{\lambda_i}{1+\lambda_i},
\end{align*}
where $c_i := z\cdot e_i$ and $\sum_{i=1}^d c_i^2 = 1$. Similarly, we
can rewrite \eqref{E:DINTP} as
\begin{align*}
  D_\INT(S) 
    & = \bigg( 1-\prod_{i=1}^d (1+\lambda_i)^{1-\gamma} \bigg)
      - (\gamma-1) \prod_{i=1}^d (1+\lambda_i)^{1-\gamma} 
        \sum_{i=1}^d \frac{\lambda_i}{1+\lambda_i}.
\\
    & = \int_0^1 \prod_{i=1}^d (1+t\lambda_i)^{1-\gamma} \Bigg\{ 
      (\gamma-1) \Bigg( \sum_{i=1}^d \frac{\lambda_i}{1+t\lambda_i} \Bigg)^2
      + \sum_{i=1}^d \bigg( \frac{\lambda_i}{1+t\lambda_i} \bigg)^2 \Bigg\} t\,dt,
\end{align*}
from which we conclude that $D_\INT(S) = 0$ if and only if all
eigenvalues $\lambda_i$ vanish, thus $S=0$. Recall that $\gamma-1>0$,
by assumption. In particular, we have $D_\INT(S)>0$ for all $S \in
\SYM{d}$ such that $\|S\|\GS \delta$. By continuity and compactness,
for any $\gamma < 1$ there exists a constant $c_\gamma > 0$ with
$D_\INT(S) \GS c_\gamma$ for all $\gamma \GS \|S\| \GS \delta$.

To simplify the notation, we will write
\[
  d(\lambda) := \prod_{i=1}^d (1+\lambda_i)^{1-\gamma},
  \quad
  S_c(\lambda) := \sum_{i=1}^d c_i^2 \frac{\lambda_i}{1+\lambda_i}
\]
for all $\lambda := (\lambda_1, \ldots, \lambda_d)$ with $\lambda_i >
-1$ and $c := (c_1, \ldots, c_d)$. We claim that
\begin{equation}
  F(\lambda) := \frac{1-d(\lambda) \Big( 1-S_c(\lambda) \Big)}
    {1-d(\lambda) \Big( 1+(\gamma-1) S_{(1, \ldots, 1)}(\lambda) \Big)}
\label{E:QUOT}
\end{equation}
is uniformly bounded away from $\lambda = 0$. Then the estimate
\eqref{E:DISSL} follows.

In order to prove the claim, we first observe that the level sets of
$d(\lambda)$ generate a partition of the orthant $(-1,\infty)^d$ into
hyperboloids. For simplicity, we only consider the case $d=2$. The
general case can be handled similarly. We introduce a coordinate
system adapted to $(-1,\infty)^2$ as follows: For all $(\alpha,\beta)
\in (0,\pi/2)^2$ let
\[
  \lambda_1(\alpha,\beta) := \sqrt{\tan(\alpha)\cot(\beta)}-1,
  \quad
  \lambda_2(\alpha,\beta) := \sqrt{\tan(\alpha)\tan(\beta)}-1.
\]
Notice that with this choice $\beta$ parameterizes the level curves of
$d(\lambda) = \tan(\alpha)^{1-\gamma}$. Expressed in these
coordinates, the function \eqref{E:QUOT} takes the form
\begin{equation}
  F(\alpha, \beta) = \frac{
    1-\tan(\alpha)^{1-\gamma} \bigg( 
      c_1^2 \sqrt{\frac{\tan(\beta)}{\tan(\alpha)}}
        + c_2^2 \sqrt{\frac{\cot(\beta)}{\tan(\alpha)}} 
      \bigg)
  }{
    1-\tan(\alpha)^{1-\gamma} \Bigg( 
      (2\gamma-1) - (\gamma-1) \bigg( 
        \sqrt{\frac{\tan(\beta)}{\tan(\alpha)}}
          + \sqrt{\frac{\cot(\beta)}{\tan(\alpha)}}
      \bigg) 
    \Bigg)
  },
\label{E:FALP}
\end{equation}
where we have used that $c_1^2+c_2^2 = 1$. For any $\alpha \in
(0,\pi/2)$ fixed, we find that
\begin{equation}
  \lim_{\beta\rightarrow 0} F(\alpha,\beta) = \frac{c_2^2}{1-\gamma},
  \quad
  \lim_{\beta\rightarrow \pi/2} F(\alpha,\beta) = \frac{c_1^2}{1-\gamma}.
\label{E:BETAL}
\end{equation}
Similarly, for any $\beta \in (0,\pi/2)$ fixed, we have
\[
  \lim_{\alpha\rightarrow 0} F(\alpha,\beta) 
    = \frac{c_1^2 \sqrt{\tan(\beta)} + c_2^2 \sqrt{\cot(\beta)}}
      {(1-\gamma) \Big( \sqrt{\tan(\beta)} + \sqrt{\cot(\beta)} \Big)},
  \quad
  \lim_{\alpha\rightarrow \pi/2} F(\alpha,\beta) = 1.
\]
Notice that $\lim_{\alpha\rightarrow 0} F(\alpha,\beta)$ converges to
the limits in \eqref{E:BETAL} as $\beta \rightarrow 0$ or $\pi/2$.

We now consider the limit $\alpha\rightarrow 0$ with $\beta(\alpha) :=
k\alpha$ for $k>0$. We find that
\[
  \lim_{\alpha\rightarrow 0} F(\alpha,k\alpha) = \frac{c_2^2}{1-\gamma}
  \quad\text{for any $k>0$,}
\]
hence $F(\alpha,\beta)$ can be continuously extended to
$(\alpha,\beta) = (0,0)$ by $c_2^2/(1-\gamma)$. Recall that
$\tan(\theta) \approx \theta$ for small $\theta$. Similarly, we
compute the limit
\[
  \lim_{\beta\rightarrow 0} F(\pi/2-k\beta,\beta) = 1
  \quad\text{for any $k>0$.}
\]
We used that $\tan(\pi/2-k\beta) = \cot(k\beta)$. The behavior of the
map $(\alpha, \beta) \mapsto F(\alpha, \beta)$ at the other corners of
the domain $(0,\pi/2)^2$ can be studied analogously. We conclude that
$F(\alpha, \beta)$ remains bounded for $(\alpha, \beta)$ near the
boundary of $(0, \pi/2)^2$, uniformly in $c=(c_1,c_2)$. It is
continuous up to the boundary expect for the points $(0,\pi/2)$ and
$(\pi/2,\pi/2)$. As long as we stay away from the unique root of the
denominator in \eqref{E:FALP}, the function $F$ is uniformly bounded
as claimed, by continuity. 

The estimate \eqref{E:DISSK} follows from Young inequality.
\end{proof}


\section{Measure-Valued Solutions}

In this section, we use the minimizations in
Sections~\ref{SS:EM}/\ref{SS:MP} to define approximate solutions to
the compressible gas dynamics equations \eqref{E:FULL}, for suitable
initial data and timestep $\tau>0$. We establish uniform bounds and
prove that a subsequence converges to a measure-valued solution of
\eqref{E:FULL} in the limit $\tau\rightarrow 0$. We will cover the
pressureless case and the Euler case simultaneously, with the
understanding that for the pressureless case the internal energy is
set to zero. Similarly, the specific entropy is considered constant in
all cases other than the full Euler case.


\subsection{Approximate Solutions}

We will construct approximate solutions to \eqref{E:FULL} on time
intervals $[0,\infty)$ by successively applying the variational
minimization step introduced in the previous sections and then
utilizing a suitable interpolation between discrete times. Consider
initial density, velocity distribution, and entropy
\[
  \BRHO \in \SP_2(\R^d),
  \quad
  \bar{\MU} \in \SP_{\BRHO}(\R^{2d}),
  \quad
  \bar{\sigma} \in \M_+(\R^d).
\]
Suppose $\INT[\BRHO,\bar{\sigma}] < \infty$ so that $\BRHO =
:\bar{r}\LEB^d$ and $\bar{\sigma} = :\BRHO \bar{S}$ for suitable
Borel functions $\bar{r}, \bar{S}$; see Definition~\ref{D:INT}. Assume
further that $\bar{\MU} =: (\ID, \bar{\BV})\#\BRHO$ with
\begin{equation}\label{E:ZEMOM}
  \bar{\BV} \in \L^2(\R^d,\BRHO)
  \quad\text{satisfying}\quad
  \int_{\R^d} \bar{\BV}(x) \,\BRHO(dx) = 0.
\end{equation}
Notice that since the hyperbolic conservation law \eqref{E:PGD} is
invariant under transformations to a moving reference frame, the
assumption \eqref{E:ZEMOM} is not restrictive.

For later use, let us introduce the initial total energy
\begin{equation}
  \bar{\E} := \int_{\R^d} \HA\BRHO|\bar{\BV}|^2
    + \INT[\BRHO,\bar{\sigma}] < \infty.
\label{E:EINI}
\end{equation}

In order to simplify the notation, in this section we will not indicate the
dependence of various quantities on the timestep $\tau>0$, which will
be arbitrary, but fixed for the following construction. Let $s^k :=
k\tau$ for all $k\in\N_0$. We define
\[
  \RHO^0 := \BRHO,
  \quad
  \MU^0 := \bar{\MU},
  \quad
  \sigma^0 := \bar{\sigma}.
\]
Then we proceed recursively: For any $k\in\N_0$ we define
\begin{alignat*}{8}
  \ORT^{k+1}   &:= \ORT_\tau,   & \quad &
  \ORW^{k+1}   &:= \ORW_\tau,   & \quad &
  \ORU^{k+1}   &:= \ORU_\tau,
\\
  \RHO^{k+1}   &:= \RHO_\tau,   & \quad &
  \MU^{k+1}    &:= \MU_\tau,    & \quad &
  \sigma^{k+1} &:= \sigma_\tau,
\end{alignat*}
with $(\ORT_\tau, \ORW_\tau, \ORU_\tau)$ and $(\RHO_\tau, \MU_\tau,
\sigma_\tau)$ taken from Definitions~\ref{D:EM}/\ref{D:NOTAT}, for the
choice
\[
  \RHO := \RHO^k,   
  \quad
  \MU := \MU^k,    
  \quad
  \sigma := \sigma^k.
\]

By induction in $k$, we observe first that $\MU^k$ is monokinetic for
every $k\in\N_0$. For $k=0$ this follows from our assumption on the
initial data. For $k\GS 1$ we refer the reader to
Definition~\ref{D:EM} and Remark~\ref{R:MONOK}, respectively. Thus
\[
  \BU^k \in \L^2(\R^d,\RHO^k)
  \quad\text{such that}\quad
  \MU^k =: (\ID, \BU^k)\#\RHO^k.
\]
is well-defined. Similarly, from Propositions~\ref{P:EB}/\ref{P:EB2}
and \eqref{E:EINI}, we obtain that
\[
  \E[\RHO^k,\BU^k,\sigma^k]
    = \int_{\R^d} \OH \RHO^k|\BU^k|^2 + \INT[\RHO^k, \sigma^k]
      \LS \bar{\E}
\]
for every $k\in\N_0$. Therefore the following maps are all
well-defined as well:
\[
  \RHO^k =: r^k \LEB^d,
  \quad
  \sigma^k =: \RHO^k S^k.
\]

For $\RHO^k$-a.e.\
$x\in\R^d$ and $k\in\N_0$, we now define
\[
  (\BT^{k+1},W^{k+1},U^{k+1})(x) 
    := (\ORT^{k+1},\ORW^{k+1},\ORU^{k+1})\big( x,\BU^k(x) \big),
\]
which are in $\L^2(\R^d, \RHO^k)$. Rewriting
Propositions~\ref{P:EB}/\ref{P:EB2}, we obtain
\begin{align}
  \E[\RHO^{k+1}, \BU^{k+1}, \sigma^{k+1}]
    & + \int_{\R^d} \Big( {\TST\frac{1}{6}} |W^{k+1}-\BU^k|^2
      + \HA |U^{k+1}-W^{k+1}|^2 \Big) \,\RHO^k(dx)
\nonumber\\
  & + \int_{\R^d} \bigg(
        P(r^k,S^k) \, \DISS^2\big( \nabla\BT^{k+1}-\ONE \big) \Big) \,dx 
    + \TRACE\big( \RES^{k+1}(dx) \big) \bigg)
\nonumber\\
  & \vphantom{\int_{\R^d}} \qquad
    = \E[\rho^k, \BU^k, \sigma^k]
  \quad\text{for all $k\in \N_0$.}
\label{E:DISPP}
\end{align}
Here $\RES^{k+1}$ is the residual tensor corresponding to the
minimizer $\BT^{k+1}$.


\subsection{Interpolation in Time}\label{SS:I1}

In the previous section, we introduced approximate solutions of
\eqref{E:FULL2} at discrete times $s^k := k \tau$ for any timestep
$\tau > 0$. Here we want to interpolate in time to define
functions/measures of time and space.

For this, we could use the path of minimal acceleration
\[
  Y_s(x,\xi) := x + s\xi - \bigg( \frac{s^2}{\tau}-\frac{s^3}{3\tau^2}
    \bigg) \frac{3}{2\tau} \Big( (x+\tau\xi)-\ORT_\tau(x,\xi) \Big)
\]
for location $x\in\R^d$, velocity $\xi\in\R^d$, and $s\in[0,\tau]$,
suitably shifted in $s$. This would be the natural choice in view of
the derivation of the work functional, which featured in our
minimization problem. Instead we prefer to apply the convex
inter\-polation that we have already utilized to derive the
displacement convexity of the internal energy and hence the energy
inequality in Proposition~\ref{P:EB2}. One can show that the
differences between both the positions and velocities of the minimal
acceleration paths and the convex interpolations remain bounded and
vanish as $\tau \rightarrow 0$. Let
\[
  \BT_s(x) 
    := x + (s-s^k) \, V^{k+1}(x),
  \quad 
  V^{k+1}(x) 
    := \DST\frac{\BT^{k+1}(x)-x}{\tau}
\]
for $\RHO^k$-a.e.\ $x\in\R^d$ and $s\in[s^k,s^{k+1})$. Notice that
$\BT_s$ is strictly monotone, therefore invertible for any $s\in[s^k,
s^{k+1})$ since $\BT^{k+1}$ is monotone. Moreover, in the cases with
pressure the map $\BT^{k+1}$ is essentially injective; see
Lemma~\ref{L:CONTROL}. We can therefore track the path of each fluid
element starting from a generic position $\BRX \in \R^d$. By composing
the transport maps of successive timesteps, we define the
transport/velocity
\begin{equation}
  X_s := \BT_s \circ \BT^k \circ \cdots \circ \BT^1,
  \quad
  \Xi_s := \begin{cases}
      W^{k+1} \circ \BT^k \circ \cdots \circ \BT^1
        & \text{if $s \in (s^k, s^{k+1})$,}
\\
      \BU^k \circ \BT^k \circ \cdots \circ \BT^1
        & \text{if $s = s^k$,}
    \end{cases}
\label{E:XXI}
\end{equation}
where $k := \lfloor s/\tau \rfloor$ (the largest integer not bigger
than $s/\tau$). Since $\BT^{l+1}$ is defined only $\RHO^l$-a.e., we
may have to discard a $\RHO^l$-null set, for every $l \in \N_0$. The
preimages of these null sets under the preceding transport maps,
however, can be traced back to a $\BRHO$-negligible set, hence $X_s$
and $\Xi_s$ are well-defined $\BRHO$-a.e. The map $s \mapsto
X_s(\BRX)$ is Lipschitz continuous for $\BRHO$-a.e.\ $\BRX \in \R^d$
because $V^{k+1}$ is finite $\RHO^k$-a.e. To simplify the notation,
let $\BBX_s := (X_s, \Xi_s)$ for all $s\GS 0$. We now define
\begin{equation}
  (\RHO_s, \sigma_s) 
    := \BT_s\#(\RHO^k,\sigma^k),
  \quad
  \MU_s 
    := \begin{cases}
      (\BT_s,W^{k+1})\#\RHO^k
        & \text{if $s\in(s^k,s^{k+1})$,}
\\
      \MU^k
        & \text{if $s=s^k$.}
    \end{cases}
\label{E:INTPO}
\end{equation}
Using the transport/velocity \eqref{E:XXI}, we can express $\MU_s$ by
following the characteristic lines back to the initial data. More
precisely, we have $\MU_s = \BBX_s\# \BRHO$ for $s\GS 0$.

It follows from the proof of Proposition~\ref{P:EB2} that in the cases
with pressure
\[
  \INT[\RHO_s,\sigma_s]
    \LS \INT[\BT_s|\RHO^k,\sigma^k]
    \LS \big( 1-\ell^k(s) \big) \; \INT[\RHO^k, \sigma^k]
      + \ell^k(s) \; \INT[\BT^{k+1}|\RHO^k, \sigma^k]
\]
for every $s\in[s^k,s^{k+1}]$ and $k\in\N_0$. Here $\ell^k(s) :=
(s-s^k)/\tau$. Applying this estimate recursively, we conclude that
$\INT[\RHO_s, \sigma_s]$ remains finite for all $s\GS 0$, thus
\[
  \RHO_s =: r_s\LEB^d,
  \quad
  \sigma_s = \RHO_s S_s
  \quad\text{with}\quad
  S_s := S^k\circ \BT_s^{-1}
\]
for $s \in [s^k,s^{k+1})$. The specific entropy $S_s$ is simply
transported along with the flow. The velocity distribution $\MU_s$ is
monokinetic for all $s\GS 0$, which defines
\[
  \BW_s \in \L^2(\R^d,\RHO_s)
  \quad\text{such that}\quad
  \MU_s =: (\ID, \BW_s)\#\RHO_s.
\]
The case $s=s^k$ has been discussed above. For $s\in(s^k,s^{k+1})$ the
claim follows from $\MU^k$ mono\-kinetic and the invertibility of
$\BT_s$. In the transition from $\BU^k$ to $W^{k+1}$ the kinetic
energy may increase; it remains bounded by $\E[\rho^k, \BU^k,
\sigma^k]$. Over the course of the time interval, the total energy
then decreases so that \eqref{E:DISPP} holds. This explains the
additional factor $2$ on the right-hand side of the following energy
bound:
\begin{equation}
  \E[\RHO_s, \BW_s, \sigma_s] 
    = \int_{\R^d} \OH \RHO_s|\BW_s|^2 + \INT[\RHO_s, \sigma_s]
      \LS 2\bar{\E}
  \quad\text{for all $s\GS 0$.}
\label{E:MNPP}
\end{equation}
For any $s \in [s^k,s^{k+1})$, we define the transport velocity
\begin{equation}
  \BV_s := V^{k+1}\circ \BT_s^{-1}
  \quad\text{so that}\quad
  \dot X_s = \BV_s\circ X_s.
\label{E:UTAUMM2}
\end{equation}
Because of \eqref{E:WTQQ}, we have
\begin{equation}
  \int_{\R^d} |\BV_s(x)|^2 \,\RHO_s(dx)
    \LS \frac{2}{3} \int_{\R^d} |W^{k+1}(x)|^2 \,\RHO^k(dx)
     +\frac{1}{3} \int_{\R^d} |\BU^k(x)|^2 \,\RHO^k(dx),
\label{E:ENBB}
\end{equation}
which can be bounded in terms of $\bar{\E}$ for all $s\GS 0$; see
\eqref{E:MNPP}.

The momentum $\BM_s := \RHO_s \BW_s$ has zero mean: We can write
\[
  \int_{\R^d} \BV_s(z) \,\RHO_s(dz)
    = \int_{\R^d} W^{k+1}(x) \,\RHO^k(dx)
    = \int_{\R^d} \BU^k(x) \,\RHO^k(dx)
\]
for all $s\in(s^k,s^{k+1})$; see Remarks~\ref{R:MOMMY}/\ref{R:MOM2} and
\eqref{E:WTQQ}. Similarly, we have
\begin{align*}
  \int_{\R^d} \BV_s(z) \,\RHO_s(dz)
    & = \int_{\R^d} U^{k+1}(x) \,\RHO^k(dx)
      = \int_{\R^d} W^{k+1}(x) \,\RHO^k(dx)
\\
    & = \int_{\R^d} \BU^k(x) \,\RHO^k(dx)
\end{align*}
for $s=s^{k+1}$. We have used that the barycentric projection
preserves the momentum and that $U^{k+1}=W^{k+1}$ in the cases with
pressure. Applying this identity recursively and using assumption
\eqref{E:ZEMOM}, we obtain the result. 


\subsection{Regularity in Time}

In the following, we will use the subscript $\tau$ to indicate
explicitly the dependence of various quantities on the timestep
$\tau>0$.

\begin{lemma}\label{L:UBDE}
For suitable initial data $(\BRHO, \bar{\BV}, \bar{\sigma})$ and
$\tau>0$, consider approximate solutions $(\RHO_\tau, \BV_\tau,
\BW_\tau, \sigma_\tau)$ as defined in Section~\ref{SS:I1}. For any
$T>0$ it holds
\[ 
  \sup_{\tau>0} \|\RHO_\tau\|_{\LIP([0,T]; \SP_2(\R^d))}
    \LS (2\bar{\E})^{1/2}.
\]
The second moments remain finite: for all $s\in[0,T]$ we have
\begin{equation}
  \sup_\tau \bigg( \int_{\R^d} |x|^2 \RHO_{\tau,s}(dx) \bigg)^{1/2}
    \LS \bigg( \int_{\R^d} |x|^2 \BRHO(dx) \bigg)^{1/2}
      + s (2\bar{\E})^{1/2}.
\label{E:MOM2}
\end{equation}
For any sequence $\tau_n \longrightarrow 0$, there exist a subsequence
(not relabeled, for simplicity of notation) and a map $\RHO \in
\LIP([0,T]; \SP_2(\R^d))$ such that
\begin{equation}
	\RHO_{\tau_n,s} \WEAK \RHO_s
	\quad\text{narrowly as $n\rightarrow\infty$, for all $s\in[0,T]$.}
\label{E:WEKMM}
\end{equation}

An analogous statement holds for the entropy density $\sigma_\tau$.
Moreover, for the limit map $\sigma \in \LIP([0,T]; \M_\ENT(\R^d))$ we
can write $\sigma_s =: \RHO_s S_s$ with
\[
	S_s \in \L_+^\infty(\R^d, \RHO_s)
	\quad\text{for all $s\in[0,T].$}
\]
\end{lemma}

\begin{proof}
We divide the proof into three steps.

\medskip

\textbf{Step~1.} Consider first $s_1<s_2$ with $s_1, s_2 \in
[s_\tau^k, s_\tau^{k+1})$ for some $k \in \N_0$. Since $\BT_{\tau,s}
(x) = x + (s-s_\tau^k) V_\tau^{k+1}(x)$ for $\RHO_\tau^k$-a.e.\
$x\in\R^d$ and $s\in[s_\tau^k, s_\tau^{k+1})$, we get
\begin{align}
  \WAS_2(\RHO_{\tau,s_2}, \RHO_{\tau,s_1})^2 
    &\LS \int_{\R^d} |\BT_{\tau,s_2}(x)-\BT_{\tau,s_1}(x)|^2 
      \,\RHO_\tau^k(dx)
\nonumber\\
    &= (s_2-s_1)^2 \int_{\R^d} |V_\tau^{k+1}(x)|^2 
      \,\RHO_\tau^k(dx).
\label{E:WAEST}
\end{align}
For every $s \in (t_\tau^k, t_\tau^{k+1})$, we have
\[
  \int_{\R^d} |V_\tau^{k+1}(x)|^2 \,\RHO_\tau^k(dx)
    = \int_{\R^d} |\BV_{\tau,s}(x)|^2  \,\RHO_{\tau,s}(dx)
\]
(see \eqref{E:INTPO}/\eqref{E:UTAUMM2}), which is bounded uniformly in
$\tau,k$ because of \eqref{E:ENBB} and \eqref{E:MNPP}. The estimate
\eqref{E:WAEST} remains true also for $s_2 = s_\tau^{k+1}$, by
continuity.

\medskip

\textbf{Step~2.} Consider now $0\LS s_1<s_2$ with the property that
there exists at least one $k\in\N$ with $s_1\LS s_\tau^k<s_2$. We use
the triangle inequality to estimate
\[
  \WAS_2(\RHO_{\tau,s_2}, \RHO_{\tau,s_1})
    \LS \WAS_2(\RHO_{\tau,s_2}, \RHO_\tau^{k_2})
      + \sum_{k=k_1+1}^{k_2-1} \WAS_2(\RHO_\tau^{k+1}, \RHO_\tau^k)
      + \WAS_2(\RHO_\tau^{k_1+1}, \RHO_{\tau,s_1}),
\]
where $k_i := \lfloor s_i/\tau \rfloor$ for $i=1..2$. For each term,
we can now apply the estimate from Step~1. Summing up all
contributions, we obtain the inequality
\[
  \WAS_2(\RHO_{\tau,s_2}, \RHO_{\tau,s_1})
    \LS |s_2-s_1| (2\bar{\E})^{1/2}
  \quad\text{for all $0\LS s_1<s_2$.}
\]
To control the second moments, we write
\begin{align*}
  & \bigg( \int_{\R^d} |z|^2 \,\RHO_{\tau,s_2}(dz) \bigg)^{1/2}
    = \bigg( \int_{\R^{2d}} |z|^2 \,\GAMMA(dx,dz) \bigg)^{1/2}
\\
  &\qquad
    \LS \bigg( \int_{\R^{2d}} |x|^2 \,\GAMMA(dx,dz) \bigg)^{1/2}
      + \bigg( \int_{\R^{2d}} |z-x|^2 \,\GAMMA(dx,dz) \bigg)^{1/2}
\\
  &\qquad
    = \bigg( \int_{\R^d} |x|^2 \,\RHO_{\tau,s_1}(dx) \bigg)^{1/2}
      + \WAS_2(\RHO_{\tau,s_2}, \RHO_{\tau,s_1}),
\end{align*}
with $\GAMMA \in \SP_2(\R^{2d})$ an optimal transport plan connecting
$\RHO_{\tau,s_1}$ and $\RHO_{\tau,s_2}$.

The uniform bound \eqref{E:MOM2} implies that the family
$\{\RHO_{\tau,s}\}_\tau$ is tight, thus precompact with respect to
narrow convergence, for any $s\in[0,T]$. We can then apply
Arzel\`{a}-Ascoli theorem to conclude; see Proposition~3.3.1 in
\cite{AmbrosioGigliSavare2008}, for example.

\medskip

\textbf{Step~3.} The statement for $\sigma_\tau$ follows analogously.
Note that the specific entropy $S_{\tau,s}$ is simply transported
along the flow and hence bounded in $\L_+^\infty(\R^d,
\RHO_{\tau,s})$. This implies, in particular, that $\sigma_s$ must be
absolutely continuous with respect to $\RHO_s$.
\end{proof}

\begin{lemma}\label{L:UBDE2}
For suitable initial data $(\BRHO, \bar{\BV}, \bar{\sigma})$ and
$\tau>0$, consider approximate solutions $(\RHO_\tau, \BV_\tau,
\BW_\tau, \sigma_\tau)$ as defined in Section~\ref{SS:I1}. Let 
$T>0$ be given. For any sequence $\tau_n \longrightarrow 0$, there exist a
subsequence (not relabeled, for simplicity of notation) and a map $\BM
\in \LIP([0,T]; \MK(\R^d; \R^d))$ with the property that
\begin{equation}
	\|\BM_{\tau_n,s}-\BM_{s}\|_{\MK(\R^d)} \longrightarrow 0
	\quad\text{as $n\rightarrow \infty$, for a.e.\ $s\in[0,T]$.}
\label{E:CONVO2}
\end{equation}
Here $\BM_{\tau,s} := \RHO_{\tau,s} \BW_{\tau,s}$ for all $\tau,s$. If
the subsequence $\tau_n \longrightarrow 0$ is such that statement
\eqref{E:WEKMM} of Lemma~\ref{L:UBDE} holds as well, then we have, for
a.e.\ $s \in[0,T]$, that

\begin{equation}
	\BM_s =: \RHO_s \BV_s
	\quad\text{with}\quad
	\BV_s \in \L^2(\R^d,\RHO_s).
\label{E:ABSMOM}
\end{equation}
\end{lemma}

\begin{proof}
We divide the proof into three steps.

\medskip

\textbf{Step~1.} Consider first $s_1<s_2$ with $s_1, s_2 \in
[s_\tau^k, s_\tau^{k+1})$ for some $k \in \N_0$. Since $\BT_{\tau,s}
(x) = x + (s-s_\tau^k) V_\tau^{k+1}(x)$ for $\RHO_\tau^k$-a.e.\
$x\in\R^d$ and $s\in[s_\tau^k, s_\tau^{k+1})$, we get
\begin{align*}
	& \int_{\R^d} \zeta(z) \cdot 
		\Big( \BM_{\tau,s_2}(dz)-\BM_{\tau,s_1}(dz) \Big)
\\
	& \qquad
		= \int_{\R^d} \Big( \zeta\big( \BT_{\tau,s_2}(x) \big)
			-\zeta\big( \BT_{\tau,s_1} \big) \Big) \cdot W_\tau^{k+1}(x)
				\,\RHO_\tau^k(dx)
\\
	& \qquad
		\LS |s_2-s_1| \bigg( \int_{\R^d} |V_\tau^{k+1}(x)|^2
				\,\RHO_\tau^k(dx) \bigg)^{1/2}
			\bigg( \int_{\R^d} |W_\tau^{k+1}(x)|^2
				\,\RHO_\tau^k(dx) \bigg)^{1/2},
\end{align*}
for any $\zeta \in \BL_1(\R^d; \R^d)$. We have used that the Lipschitz
constant of $\zeta$ is bounded by $1$; see Definition~\ref{D:KANTOO}.
Consider now $\varphi \in \C^1_c(\R^d)$ with $\eta(\R^d) \subset
[0,1]$ and
\[
	\text{$\varphi(x) = 1$ if $|x| \LS 1$,}
	\quad
	\text{$\varphi(x) = 0$ if $|x| \GS 2$.}
\]
For any $R,\EPS>0$ we define the rescaled cut-off function/mollifier
\[
	\eta_R(x) := \varphi(x/R),
	\quad
	\varphi_\EPS(x) := \EPS^{-d} \varphi(x/\EPS)
\]
for all $x\in\R^d$. Then we can decompose
\begin{align}
	& \int_{\R^d} \zeta(x) \cdot \Big( W_\tau^{k+1}(x)-\BU_\tau^k(x) \Big) 
		\,\RHO_\tau^k(dx)
\nonumber\\
	& \qquad
		= \int_{\R^d} \big( 1-\eta_R(x) \big) \zeta(x) \cdot 
			\Big( W_\tau^{k+1}(x)-\BU_\tau^k(x) \Big) \,\RHO_\tau^k(dx)
\label{E:OPKL}\\
	& \qquad\quad
		+ \int_{\R^d} \Big( \eta_R(x)\zeta(x)-\zeta_{R,\EPS}(x) \Big) 
			\cdot \Big( W_\tau^{k+1}(x)-\BU_\tau^k(x) \Big) \,\RHO_\tau^k(dx)
\nonumber\\
	& \qquad\quad
		+ \int_{\R^d} \zeta_{R,\EPS}(x)
			\cdot \Big( W_\tau^{k+1}(x)-\BU_\tau^k(x) \Big) \,\RHO_\tau^k(dx)
\nonumber
\end{align}
with $\zeta_{R,\EPS} := (\eta_R\zeta) \star \varphi_\EPS$. The first
term on the right-hand side of \eqref{E:OPKL} satisfies
\begin{align}
	& \bigg| \int_{\R^d} \big( 1-\eta_R(x) \big) \zeta(x) \cdot 
		\Big( W_\tau^{k+1}(x)-\BU_\tau^k(x) \Big) \,\RHO_\tau^k(dx) \bigg|
\label{E:AL1}\\
	& \qquad
		\LS \frac{C}{R} \bigg( \int_{\R^d} |x|^2 \,\RHO_\tau^k(dx) \bigg)^{1/2}
\nonumber\\
	& \qquad\qquad \times \Bigg\{ 
		\bigg( \int_{\R^d} |W_\tau^{k+1}|^2 \,\RHO_\tau^k(dx) \bigg)^{1/2}
			+ \bigg( \int_{\R^d} |\BU_\tau^k|^2 \,\RHO_\tau^k(dx) \bigg)^{1/2}
				\Bigg\}
\nonumber
\end{align}
with constant $C$ depending on the $\sup$-norm of $\zeta$. Recall that
the second moment of $\RHO_\tau^k$ is bounded uniformly in $\tau, k$,
as shown in Lemma~\ref{L:UBDE}. Moreover, the terms in curly brackets
are uniformly bounded because of \eqref{E:ENBB} and \eqref{E:MNPP}. We
conclude that \eqref{E:AL1} vanishes as $R\rightarrow\infty$,
uniformly in $\tau, k$. We observe that
\[
	\| \eta_R\zeta-\zeta_{R,\EPS} \|_{\L^\infty(\R^d)}
		\longrightarrow 0
	\quad\text{as $\EPS \rightarrow 0$,}
\]
by standard properties of mollication. Notice that $\eta_R\zeta$ has
compact support in $\R^d$. Arguing as above, we find that the second
term in \eqref{E:OPKL} also converges to zero, uniformly in $\tau, k$
and $R$, as $\EPS \rightarrow 0$. Finally, we have the identity
\begin{align}
	& \int_{\R^d} \zeta_{R,\EPS}(x)
		\cdot \Big( W_\tau^{k+1}(x)-\BU_\tau^k(x) \Big) \,\RHO_\tau^k(dx)
\label{E:AL2}\\
	& \qquad 
      = \tau \int_{\R^d} \nabla\zeta_{R,\EPS}(x) : \Big( \PK_\tau^k(x) \,dx 
        + \RES_\tau^{k+1}(dx) \Big),
\nonumber
\end{align}
with $\RES_\tau^{k+1}$ is the residual tensor corresponding to
$\BT_\tau^{k+1}$ and pressure term
\[
	\PK_\tau^k(dx) := P\big( r_\tau^k(x),S_\tau^k(x) \big) 
		\; \det\big( \nabla\BT^{k+1,\S}_\tau(x)
    \big)^{1-\gamma} \big( \nabla\BT^{k+1,\S}_\tau(x) \big)^{-1} \,dx.
\]
We use \eqref{E:SIGID2} (with $x = \BT_\tau^{k+1}(x)-\tau
V_\tau^{k+1}(x)$) and \eqref{E:HGAA2} to get
\begin{align*}
	& \int_{\R^d} \Big( \TRACE\big( \PK_\tau^k(x) \big) \,dx 
  		+ \TRACE\big( \RES_\tau^{k+1}(dx) \big) \Big)
    		\LS d(\gamma-1) \; \INT[\BT_\tau^{k+1}|\RHO_\tau^k,\sigma_\tau^k]
\\
	&\qquad 
    	+ \bigg( \int_{\R^d} |W_\tau^{k+1}(x)-\BU_\tau^k(x)|^2 
      		\,\RHO_n^k(dx) \bigg)^{1/2} \bigg( \int_{\R^d} 
        		|V_\tau^{k+1}(x)|^2 \,\RHO_\tau^k(dx) \bigg)^{1/2},
\end{align*}
which can be bounded in terms of $\bar{\E}$, thus uniformly in $\tau,
k$, because of the energy balance \eqref{E:DISPP}. The $\sup$-norm of
$\nabla\zeta_{R,\EPS}$ in \eqref{E:AL2} can be bounded uniformly in
$R\GS 1$ and $\EPS>0$, by choice of $\eta_R$ and because $\zeta \in
\BL_1(\R^d; \R^d)$.

Collecting all terms and letting first $R \rightarrow \infty$, then
$\EPS \rightarrow 0$, we conclude that
\begin{equation}
	\|\BM_{\tau,s_2}-\BM_{\tau,s_1}\|_{\MK(\R^d)} 
		\LS C \big( |s_2-s_1|+\tau \big)
	\quad\text{for all $s_1, s_2\in[0,T]$,}
\label{E:LOCL}
\end{equation}
for some constant $C$ that can bounded in terms of $\bar{\E}$, hence
uniformly in $\tau$. The additional $\tau$ on the right-hand side of
\eqref{E:LOCL} occurs since the jumps in $\BM_\tau$ at discrete times
$t_\tau^k$ are always of order $\tau$, not fractions of timesteps.

\medskip

\textbf{Step~2.} Using \eqref{E:LOCL}, we conclude that for any choice
of times $t_0 \LS t_1 \LS \ldots \LS t_m$ contained in $[0,T]$, we can
bound the variation uniformly in $\tau$ as
\[
	\sum_{i=1}^m \|\BM_{\tau,t_{i-1}}-\BM_{\tau,t_i}\|_{\MK(\R^d)}
		\LS C\big( (t_m-t_0) + \tau \big).
\]
Therefore the map $t \mapsto \BM_{\tau,t} \in \MK(\R^d; \R^d)$ is of
uniform bounded variation.

By Cauchy-Schwarz inequality, for each $s\GS 0$ we can estimate
\[
	\int_{\R^d} |\BM_{\tau,s}(dx)| 
		\LS \bigg( \int_{\R^d} |\BW_{\tau,s}(x)|^2 \,\RHO_{\tau,s}(dx) 
			\bigg)^{1/2},
\]
which is uniformly bounded because of \eqref{E:MNPP}. The
Monge-Kantorovich norm can be controled by the total variation, which
implies that $\{\BM_{\tau,s}\}_\tau$ is precompact in the dual space
$\BL(\R^d; \R^d)^*$, for all $s\GS 0$. On the other hand, a sequence
of $\R^d$-valued measures with uniformly bounded total variation,
converging in Monge-Kantorovich norm, has as limit again a measure;
see Theorem~3.2 in \cite{HilleSzarekWormZiemlanska2017}. We now apply
Helly's theorem in the form of Theorem~2.3 in
\cite{FleischerPorter2001} to obtain \eqref{E:CONVO2}.

\medskip

\textbf{Step~3.} The limit map $s \mapsto \BM_s \in \MK(\R^d; \R^d)$
satisfies the inequality
\[
	\|\BM_{s_2}-\BM_{s_1}\|_{\MK(\R^d)} 
		\LS C |s_2-s_1|
	\quad\text{for all $s_1, s_2\in[0,T]$}
\]
(recall \eqref{E:CONVO2} with $\tau\rightarrow 0$) and is therefore
Lipschitz continuous, as claimed. Moreover, if $\RHO$ and $\BM$ are
obtained from the same sequence $\tau_n\rightarrow 0$, then by lower
semicontinuity of the kinetic energy functional with respect to narrow
convergence of density and momentum, we conclude that $\BM_s$ must be
absolutely continuous with respect to $\RHO_t$, which proves the
decomposition \eqref{E:ABSMOM} for all $s\in[0,T]$.
\end{proof}


\subsection{Compactification}

We will need compactifications of the state space.

\begin{lemma}\label{L:COMPACT}
Let $X$ be a completely regular space and $\FS \subset \C(X,I)$, with
$I := [0,1]$, a set of continuous functions that separates points and
closed sets: for every closed set $E \subset X$ and every $x \in
X\setminus E$, there exists $\Phi \in \FS$ with $\Phi(u) \not\in
\overline{\Phi(E)}$. Then there exist a compact Hausdorff space
$\BFX$ and an embedding $e\colon X\longrightarrow\BFX$ such that
$e(X)$ is dense in $\BFX$. Moreover, for any $\Phi \in \FS$, the
composition $\Phi\circ e^{-1}\colon e(X) \longrightarrow \R$ has a
continuous extension to all of $\BFX$. If $\FS$ is countable, then
$\BFX$ is metrizable.
\end{lemma}

\begin{proof}
Consider the product space $I^\FS$, which is compact in the product
topology, by Tykhonov's theorem. Let the map $e\colon X
\longrightarrow I^\FS$ be defined by
\[
  \pi_\Phi\big(e(u)\big) := \Phi(u)
  \quad\text{for all $u \in X$ and $\Phi\in\FS$,}
\]
where $\pi_\Phi\colon I^\FS \longrightarrow I$ denotes the projection
onto the $\Phi$-component. Since $\FS$ separates points and closed
sets, the map $e$ is in fact an embedding (a homeomorphism between $X$
and its image, with $e(X)$ given the relative topology of $I^\FS$). We
refer the reader to Proposition~4.53 of \cite{Folland1999} for a
proof. We now define $\BFX$ to be the closure of $e(X)$ in $I^\FS$.
Being a closed subset of a compact Hausdorff space, the set $\BFX$ is
itself compact and Hausdorff. The set $e(X)$ is dense in $\BFX$, by
construction. We denote by $\ALG$ the smallest closed subalgebra in
$\CB(X)$ containing $\FS$. For any $\Phi\in\ALG$, there exists a
continuous extension of $\Phi\circ e^{-1}$ to all of $\BFX$; see
Proposition~4.56 in \cite{Folland1999}. If $\FS$ is countable, then
the set $I^\FS$ is metrizable. Therefore, since every subset of a
metrizable space is metrizable, we obtain that $\BFX$ is metrizable.
We refer the reader to Section~4.8 of \cite{Folland1999} for
additional information on compactifications.
\end{proof}

For simplicity of notation, we will identify $X$ with its image
$e(X)$. Then every function $\Phi\in\ALG$ can be extended as a
continuous function on $\BFX$. Notice that such an extension is
uniquely determined because $e(X)$ is dense in $\BFX$. We denote by
$\C(\BFX)$ the space of all extensions obtained this way, and we will
use the same symbols to indicate functions in $\ALG$ and their
extensions in $\C(\BFX)$.


\subsection{Young Measures}\label{SS:YM}

We will use Young measures to capture the behavior of weakly
convergent sequences of approximate solutions of the compressible
Euler equations \eqref{E:FULL}. Recall that we assumed the specific
entropy $S$ to be non-negative and bounded at initial time. Since $S$
is simply transported along with the flow, the same is true for all
times, thus $S \in [0,S_{\max}]$ for some $S_{\max} \GS 0$. The state
space for density, velocity, and specific entropy $(\RHO, \BV, S)$ is
therefore given by
\begin{equation}
	X := [0,\infty) \times \R^d \times [0,S_{\max}].
\label{E:STATE}
\end{equation}
Equipped with the usual topology, it is a completely regular space.

Our goal is to define a suitable compactification of the state space.
Equivalently, we must specify the set of continuous and bounded
functions on $X$, for which we need to be able to describe weak limits
of compositions with approximate solutions. Let us first consider a
function that represents the total energy and mass. In slight abuse of
notation, we use the same symbols $(\RHO, \BV, S)$ for elements in
$X$. Let
\begin{equation}
  h(\RHO, \BV, S) 
    := \RHO + \Big( \HA \RHO |\BV|^2 + U(\RHO,S) \Big);
\label{E:FEN}
\end{equation}
see Definition~\ref{D:INT}. We now introduce the set
\[
  \W(X) := \left\{ 
    \begin{aligned}
      \varphi + \Bigg( 
        c_\RHO \cdot \begin{pmatrix} \RHO \\ \RHO\BV \end{pmatrix}
        + c_\sigma \cdot \begin{pmatrix} \RHO S \\ \RHO\BV S \end{pmatrix}
        + c_K : \RHO\BV\otimes\BV
        + c_U U(\RHO,S) 
      \Bigg) / h \colon \;\; &
\\
      \text{ $\varphi \in \C_0(X)$, 
        $c_\RHO,c_\sigma \in \R^{d+1}$,
        $c_K \in \SYM{d}$,
        $c_U \in \R$} \; &
    \end{aligned} \right\}.
\]
One can check that the functions in $\W(X)$ are continuous and
bounded. For this, it is convenient to introduce a parameterization of
the state space $X$, similarly to the construction in the the proof of
Lemma~\ref{L:COMML}. One possible choice is
\[
  \RHO(\alpha) := \tan(\alpha),
  \quad
  \BV(u) := \frac{u}{\sqrt{1-|u|^2}}
\]
for $\alpha \in [0,\pi/2)$ and $u \in B := \{ u\in\R^d \colon |u| < 1
\}$. For any $\Phi \in \W(X)$, the map
\[
  (\alpha, u, S) \mapsto \Phi\big( \RHO(\alpha), \BV(u), S) 
\]
can be extended to a bounded function on the compact set $[0,\pi/2]
\times \bar{B} \times [0,S_{\max}]$. This extension may be
discontinuous at parts of the boundary. One can also check that the
set $\W(X)$ is a closed \emph{separable} vector space with respect to the
$\sup$-norm. To this end, notice that $\W(X)$ is a finite-dimensional
augmentation of the vector space $\C_0(X)$, which is known to be
separable; see also Lemma~2 in \cite{BlountKouritzin2010}.

There exists a countable set $\FS$ that is dense in
$\W(X)\cap\C(X,I)$, $I=[0,1]$, and separates points and closed sets.
Indeed consider any closed set $E\subset X$ and any point $u \in
X\setminus E$. One can find a $\Psi \in \C_0(X,I)$ with $\Psi(u)=1$
and $\Psi|E\equiv 0$, and since $\FS$ is dense there exists
$\Phi\in\FS$ with $\|\Phi-\Psi\|_{\C(X)} < \EPS$ for some
$0<\EPS<1/2$. Applying Lemma~\ref{L:COMPACT}, we obtain a
compactification $\BFX$ (a compact, metrizable Hausdorff space) of
\eqref{E:STATE}. The closed subalgebra $\ALG$ in Lemma~\ref{L:COMPACT}
contains the set $\W(X)$.

Recall that $\dot\R^d$ is the one-point compactification of $\R^d$;
see Section~\ref{SS:PC}. Then
\[
  \BBE := \L^1\big( [0,\infty),\C(\dot\R^d\times\BFX) \big)
\]
is the space of (equivalence classes of) measurable maps $\PHI \colon
[0,\infty) \longrightarrow \C(\dot\R^d\times\BFX)$ (i.e., pointwise
limits of sequences of simple functions) with finite norm:
\[
  \|\PHI\|_\BBE 
    := \int_0^\infty \|\PHI(s,\cdot)\|_{\C(\dot\R^d\times\BFX)} \,ds 
    < \infty.
\]
Notice that $\BFX$ is compact and metrizable, hence separable. One can
then show that $\BBE$ is a separable Banach space. Its topological
dual is given by
\[
  \BBE^* := \L^\infty_w\big( [0,\infty), \M_+(\dot\R^d\times\BFX) \big),
\]
the space of (equivalence classes of) $\nu \colon [0,\infty)
\longrightarrow \M_+(\dot\R^d\times\BFX)$ with
\begin{gather*}
  \text{$s \mapsto \int_{\dot\R^d\times\BFX} \phi(x, \FX) \,\nu_s(dx, d\FX)$ 
    measurable for all $\phi \in \C(\dot\R^d\times\BFX)$, and}
\\
  \|\nu\|_{\BBE^*} 
    := \ESUP_{s\in[0,\infty)} \|\nu_s\|_{\M(\dot\R^d\times\BFX)} 
    < \infty
\end{gather*}
(we write $s \mapsto \nu_s$ and $\FX := (\RHO, \BV, S) \in \BFX$). The
duality is induced by the pairing
\begin{equation}
  \langle \nu,\PHI\rangle := \int_0^\infty \int_{\dot\R^d\times\BFX} 
    \PHI(s, x, \FX) \,\nu_s(dx, d\FX) \,ds
\label{E:DUPA}
\end{equation}
for $\PHI \in \BBE$ and $\nu \in \BBE^*$. Bounded closed balls in
$\BBE^*$ endowed with the weak* topology are metrizable and
(sequentially) compact, by Banach-Alaoglu theorem.

For any timestep $\tau>0$, we now define $\nu^1_\tau \in \BBE^*$ by
\begin{align}
  & \int_{\dot\R^d\times\BFX} \phi(x,\FX) \,\nu^1_{\tau,s}(dx,d\FX)
\label{E:UPS}\\
  & \qquad
    := \int_{\R^d} \phi\big( x,r_{\tau,s}(x), \BW_{\tau,s}(x), 
      S_{\tau,s}(x) \big) \, h\big( r_{\tau,s}(x), \BW_{\tau,s}(x), 
      S_{\tau,s}(x) \big) \,dx
\nonumber
\end{align}
for all $\phi \in \C(\dot\R^d\times\BFX)$ and $s\GS 0$; see
\eqref{E:FEN}. As usual, we have
\[
  \RHO_{\tau,s} =: r_{\tau,s} \LEB^d,
  \quad
  \sigma_{\tau,s} =: \RHO_{\tau,s} S_{\tau,s},
\]
with approximate solutions $(\RHO_\tau, \BV_\tau, \sigma_\tau)$
constructed in Section~\ref{SS:I1}. Because of \eqref{E:MNPP}, the
family $\{ \nu^1_\tau \}_{\tau>0}$ is uniformly bounded in $\BBE^*$:
We have
\[
  \|\nu^1_{\tau,s}\|_{\M(\dot\R^d\times\BFX)} 
    = \int_{\dot\R^d\times\BFX} \nu^1_{\tau,s}(dx, d\FX)
    = 1 + \E[\RHO_{\tau,s}, \BW_{\tau,s}, \sigma_{\tau,s}]
\]
for $\tau>0$ and $s\GS 0$. Notice that $\nu^1_{\tau,s}$ is
non-negative. From this, we get the relative (sequential) compactness
of $\{ \nu^1_\tau \}_{\tau>0}$ with respect to the weak* topology.

\medskip

In summary, we have the following result:

\begin{proposition}[Young Measure]\label{P:YUM}
Consider a sequence $\tau_n \longrightarrow 0$ for $n \rightarrow
\infty$ and let $\nu^1_n := \nu^1_{\tau_n} \in \BBE^*$ be defined by
\eqref{E:UPS}. Then there exist $\nu^1 \in \BBE^*$ and a subsequence
(still denoted by $\{\nu^1_n\}_n$ for simplicity) with the property
that
\begin{align*}
  & \lim_{n\rightarrow\infty} \int_0^\infty \int_{\R^d} 
    \PHI\big( s,x,r_{n,s}(x), \BW_{n,s}(x), S_{n,s}(x) \big) \, 
      h\big( r_{n,s}(x), \BW_{n,s}(x), S_{n,s}(x) \big) \,dx \,ds
\\
  & \qquad
    = \int_0^\infty \int_{\dot\R^d\times\BFX} \PHI(s, x, \FX) 
      \,\nu^1_s(dx, d\FX) \,ds
  \quad\text{for all $\PHI \in \BBE$.}
\end{align*}
To simplify the notation, we have used the subscript $n$ instead of
$\tau_n$. We write
\begin{equation}
  \lbrack f(\RHO,\BV,S) \rbrack_s(dx) 
    := \int_{\BFX} f(\FX)/h(\FX) \,\nu^1_s(dx, d\FX)
  \quad\text{for a.e.\ $s\GS 0$}
\label{E:BRACK}
\end{equation}
and for any $f \colon X \longrightarrow \R$ such that $f/h \in \ALG$;
see Lemma~\ref{L:COMPACT}.
\end{proposition}

\begin{remark}\label{R:YUM}
In the same way, we define a second Young measure $\nu^2$, using
piecewise constant instead of piecewise linear interpolation in time.
More precisely, for any $\tau>0$, using the same notation as
above, we define $\nu^2_\tau \in \BBE^*$ by
\begin{align*}
	& \int_{\dot\R^d\times\BFX} \phi(x,\FX) \,\nu^2_{\tau,s}(dx,d\FX)
\\
	& \qquad
		:= \int_{\R^d} \phi\big( x,r_{\tau,s_\tau}(x), \BW_{\tau,s_\tau}(x), 
      		S_{\tau,s_\tau}(x) \big) \, h\big( r_{\tau,s_\tau}(x), 
      			\BW_{\tau,s_\tau}(x), S_{\tau,s_\tau}(x) \big) \,dx
\end{align*}
for all $\phi \in \C(\dot\R^d\times\BFX)$ and $s\GS 0$, where $s_\tau
:= \lfloor s/\tau \rfloor \tau$ denotes the largest integer multiple
of $\tau$ less than or equal to $s$. Passing to the limit along a
suitable sequence $\tau_n \longrightarrow 0$, we obtain the Young
measure $\nu^2 \in \BBE^*$, which again can be used to capture
concentrations/oscillations in weakly convergent sequences of
approximate solutions of \eqref{E:FULL}; see Proposition~\ref{P:YUM}.
Similar to \eqref{E:UPS}, we use double brackets $\llbracket \cdot
\rrbracket$ to indicate the pairing of $\nu^2$ with suitable functions
of $(\RHO,\BV,S)$.
\end{remark}



\subsection{Global Existence}

In this section, we establish the global existence of measure-valued
solutions to \eqref{E:FULL}, using the results of Sections~\ref{SS:I1}
and \ref{SS:YM}.

\begin{proof}[Proof of Theorem~\ref{T:GLOBAL}]
We consider a sequence of timesteps $\tau_n \longrightarrow 0$ as
$n\rightarrow \infty$ with the property that the pointwise in time
convergence in Lemmas~\ref{L:UBDE}/\ref{L:UBDE2} holds and that the
approximate Young measures in Proposition~\ref{P:YUM} and
Remark~\ref{R:YUM} converge to $\nu^1$ and $\nu^2$ along
$\{\tau_n\}_n$. We will use the notation introduced in
Sections~\ref{SS:I1} and \ref{SS:YM}, but with subscript $n$ in place
of $\tau_n$, for simplicity.

By construction, it holds that
\[
  \left.\begin{array}{r}
    \partial_s \RHO_n
      +\nabla\cdot(\RHO_n\BV_n) = 0
\\[1ex]
    \partial_s \sigma_n
      + \nabla\cdot(\sigma_n\BV_n) = 0
  \end{array}\right\}
  \quad\text{in $\Big( \C^1_c\big( [0,T) \big)\otimes\AF \Big)^*$.}
\]
Recall that density $\RHO_{n,s}$ and entropy $\sigma_{n,s}$ have
finite second moments for all $s\GS 0$. Passing to the limit
$n\rightarrow \infty$, we get the first and third equations in
\eqref{E:FULL2}.

It remains to prove the momentum equation. We observe that
\begin{align}
	& -\int_{\R^d} \eta(0) \zeta(x) \cdot \bar{\BV}(x) \,\bar\RHO(dx)
		= \int_0^T \frac{d}{ds} \bigg( \int_{\R^d} \eta(s) \zeta(x) 
			\cdot \BW_{n,s}(x) \, \RHO_{n,s}(dx) \bigg) \, ds 
\nonumber\\
	& \qquad
		= \sum_{k\in\N_0} \int_{s_n^k}^{s_n^{k+1}} \int_{\R^d}
			\eta'(s) \zeta(z) \cdot \BW_{n,s}(z) \,\RHO_{n,s}(dz) \,ds
\label{E:DCPO}\\
	& \qquad\quad
   		+ \sum_{k\in\N_0} \int_{s_n^k}^{s_n^{k+1}} \int_{\R^d} 
      		\eta(s) \nabla\zeta(z) : \Big( \BW_{n,s}(z) \otimes \BV_{n,s}(z) 
      			\Big) \,\RHO_{n,s}(dz) \,ds
\nonumber\\
  & \qquad\quad
    + \sum_{k\in\N_0} \tau_n
      \int_{\R^d} \eta(s_n^k) \nabla\zeta(x) : \Big( \PK_n^k(x) \,dx 
        + \RES_n^{k+1}(dx) \Big)
\nonumber
\end{align}
for any $\eta \in \C^1_c([0,T))$ and $\zeta \in \AF$. We proceed in
four steps.

\medskip

\textbf{Step~1.} In the first term of the right-hand side of
\eqref{E:DCPO}, we directly apply \eqref{E:UPS} and pass to the limit.
Using the definition of $\BV$ in Lemma~\ref{L:UBDE2}, we get that
\begin{align*}
	& \lim_{n\rightarrow\infty} \sum_{k\in\N_0} \int_{s_n^k}^{s_n^{k+1}} 
		\int_{\R^d} \eta'(s) \zeta(z) 
			\cdot \BW_{n,s}(z) \,\RHO_{n,s}(dz) \,ds
\\
	& \qquad
		= \int_0^\infty \int_{\R^d} \eta'(s) \zeta(z) 
			\cdot \BV_s(z) \,\RHO_s(dz) \,ds.
\end{align*}
Recall that the momentum $\BM_{n,s} := \RHO_{n,s}\BV_{n,s}$ has finite
first moment for all $s\GS 0$.

\medskip

\textbf{Step~2.} In the second term on the right-hand side of
\eqref{E:DCPO}, we need to replace the transport velocity $\BV_{n,s}$
by the transported velocity $\BW_{n,s}$ because the approximative
Young measure \eqref{E:UPS} only captures the latter. We first rewrite
\begin{align*}
	& \int_{\R^d} \nabla_z\zeta(z) : \Big( \BW_{n,s}(z) \otimes 
		\big( \BV_{n,s}(z)-\BW_{n,s}(z) \big) \Big) \,\RHO_{n,s}(dz)
\\
	& \qquad
		= \int_{\R^d} \nabla_z\zeta\big( \BT_{n,s}(x) \big) : \Big( 
			W^{k+1}(x) \otimes \big( V^{k+1}(x)-W^{k+1}(x) \big) \Big)
				\,\RHO_n^k(dx)
\end{align*}
for all $s \in (s_n^k, s_n^{k+1})$ and $k\in\N_0$. Then we apply
\eqref{E:DISSK} to estimate
\begin{align}
	& \bigg| \sum_{k\in\N_0} \int_{s_n^k}^{s_n^{k+1}} \int_{\R^d}
		\eta(s) \nabla_z\zeta(z) : \Big( \BW_{n,s}(z) \otimes 
			\big( \BV_{n,s}(z)-\BW_{n,s}(z) \big) \Big) \,\RHO_{n,s}(dz)
				\,ds \bigg|
\nonumber\\
	& \qquad
		\LS C \bigg\{ \EPS T \; 
			\max_{k\in\N_0} \int_{\R^d} |W_n^{k+1}(x)|^2 \,\RHO_n^k(dx)
\label{E:FORTH}\\
	& \qquad\qquad\qquad 
		+ C_\EPS \tau_n \; \sum_{k\in\N_0} 
			\int_{\R^d} |W_n^{k+1}(x)-\BU^k(x)|^2 \,\RHO_n^k(dx) \bigg\}
\nonumber
\end{align}
for any $\EPS>0$ and suitable constant $C_\EPS$. Here $C$ depends on
the $\sup$-norm of $\eta\nabla_z\zeta$, which is finite; recall
\eqref{E:CSTO}. Both the $\max$ and sum on the right-hand side of
\eqref{E:FORTH} are bounded by $\bar\E$ uniformly in $n$, because of
energy equality \eqref{E:DISPP} and \eqref{E:WTQQ}. Since $\EPS$ was
arbitrary, we find that the left-hand side of \eqref{E:FORTH} vanishes
as $n\rightarrow \infty$.

We can now write
\[
	\sum_{k\in\N_0} \int_{s_n^k}^{s_n^{k+1}} \int_{\R^d}
		\eta(s) \nabla\zeta(z) : \Big( \BW_{n,s}(z) \otimes \BW_{n,s}(z) 
			\Big) \,\RHO_{n,s}(dz) \,ds
    = \langle \nu_n,\PHI \rangle
\]
(recall definition \eqref{E:DUPA} of the dual pairing), with test
function
\[
	\PHI(s, x, \FX) := \eta(s) \nabla\zeta(x) : \RHO(\BW \otimes \BW)
		/ h(\RHO,\BW,S)
\]
for all $s\GS 0$, $x\in\dot\R^d$, and $\FX = (\RHO, \BW, S) \in \BFX$.
From Proposition~\ref{P:YUM}, we get
\begin{align*}
	& \lim_{n\rightarrow\infty} \sum_{k\in\N_0}
		\int_{s_n^k}^{s_n^{k+1}} \int_{\R^d} \eta(s) \nabla\zeta(z) : 
			\Big( \BW_{n,s}(z) \otimes \BW_{n,s}(z) 
				\Big) \,\RHO_{n,s}(dz) \,ds
\\
  & \qquad
    = \int_0^\infty \int_{\dot\R^d} \eta(s) \nabla\zeta(x) : 
      \lbrack \RHO\BV\otimes\BV \rbrack_s(dx) \,ds.
\end{align*}
We refer the reader to \eqref{E:BRACK} for notation.

\medskip

\textbf{Step~3.} In the pressure term in \eqref{E:DCPO}, we want to
replace
\[
  \det\big( \nabla\BT_n^{k+1,\S}(x) \big)^{1-\gamma}
    \big( \nabla\BT_n^{k+1,\S}(x) \big)^{-1}
\]
by $\ONE$. Using \eqref{E:DISSL}, we obtain the estimate
\begin{align}
  & \bigg| \tau_n \sum_{k\in\N_0}
    \int_{\R^d} \eta(s_n^k) \nabla\zeta(x) : \Big( \PK_n^k(x)-P\big( r_n^k(x), 
      S_n^k(x) \big) \ONE \Big) \,dx \bigg|
\nonumber\\
  & \qquad
    \LS C \bigg\{ \EPS T \; 
      \max_{k\in\N_0} \int_{\R^d} U\big( r_n^k(x), S_n^k(x) \big) \,dx
\label{E:FVTH}\\
  & \qquad\qquad\qquad 
    + C_\EPS \tau_n \; \sum_{k\in\N_0} 
      \int_{\R^d} P\big( r_n^k(x), S_n^k(x) \big) 
        D_\INT\big( \nabla\BT_n^{k+1,\S}(x)-\ONE \big) \,dx \bigg\}
\nonumber
\end{align}
for any $\EPS>0$ and suitable constant $C_\EPS$. Here $C$ depends on
the $\sup$-norm of $\eta\nabla\zeta$, which is bounded. Both $\max$
and sum on the right-hand side of \eqref{E:FVTH} are bounded by
$\bar\E$ uniformly in $n$, because of the energy equality
\eqref{E:DISPP}. Since $\EPS$ was arbitrary, we conclude that the
left-hand side of \eqref{E:FVTH} vanishes as $n\rightarrow \infty$.

Similarly, we can estimate
\begin{align}
	& \bigg| \sum_{k\in\N_0} \int_{\R^d} 
  		\bigg( \tau_n \eta(s_n^k)-\int_{s_n^k}^{s_n^{k+1}} \eta(s) \,ds \bigg)
			\nabla\zeta(x) : \Big( P\big( r_n^k(x), S_n^k(x) \big) \ONE \Big) 
				\,dx \bigg|
\nonumber\\
  & \qquad
    \LS CT \omega(\tau_n,\eta) \; 
      \max_{k\in\N_0} \int_{\R^d} U\big( r_n^k(x), S_n^k(x) \big) \,dx,
\label{E:SXTH}
\end{align}
with $C$ depending on the $\sup$-norm of $\nabla\zeta$ and modulus of
continuity
\[
	\omega(\tau_n,\eta) := \sup_{\substack{s_1,s_2 \in [0,T] \\ 
		|s_2-s_1| \LS \tau_n}} |\eta(s_2)-\eta(s_1)|.
\]
The $\max$ on the right-hand side of \eqref{E:SXTH} is bounded by
$\bar{\E}$ uniformly in $n$, because of the energy balance
\eqref{E:DISPP}. The left-hand side therefore vanishes as
$n\rightarrow \infty$.

We can now write
\begin{align*}
	& \sum_{k\in\N_0}
		\int_{s_n^k}^{s_n^{k+1}} \int_{\R^d} \eta(s) \nabla\zeta(z) : 
			\Big( P\big( r_n^k(z), S_n^k(z) \big) \ONE \Big) \,dz \,ds
	    = \langle \nu_n,\PHI \rangle
\end{align*}
(recall definition \eqref{E:DUPA} of the dual pairing), with test
function
\[
	\PHI(s, x, \FX) := \eta(s) \nabla\zeta(x) : 
		\Big( P(\RHO,S)\ONE \Big)/h(\RHO,\BW,S)
\]
for all $s\GS 0$, $x\in\dot\R^d$, and $\FX = (\RHO, \BW, S) \in \BFX$.
From Remark~\ref{R:YUM}, we get
\begin{align*}
	& \lim_{n\rightarrow\infty} \sum_{k\in\N_0} 
		\int_{s_n^k}^{s_n^{k+1}} \int_{\R^d} \eta(s) \nabla\zeta(z) : 
			\Big( P\big( r_n^k(z), S_n^k(z) \big) \ONE \Big) \,dz \,ds
\\
  & \qquad
    = \int_0^\infty \int_{\dot\R^d} \eta(s) \nabla\zeta(z) : 
      \llbracket P(\RHO,S)\ONE \rrbracket_s(dz) \,ds.
\end{align*}
We refer the reader to \eqref{E:BRACK} for notation.

\medskip

\textbf{Step~4.} The residual term in \eqref{E:DCPO} can be estimated
as
\[
	\bigg| \tau_n \sum_{k\in\N_0} \int_{\R^d} \eta(s_n^k) \nabla\zeta(x) 
  		: \RES_n^{k+1}(dx) \bigg|
	   		\LS C \tau_n \sum_{k\in\N_0} \int_{\R^d} 
				\TRACE\big( \RES_n^{k+1}(dx) \big),
\]
with the sum on the right-hand side bounded by $\bar{\E}$ uniformly in
$n$, because of the energy balance \eqref{E:DISPP}. Here $C$ is some
constant depending on the $\sup$-norm of $\eta\nabla\zeta$. The
left-hand side therefore vanishes as $n \rightarrow \infty$; see also
Remark~\ref{R:RESA}.

Combining Steps~1--4, we have proved the momentum equation.
\end{proof}



\begin{bibdiv}
\begin{biblist}

\bib{AlbertiAmbrosio1999}{article}{ 
    AUTHOR = {Alberti, G.}, 
    AUTHOR = {Ambrosio, L.},
     TITLE = {A geometrical approach to monotone functions in 
              {$\boldsymbol{R}\sp n$}},
   JOURNAL = {Math. Z.},
    VOLUME = {230}, 
      YEAR = {1999}, 
    NUMBER = {2}, 
     PAGES = {259--316}, 
}

\bib{AmbrosioFuscoPallara2000}{book}{
    AUTHOR = {Ambrosio, L.},
    AUTHOR = {Fusco, N.},
    AUTHOR = {Pallara, D.},
     TITLE = {Functions of bounded variation and free discontinuity
              problems},
    SERIES = {Oxford Mathematical Monographs},
 PUBLISHER = {The Clarendon Press Oxford University Press},
   ADDRESS = {New York},
      YEAR = {2000},
}

\bib{AmbrosioGigli2008}{article}{
    AUTHOR = {Ambrosio, L.},
    AUTHOR = {Gigli, N.},
     TITLE = {Construction of the parallel transport in the {W}asserstein
              space},
   JOURNAL = {Methods Appl. Anal.},
    VOLUME = {15},
      YEAR = {2008},
    NUMBER = {1},
     PAGES = {1--29},
}

\bib{AmbrosioGigliSavare2008}{book}{ 
    AUTHOR = {Ambrosio, L.}, 
    AUTHOR = {Gigli, N.},
    AUTHOR = {Savar\'{e}, G.}, 
     TITLE = {Gradient Flows in Metric Spaces and in the Space of Probability 
              Measures}, 
    SERIES = {Lectures in Mathematics}, 
 PUBLISHER = {Birkh\"{a}user Verlag}, 
   ADDRESS = {Basel}, 
      YEAR = {2008}, 
}

\bib{AmbrosioTrevisan2014}{article}{
    AUTHOR = {Ambrosio, L.},
    AUTHOR = {Trevisan, D.},
     TITLE = {Well-posedness of {L}agrangian flows and continuity equations
              in metric measure spaces},
   JOURNAL = {Anal. PDE},
    VOLUME = {7},
      YEAR = {2014},
    NUMBER = {5},
     PAGES = {1179--1234},
}

\bib{BauschkeWang2009}{article}{
    AUTHOR = {Bauschke, H. H.},
    AUTHOR = {Wang, X.},
     TITLE = {The kernel average for two convex functions and its
              application to the extension and representation of monotone
              operators},
   JOURNAL = {Trans. Amer. Math. Soc.},
    VOLUME = {361},
      YEAR = {2009},
    NUMBER = {11},
}

\bib{BlountKouritzin2010}{article}{
    AUTHOR = {Blount, D.},
    AUTHOR = {Kouritzin, M. A.},
     TITLE = {On convergence determining and separating classes of
              functions},
   JOURNAL = {Stochastic Process. Appl.},
    VOLUME = {120},
      YEAR = {2010},
    NUMBER = {10},
     PAGES = {1898--1907},
}


\bib{BouchitteButtazzoDePascale2010}{article}{
    AUTHOR = {Bouchitt\'{e}, G.},
    AUTHOR = {Buttazzo, G.},
    AUTHOR = {De Pascale, L.},
     TITLE = {The {M}onge-{K}antorovich problem for distributions and
              applications},
   JOURNAL = {J. Convex Anal.},
    VOLUME = {17},
      YEAR = {2010},
    NUMBER = {3-4},
     PAGES = {925--943},
}

\bib{BouchutJames1995}{article}{
    AUTHOR = {Bouchut, F.},
    AUTHOR = {James, F.},
     TITLE = {\'{E}quations de transport unidimensionnelles \`a coefficients
              discontinus},
   JOURNAL = {C. R. Acad. Sci. Paris S\'er. I Math.},
    VOLUME = {320},
      YEAR = {1995},
    NUMBER = {9},
     PAGES = {1097--1102},
}
	
\bib{BouchutJames1999}{article}{
    AUTHOR = {Bouchut, F.},
    AUTHOR = {James, F.},
     TITLE = {Duality solutions for pressureless gases, monotone scalar
              conservation laws, and uniqueness},
   JOURNAL = {Comm. Partial Differential Equations},
    VOLUME = {24},
      YEAR = {1999},
    NUMBER = {11-12},
     PAGES = {2173--2189},
}


\bib{Brenier2009}{article}{
    AUTHOR = {Brenier, Y.},
     TITLE = {{$L^2$} formulation of multidimensional scalar conservation
              laws},
   JOURNAL = {Arch. Ration. Mech. Anal.},
    VOLUME = {193},
      YEAR = {2009},
    NUMBER = {1},
     PAGES = {1--19},
}

\bib{BrenierGangboSavareWestdickenberg2013}{article}{
    AUTHOR = {Brenier, Y.},
    AUTHOR = {Gangbo, W.},
    AUTHOR = {Savar{\'e}, G.},
    AUTHOR = {Westdickenberg, M.},
     TITLE = {Sticky particle dynamics with interactions},
   JOURNAL = {J. Math. Pures Appl. (9)},
    VOLUME = {99},
      YEAR = {2013},
    NUMBER = {5},
     PAGES = {577--617},
}

\bib{BrenierGrenier1998}{article}{
    AUTHOR = {Brenier, Y.},
    AUTHOR = {Grenier, E.},
     TITLE = {Sticky particles and scalar conservation laws},
   JOURNAL = {SIAM J. Numer. Anal.},
    VOLUME = {35},
      YEAR = {1998},
    NUMBER = {6},
     PAGES = {2317--2328 (electronic)},
}




\bib{BressanNguyen2014}{article}{
    AUTHOR = {Bressan, A.},
    AUTHOR = {Nguyen, T.},
     TITLE = {Non-existence and non-uniqueness for multidimensional sticky
              particle systems},
   JOURNAL = {Kinet. Relat. Models},
    VOLUME = {7},
      YEAR = {2014},
    NUMBER = {2},
     PAGES = {205--218},
}


\bib{BrezinaFeireisl2017}{article}{
    AUTHOR = {Brezina, J.},
    AUTHOR = {Eduard, E.},
     TITLE = {Maximal dissipation principle for the complete Euler system},
      YEAR = {2017},
    EPRINT = {arXiv:1712.04761 [math.AP]},
}

\bib{Buckmaster2015}{article}{
    AUTHOR = {Buckmaster, T.}, 
     TITLE = {Onsager's conjecture almost everywhere in time}, 
   JOURNAL = {Comm. Math. Phys.}, 
     YEAR = {2015}, 
   NUMBER = {3}, 
    PAGES = {1175--1198}, 
} 
	
\bib{BuckmasterDeLellisSzekelyhidi2016}{article}{
     AUTHOR = {Buckmaster, T.},
     AUTHOR = {De Lellis, C.},
     AUTHOR = {Sz\'ekelyhidi, Jr., L.},
      TITLE = {Dissipative {E}uler flows with {O}nsager-critical spatial
               regularity},
    JOURNAL = {Comm. Pure Appl. Math.},
     VOLUME = {69},
       YEAR = {2016},
     NUMBER = {9},
      PAGES = {1613--1670},
}

\bib{CastaingRaynaudDeFitteValadier2004}{book}{
    AUTHOR = {Castaing, C.},
    AUTHOR = {Raynaud de Fitte, P.},
    AUTHOR = {Valadier, M.},
     TITLE = {Young measures on topological spaces},
    SERIES = {Mathematics and its Applications},
    VOLUME = {571},
      NOTE = {With applications in control theory and probability theory},
 PUBLISHER = {Kluwer Academic Publishers, Dordrecht},
      YEAR = {2004},
}

\bib{CavallettiSedjroWestdickenberg2014}{article}{
    AUTHOR = {Cavalletti, F.},
    AUTHOR = {Sedjro, M.},
    AUTHOR = {Westdickenberg, M.},
     TITLE = {A Simple Proof of Global Existence for the 1D Pressureless
              Gas Dynamics Equations},
   JOURNAL = {SIAM Math. Anal.},
    VOLUME = {44},
      YEAR = {2015},
    NUMBER = {1},
     PAGES = {66--79},
}

\bib{CavallettiWestdickenberg2014}{article}{
    AUTHOR = {Cavalletti, F.},
    AUTHOR = {Westdickenberg, M.},
     TITLE = {The polar cone of the set of monotone maps},
   JOURNAL = {Proc. Amer. Math. Soc.},
    VOLUME = {143},
      YEAR = {2015},
     PAGES = {781--787},
}

\bib{Chen2000}{incollection}{
    AUTHOR = {Chen, G.-Q.},
     TITLE = {Compactness methods and nonlinear hyperbolic conservation
              laws},
 BOOKTITLE = {Some current topics on nonlinear conservation laws},
    SERIES = {AMS/IP Stud. Adv. Math.},
    VOLUME = {15},
     PAGES = {33--75},
 PUBLISHER = {Amer. Math. Soc., Providence, RI},
      YEAR = {2000},
}

\bib{ChenLeFloch2000}{article}{
    AUTHOR = {Chen, G.-Q.},
    AUTHOR = {LeFloch, P. G.},
     TITLE = {Compressible {E}uler equations with general pressure law},
   JOURNAL = {Arch. Ration. Mech. Anal.},
    VOLUME = {153},
      YEAR = {2000},
    NUMBER = {3},
     PAGES = {221--259},
}

\bib{ChenPerepelitsa2010}{article}{
    AUTHOR = {Chen, G.-Q.},
    AUTHOR = {Perepelitsa, M.},
     TITLE = {Vanishing viscosity limit of the {N}avier-{S}tokes equations
              to the {E}uler equations for compressible fluid flow},
   JOURNAL = {Comm. Pure Appl. Math.},
    VOLUME = {63},
      YEAR = {2010},
    NUMBER = {11},
     PAGES = {1469--1504},
}

\bib{Chiodaroli2014}{article}{
    AUTHOR = {Chiodaroli, E.},
     TITLE = {A Counterexample to Well-Posedness of Entropy Solutions to the
              Compressible Euler System},
   JOURNAL = {Accepted for publication in J. Hyperbolic Differ. Equ.},
      YEAR = {2014},
}

\bib{ChiodaroliKreml2014}{article}{
    AUTHOR = {Chiodaroli, E.},
    AUTHOR = {Kreml, O.},
     TITLE = {On the Energy Dissipation Rate of Solutions to the Compressible
              Isentropic Euler System},
   JOURNAL = {Accepted for publication in Arch. Ration. Mech. Anal.},
      YEAR = {2014},
}

\bib{ChiodaroliDeLellisKreml2014}{article}{
    AUTHOR = {Chiodaroli, E.},
    AUTHOR = {De Lellis, C.},
    AUTHOR = {Kreml, O.},
     TITLE = {Global Ill-Posedness of the Isentropic System of Gas Dynamics},
   JOURNAL = {Accepted for publication in Comm. Pure App. Math.},
      YEAR = {2014},
}

\bib{ChitescuMikulescuNitaIoana2016}{article}{
    AUTHOR = {Chitescu, I.},
    AUTHOR = {Miculescu, R.},
    AUTHOR = {Nita, L.},
    AUTHOR = {Ioana, L.},
     TITLE = {Monge-{K}antorovich norms on spaces of vector measures},
   JOURNAL = {Results Math.},
    VOLUME = {70},
      YEAR = {2016},
    NUMBER = {3-4},
     PAGES = {349--371},
}

\bib{DeLellisSzekelyhidi2009}{article}{
    AUTHOR = {De Lellis, C.},
    AUTHOR = {Sz{\'e}kelyhidi, Jr., L.},
     TITLE = {The {E}uler equations as a differential inclusion},
   JOURNAL = {Ann. of Math. (2)},
    VOLUME = {170},
      YEAR = {2009},
    NUMBER = {3},
     PAGES = {1417--1436},
}

\bib{DeLellisSzekelyhidi2010}{article}{
    AUTHOR = {De Lellis, C.},
    AUTHOR = {Sz{\'e}kelyhidi, Jr., L.},
     TITLE = {On admissibility criteria for weak solutions of the {E}uler
              equations},
   JOURNAL = {Arch. Ration. Mech. Anal.},
    VOLUME = {195},
      YEAR = {2010},
    NUMBER = {1},
     PAGES = {225--260},
}

\bib{Dafermos1973}{article}{
    AUTHOR = {Dafermos, C. M.},
     TITLE = {The entropy rate admissibility criterion for solutions of
              hyperbolic conservation laws},
   JOURNAL = {J. Differential Equations},
    VOLUME = {14},
      YEAR = {1973},
     PAGES = {202--212},
}

\bib{DemouliniStuartTzavaras2001}{article}{
    AUTHOR = {Demoulini, S.},
    AUTHOR = {Stuart, D. M. A.},
    AUTHOR = {Tzavaras, A. E.},
     TITLE = {A variational approximation scheme for three-dimensional
              elastodynamics with polyconvex energy},
   JOURNAL = {Arch. Ration. Mech. Anal.},
    VOLUME = {157},
      YEAR = {2001},
    NUMBER = {4},
     PAGES = {325--344},
}

\bib{DemouliniStuartTzavaras2012}{article}{
    AUTHOR = {Demoulini, S.},
    AUTHOR = {Stuart, D. M. A.},
    AUTHOR = {Tzavaras, A. E.},
     TITLE = {Weak-strong uniqueness of dissipative measure-valued solutions
              for polyconvex elastodynamics},
   JOURNAL = {Arch. Ration. Mech. Anal.},
    VOLUME = {205},
      YEAR = {2012},
    NUMBER = {3},
     PAGES = {927--961},
}

\bib{DeGiorgi1993}{incollection}{
    AUTHOR = {De Giorgi, E.},
     TITLE = {New problems on minimizing movements},
 BOOKTITLE = {Boundary value problems for partial differential equations and
              applications},
    SERIES = {RMA Res. Notes Appl. Math.},
    VOLUME = {29},
     PAGES = {81--98},
 PUBLISHER = {Masson, Paris},
      YEAR = {1993},
}

\bib{Dermoune2005}{article}{
    AUTHOR = {Dermoune, A.},
     TITLE = {{$d$}-dimensional pressureless gas equations},
   JOURNAL = {Teor. Veroyatn. Primen.},
    VOLUME = {49},
      YEAR = {2004},
    NUMBER = {3},
     PAGES = {610--614},
}

\bib{DingChenLuo1987}{article}{
    AUTHOR = {Ding, X. X.},
    AUTHOR = {Chen, G.-Q.},
    AUTHOR = {Luo, P. Z.},
     TITLE = {Convergence of the {L}ax-{F}riedrichs scheme for the system of
              equations of isentropic gas dynamics. {I}},
   JOURNAL = {Acta Math. Sci. (Chinese)},
    VOLUME = {7},
      YEAR = {1987},
    NUMBER = {4},
     PAGES = {467--480},
}

\bib{DingChenLuo1988}{article}{
    AUTHOR = {Ding, X. X.},
    AUTHOR = {Chen, G.-Q.},
    AUTHOR = {Luo, P. Z.},
     TITLE = {Convergence of the {L}ax-{F}riedrichs scheme for the system of
              equations of isentropic gas dynamics. {II}},
   JOURNAL = {Acta Math. Sci. (Chinese)},
    VOLUME = {8},
      YEAR = {1988},
    NUMBER = {1},
     PAGES = {61--94},
}

\bib{DiPerna1983}{article}{
    AUTHOR = {DiPerna, R. J.},
     TITLE = {Convergence of the viscosity method for isentropic gas
              dynamics},
   JOURNAL = {Comm. Math. Phys.},
    VOLUME = {91},
      YEAR = {1983},
    NUMBER = {1},
     PAGES = {1--30},
}

\bib{Dudley1966}{article}{
    AUTHOR = {Dudley, R. M.},
     TITLE = {Convergence of {B}aire measures},
   JOURNAL = {Studia Math.},
    VOLUME = {27},
      YEAR = {1966},
     PAGES = {251--268},
}

\bib{ERykovSinai1996}{article}{
    AUTHOR = {E, W.},
    AUTHOR = {Rykov, Yu. G.},
    AUTHOR = {Sinai, Ya. G.},
     TITLE = {Generalized variational principles, global weak solutions and
              behavior with random initial data for systems of conservation
              laws arising in adhesion particle dynamics},
   JOURNAL = {Comm. Math. Phys.},
    VOLUME = {177},
      YEAR = {1996},
    NUMBER = {2},
     PAGES = {349--380},
}

\bib{Feireisl2014}{article}{
    AUTHOR = {Feireisl, E.},
     TITLE = {Maximal dissipation and well-posedness for the compressible
              {E}uler system},
   JOURNAL = {J. Math. Fluid Mech.},
    VOLUME = {16},
      YEAR = {2014},
    NUMBER = {3},
     PAGES = {447--461},
}


\bib{FeireislGwiazdaSwierczewskaGwiazdaWiedemann2017}{article}{
    AUTHOR = {Feireisl, E.},
    AUTHOR = {Gwiazda, P.},
    AUTHOR = {\'Swierczewska-Gwiazda, A.},
    AUTHOR = {Wiedemann, E.}, 
     TITLE = {Regularity and energy conservation for the compressible 
              {E}uler equations}, 
   JOURNAL = {Arch. Ration. Mech. Anal.}, 
    VOLUME = {223}, 
      YEAR = {2017}, 
    NUMBER = {3}, 
     PAGES = {1375--1395}, 
} 



\bib{FjordholmKaeppeliMishraTadmor2014}{article}{
    AUTHOR = {Fjordholm, U. S.},
    AUTHOR = {Kaeppeli, R.},
    AUTHOR = {Mishra, S.},
    AUTHOR = {Tadmor, E.},
     TITLE = {Construction of approximate entropy measure valued
              solutions for hyperbolic systems of conservation laws},
   JOURNAL = {Preprint},
      YEAR = {2014},
}

\bib{FleischerPorter2001}{article}{
    AUTHOR = {Fleischer, I.},
    AUTHOR = {Porter, J. E.},
     TITLE = {Convergence of metric space-valued {BV} functions},
   JOURNAL = {Real Anal. Exchange},
    VOLUME = {27},
      YEAR = {2001/02},
    NUMBER = {1},
     PAGES = {315--319},
}

\bib{Folland1999}{book}{
    AUTHOR = {Folland, G. B.},
     TITLE = {Real analysis},
    SERIES = {Pure and Applied Mathematics (New York)},
   EDITION = {Second},
      NOTE = {Modern techniques and their applications,
              A Wiley-Interscience Publication},
 PUBLISHER = {John Wiley \& Sons Inc.},
   ADDRESS = {New York},
      YEAR = {1999},
}
	
\bib{GangboNguyenTudorascu2009}{article}{
    AUTHOR = {Gangbo, W.},
    AUTHOR = {Nguyen, T.},
    AUTHOR = {Tudorascu, A.},
     TITLE = {Euler-{P}oisson systems as action-minimizing paths in the
              {W}asserstein space},
   JOURNAL = {Arch. Ration. Mech. Anal.},
    VOLUME = {192},
      YEAR = {2009},
    NUMBER = {3},
     PAGES = {419--452},
}

\bib{GangboWestdickenberg2009}{article}{
    AUTHOR = {Gangbo, W.},
    AUTHOR = {Westdickenberg, M.},
     TITLE = {Optimal transport for the system of isentropic
              Euler equations},
   JOURNAL = {Comm. PDE},
    VOLUME = {34},
      YEAR = {2009},
    NUMBER = {9},
     PAGES = {1041--1073},
}

\bib{Ghoussoub2008}{article}{
    AUTHOR = {Ghoussoub, N.},
     TITLE = {A variational theory for monotone vector fields},
   JOURNAL = {J. Fixed Point Theory Appl.},
    VOLUME = {4},
      YEAR = {2008},
    NUMBER = {1},
     PAGES = {107--135},
}

\bib{Gigli2004}{thesis}{
    AUTHOR = {Gigli, N.},
     TITLE = {On the geometry of the space of probability measures
              endowed with the quadratic optimal transport distance},
      TYPE = {Ph.D. Thesis},
      YEAR = {2004},
}

\bib{Grenier1995}{article}{
    AUTHOR = {Grenier, E.},
     TITLE = {Existence globale pour le syst\`eme des gaz sans pression},
   JOURNAL = {C. R. Acad. Sci. Paris S\'er. I Math.},
    VOLUME = {321},
      YEAR = {1995},
    NUMBER = {2},
     PAGES = {171--174},
}


\bib{Higham1988}{article}{
    AUTHOR = {Higham, N. J.},
     TITLE = {Computing a nearest symmetric positive semidefinite matrix},
   JOURNAL = {Linear Algebra Appl.},
    VOLUME = {103},
      YEAR = {1988},
     PAGES = {103--118},
}

\bib{HilleSzarekWormZiemlanska2017}{article}{
    AUTHOR = {Hille, S. C.},
    AUTHOR = {Szarek, T.},
    AUTHOR = {Worm, D. T. H. },
    AUTHOR = {Ziemlan\'{n}ska, M. A.},
     TITLE = {On a {S}chur-like property for spaces of measures},
      YEAR = {2017},
    EPRINT = {arXiv:1703.00677 [math.FA]},
}

\bib{HuangWang2001}{article}{
    AUTHOR = {Huang, F.},
    AUTHOR = {Wang, Z.},
     TITLE = {Well posedness for pressureless flow},
   JOURNAL = {Comm. Math. Phys.},
    VOLUME = {222},
      YEAR = {2001},
    NUMBER = {1},
     PAGES = {117--146},
}

\bib{Isett2017}{book}{ 
    AUTHOR = {Isett, P.}, 
     TITLE = {H\"older continuous {E}uler flows in three dimensions with 
              compact support in time}, 
    SERIES = {Annals of Mathematics Studies}, 
    VOLUME = {196}, 
 PUBLISHER = {Princeton University Press, Princeton, NJ}, 
      YEAR = {2017}, 
} 

\bib{JordanKinderlehrerOtto1998}{article}{
    AUTHOR = {Jordan, R.},
    AUTHOR = {Kinderlehrer, D.},
    AUTHOR = {Otto, F.},
     TITLE = {The variational formulation of the {F}okker-{P}lanck 
              equation},
   JOURNAL = {SIAM J. Math. Anal.},
    VOLUME = {29},
      YEAR = {1998},
    NUMBER = {1},
     PAGES = {1--17},
}

\bib{KristensenRindler2010}{article}{
    AUTHOR = {Kristensen, J.},
    AUTHOR = {Rindler, F.},
     TITLE = {Characterization of generalized gradient {Y}oung measures
              generated by sequences in {$W^{1,1}$} and {BV}},
   JOURNAL = {Arch. Ration. Mech. Anal.},
    VOLUME = {197},
      YEAR = {2010},
    NUMBER = {2},
     PAGES = {539--598},
}

\bib{LeFlochWestdickenberg2007}{article}{
    AUTHOR = {LeFloch, P. G.},
    AUTHOR = {Westdickenberg, M.},
     TITLE = {Finite energy solutions to the isentropic {E}uler equations
              with geometric effects},
   JOURNAL = {J. Math. Pures Appl. (9)},
    VOLUME = {88},
      YEAR = {2007},
    NUMBER = {5},
     PAGES = {389--429},
}

\bib{LimYuGlimmLiSharp2008}{article}{
    AUTHOR = {Lim, H.},
    AUTHOR = {Yu, Y.},
    AUTHOR = {Glimm, J.},
    AUTHOR = {Li, X. L.},
    AUTHOR = {Sharp, D. H.},
     TITLE = {Chaos, transport and mesh convergence for fluid mixing},
   JOURNAL = {Acta Math. Appl. Sin. Engl. Ser.},
    VOLUME = {24},
      YEAR = {2008},
    NUMBER = {3},
     PAGES = {355--368},
}

\bib{LimIwerksGlimmSharp2010}{article}{
    AUTHOR = {Lim, H.},
    AUTHOR = {Iwerks, J.},
    AUTHOR = {Glimm, J.},
    AUTHOR = {Sharp, D. H.},
     TITLE = {Nonideal {R}ayleigh-{T}aylor mixing},
   JOURNAL = {Proc. Natl. Acad. Sci. USA},
    VOLUME = {107},
      YEAR = {2010},
    NUMBER = {29},
     PAGES = {12786--12792},
}

\bib{LionsPerthameSouganidis1996}{article}{
    AUTHOR = {Lions, P.-L.},
    AUTHOR = {Perthame, B.},
    AUTHOR = {Souganidis, P. E.},
     TITLE = {Existence and stability of entropy solutions for the
              hyperbolic systems of isentropic gas dynamics in {E}ulerian
              and {L}agrangian coordinates},
   JOURNAL = {Comm. Pure Appl. Math.},
    VOLUME = {49},
      YEAR = {1996},
    NUMBER = {6},
     PAGES = {599--638},
}

\bib{LionsPerthameTadmor1994}{article}{
    AUTHOR = {Lions, P.-L.},
    AUTHOR = {Perthame, B.},
    AUTHOR = {Tadmor, E.},
     TITLE = {Kinetic formulation of the isentropic gas dynamics and
              {$p$}-systems},
   JOURNAL = {Comm. Math. Phys.},
    VOLUME = {163},
      YEAR = {1994},
    NUMBER = {2},
     PAGES = {415--431},
}

\bib{Mandelkern1989}{article}{
    AUTHOR = {Mandelkern, M.},
     TITLE = {Metrization of the one-point compactification},
   JOURNAL = {Proc. Amer. Math. Soc.},
    VOLUME = {107},
      YEAR = {1989},
    NUMBER = {4},
     PAGES = {1111--1115},
}
	
\bib{Moutsinga2008}{article}{
    AUTHOR = {Moutsinga, O.},
     TITLE = {Convex hulls, sticky particle dynamics and pressure-less gas
              system},
   JOURNAL = {Ann. Math. Blaise Pascal},
    VOLUME = {15},
      YEAR = {2008},
    NUMBER = {1},
     PAGES = {57--80},
}

\bib{Nash1954}{article}{
    AUTHOR = {Nash, J.},
     TITLE = {$C^1$ isometric imbeddings},
   JOURNAL = {Ann. of Math.},
    VOLUME = {60},
      YEAR = {1954},
    NUMBER = {2},
     PAGES = {383--396},
}

\bib{Nash1956}{article}{
    AUTHOR = {Nash, J.},
     TITLE = {The imbedding problem for Riemannian manifolds},
   JOURNAL = {Ann. of Math.},
    VOLUME = {63},
      YEAR = {1956},
    NUMBER = {2},
     PAGES = {20--63},
}

\bib{NatileSavare2009}{article}{
    AUTHOR = {Natile, L.},
    AUTHOR = {Savar{\'e}, G.},
     TITLE = {A {W}asserstein approach to the one-dimensional sticky
              particle system},
   JOURNAL = {SIAM J. Math. Anal.},
    VOLUME = {41},
      YEAR = {2009},
    NUMBER = {4},
     PAGES = {1340--1365},
}

\bib{NguyenTudorascu2008}{article}{
    AUTHOR = {Nguyen, T.},
    AUTHOR = {Tudorascu, A.},
     TITLE = {Pressureless {E}uler/{E}uler-{P}oisson systems via adhesion
              dynamics and scalar conservation laws},
   JOURNAL = {SIAM J. Math. Anal.},
    VOLUME = {40},
      YEAR = {2008},
    NUMBER = {2},
     PAGES = {754--775},
}

\bib{Onsager1949}{article}{
    AUTHOR = {Onsager, L.},
     TITLE = {Statistical hydrodynamics},
   JOURNAL = {Nuovo Cimento (9)},
    VOLUME = {6},
      YEAR = {1949},
    NUMBER = {Supplemento, 2 (Convegno Internazionale di Meccanica
              Statistica)},
     PAGES = {279--287},
}

\bib{PoupaudRascle1997}{article}{
    AUTHOR = {Poupaud, F.},
    AUTHOR = {Rascle, M.},
     TITLE = {Measure solutions to the linear multi-dimensional transport
              equation with non-smooth coefficients},
   JOURNAL = {Comm. Partial Differential Equations},
    VOLUME = {22},
      YEAR = {1997},
    NUMBER = {1-2},
     PAGES = {337--358},
}

\bib{Rindler2011}{article}{
    AUTHOR = {Rindler, F.},
     TITLE = {Lower semicontinuity for integral functionals in the space of
              functions of bounded deformation via rigidity and {Y}oung
              measures},
   JOURNAL = {Arch. Ration. Mech. Anal.},
    VOLUME = {202},
      YEAR = {2011},
    NUMBER = {1},
     PAGES = {63--113},
}


\bib{Rockafellar1997}{book}{
    AUTHOR = {Rockafellar, R. T.},
     TITLE = {Convex analysis},
    SERIES = {Princeton Landmarks in Mathematics},
      NOTE = {Reprint of the 1970 original,
              Princeton Paperbacks},
 PUBLISHER = {Princeton University Press},
   ADDRESS = {Princeton, NJ},
      YEAR = {1997},
}

\bib{Sever2001}{article}{
    AUTHOR = {Sever, M.},
     TITLE = {An existence theorem in the large for zero-pressure gas
              dynamics},
   JOURNAL = {Differential Integral Equations},
    VOLUME = {14},
      YEAR = {2001},
    NUMBER = {9},
     PAGES = {1077--1092},
}

\bib{Tadmor1986}{article}{
    AUTHOR = {Tadmor, E.},
     TITLE = {A minimum entropy principle in the gas dynamics equations},
   JOURNAL = {Appl. Numer. Math.},
    VOLUME = {2},
      YEAR = {1986},
    NUMBER = {3-5},
     PAGES = {211--219},
}

\bib{Westdickenberg2010}{article}{
    AUTHOR = {Westdickenberg, M.},
     TITLE = {Projections onto the cone of optimal transport maps and 
              compressible fluid flows},
   JOURNAL = {J. Hyperbolic Differ. Equ.},
    VOLUME = {7},
      YEAR = {2010},
     PAGES = {605--649},
}


\bib{Wolansky2007}{article}{
    AUTHOR = {Wolansky, G.},
     TITLE = {Dynamics of a system of sticking particles of finite size on
              the line},
   JOURNAL = {Nonlinearity},
    VOLUME = {20},
      YEAR = {2007},
    NUMBER = {9},
     PAGES = {2175--2189},
}

\bib{Zarantonello1971}{article}{
    AUTHOR = {Zarantonello, E. H.},
     TITLE = {Projections on convex sets in Hilbert space and
              spectral theory. I. Projections on convex sets},
CONFERENCE = {
     TITLE = {Contributions to nonlinear functional analysis (Proc.
              Sympos., Math. Res. Center, Madison, WI, 1971)},
             },
      BOOK = {
 PUBLISHER = {Academic Press},
     PLACE = {New York},
             },
      DATE = {1971},
     PAGES = {237--341},
}

\bib{Zeldovich1970}{article}{
    AUTHOR = {Zel'dovich, Ya. B.},
     TITLE = {Gravitational instability: An approximate theory for 
              large density perturbations},
   JOURNAL = {Astro. Astrophys.},
    VOLUME = {5},
      YEAR = {1970},
     PAGES = {84--89},
}
  
\end{biblist}
\end{bibdiv}
\end{document}